\documentclass[a4paper,12pt]{article}
\usepackage{cite}
\usepackage{amsmath,dsfont}
\usepackage{amssymb}
\usepackage{amsthm}
\usepackage{graphicx}
\usepackage{caption}
\usepackage{subcaption}
\usepackage{float}
\usepackage[linkcolor=blue, urlcolor=blue, citecolor=blue,
 colorlinks, bookmarks]{hyperref}
\usepackage{amsfonts}
\usepackage[top=2cm,bottom=2cm,left=2cm,right=2cm]{geometry}
\usepackage{tabularx}
\everymath{\displaystyle}
\newtheorem{Theorem}{Theorem}
\numberwithin{equation}{section} 
\numberwithin{Theorem}{section}
\newtheorem{Definition}{Definition}
\numberwithin{Definition}{section}

\numberwithin{Algorithm}{section}

\numberwithin{Lemma}{section} \numberwithin{Proposition}{section}
\newtheorem{Example}{Example}
\newtheorem{Remark}{Remark}
\numberwithin{Remark}{section}
\numberwithin{Example}{section}

\numberwithin{Corollary}{section}

\usepackage{fancyhdr}
\usepackage[ruled,vlined]{algorithm2e}

\include{pythonlisting}

\SetCommentSty{mycommfont}

\begin{document}
\date{}
\title{\bf Computational Wavelet Method for Multidimensional Integro-Partial Differential Equation of Distributed Order}
\author{
Yashveer Kumar$^1$ \footnote{Author's e-mail: yashveerkumar.iitbhu@gmail.com},
Somveer Singh$^1$ \footnote{Author's e-mail: rathaurbhu.90@gmail.com},
Reshma Singh$^2$ \footnote{Author's e-mail: reshma.gzp@gmail.com},
Vineet Kumar Singh$^{1}$ \footnote{Corresponding author, e-mail: vksingh.mat@iitbhu.ac.in}
\\
$^1$\textit{\small{Department of Mathematical Sciences, Indian Institute of Technology}} \\
\textit{\small{(Banaras Hindu University), Varanasi, India.}} \\
$^2$\textit{\small{Department of Mathematics, Ram Briksh Benipuri Mahila College, Muzaffarpur}} \\
\textit{\small{Babasaheb Bhimrao Ambedkar Bihar University, Muzaffarpur India.}} \\
}
\maketitle
\begin{abstract}
This article provides an effective computational algorithm based on Legendre wavelet (LW) and standard tau approach to approximate the solution of multi-dimensional distributed order time-space fractional weakly singular integro-partial differential equation (DOT-SFWSIPDE). To the best of our understanding, the proposed computational algorithm is new and has not been previously applied for solving DOT-SFWSIPDE. The matrix representation of distributed order fractional derivatives, integer order derivatives and weakly singular kernel associated with the integral based on LW are established to find the numerical solutions of the proposed DOT-SFWSIPDE. Moreover, the association of standard tau rule and Legendre-Gauss quadrature (LGQ) techniques along with constructed matrix representation of  differential and integral operators diminish DOT-SFWSIPDE into system of linear algebraic equations. Error bounds, convergence analysis, numerical algorithms and also error estimation of the DOT-SFWSIPDE are regorously investigated. For the reliability of the proposed computational algorithm, numerous test examples has been incorporated in the manuscript to ensure the robustness and theoretical results of proposed technique.   


\medskip \noindent{\bf Keywords:} Multi-dimensional distributed order time-space fractional weakly singular integro-differential equation, Fractional order Caputo derivative, Legendre wavelets, Operational matrices, Convergence analysis, Error estimation.
\end{abstract}

\section{ Introduction}
 Fractional differential and integral models have sparked a lot of interest, because of their applications in many fields of science, finance as well as in engineering \cite{podlubny1998fractional}. There are some fascinating implementations of fractional calculus in viscoelasticity model \cite{bagley1985fractional}, electromagnetic waves \cite{ichise1971analog}, chaotic systems \cite{hajipour2018efficient}, physical systems \cite{baleanu2017motion}, optimization \cite{jajarmi2018new}, nonlinear dynamical systems \cite{baleanu2018nonlinear}, in the modeling of heat transfer \cite{sierociuk2013modelling}, and dynamics of interfaces between nano particles and substrates \cite{chow2005fractional}. Furthermore, the use of fractional calculus in viscoelasticity has emerged as a promising area of research. For example, fractional derivatives without singular kernels, have been proposed as mathematical methods for describing viscoelasticity models. A new fractional-order algorithm to explain the dynamic behaviours of general fractional-order viscoelasticity with memory effect, Maxwell and Voigt models are suggested and used within the framework of general fractional derivatives \cite{yang2017new}. A new model is presented in \cite{gao2016fractional} to demonstrate the efficiancy of fractional-order operators in the case of line viscoelasticity.

Fractional differential equations (FDEs) and fractional integral equations have attracted the interest of many researchers due to their practical applications in various fields of science and engineering. Although some techniques exist for obtaining analytical solutions to some FDEs, analytical solutions to FDEs remain unknown in the overwhelming majority of cases.  As a result, several scholars have devised numerous computational methods for obtaining approximate solutions to fractional order integral and fractional order differential equations. The most commonly used methods are variational iteration method \cite{odibat2010study}, generalized  transform
method \cite{momani2007generalized,odibat2008generalized}, adomian decomposition method \cite{garg2011solution,ray2005approximate}  and wavelet method \cite{chen2012error,babolian2007numerical} . A. Saadatmandi \& M. Dehghan debated on the solution of space-fractional diffusion equation with Caputo derivative by the tau approach \cite{saadatmandi2011tau}. In \cite{srivastava2021efficient}, authors discussed semi-discrete scheme for Riesz- FDE. In \cite{dehghan2019error}, a finite element/finite difference scheme has been proposed to solve the 2D time and space fractional  partial integro differential equation with  weakly singular kernel. {For more study about the methods to solve FDE readers one can see \cite{abbaszadeh2019meshless,abbaszadeh2019alternating,dehghan2010solving,sun2016some}}. 

 Now a days, distributed order operator is an attractive tool to explain the physical phenomena of mathematical models governed by the fields of science, finance and engineering. The distributed order fractional (DOF) derivative has a long and illustrious history. In 1969, Caputo was the first to introduce the concept of DOF operator, and in 1995, he was also the first to solve it.  The distributed order fractional differential equation (DOFDE) is stated in its generic form as \cite{jiao2012distributed}
\begin{equation*}
\sum_{i=1}^k a_i\int_{\alpha_1}^{\alpha_2}\rho_i(\alpha)D_{\varrho}^{i-\alpha}\mathbb{U}(\varkappa,\varrho)d\alpha+\sum_{j=0}^kb_j\mathbb{U}^{(j)}(\varkappa,\varrho)=f(\varkappa,\varrho),
\end{equation*}

where, $\rho_i(\alpha)$ denotes the  weight function of distribution of order $\alpha\in[\alpha_1,\alpha_2]$ and $k\in\mathbb{Z^+}$. Hence, the above equation can be viewed as a generalisation form of
\begin{itemize}
	\item If $\rho_i(\alpha)\equiv 0$, then we get differential equation of integer order.
	\item If $\rho_i(\alpha)$ takes any discrete values in $[\alpha_1,\alpha_2]$, then we get FDE's. 
\end{itemize}
As a result, differential equation with integer and non-integer order can be considered as special cases of distributed order fractional  differential equations (DOFDE).  In the fields of engineering, science and financial mathematics, distributed-order differential equations have a wide range of applications. For instance, they are used in the modelling of  dielectric induction and diffusion \cite{Caputo2001}. In 2004, Sokolov and Chechkin \cite{sokolov2004distributed} debated on the DOF kinetics.  Umarov et al.  provided  randam walk models \cite{umarov2006random} are governed with the help of  DOFDE. The financial mathematical model goverened with the help of  DOF derivative defined in \cite{morgado2019black} considered time DOF Black-Scholes equation. With the development of DOFDE, various numerical methods were constructed for their solutions. In \cite{kumar2020wavelet} the authors presented a numerical wavelet scheme for DOFDE. In \cite{abbaszadeh2019error}, Riesz-space DOFDE using second-order finite difference scheme has been proposed. For solving the time DOF advection-diffusion equation, the authors of \cite{gao2017temporal} discovered a special point for the linear combination of multi-term fractional derivatives interpolation approximation and obtained a numerical differentiation formula with second-order precision. In \cite{abbaszadeh2020crank} authors applied Crank-Nicolson/Galerkin spectral method for solving two-dimensional time-space DOF integro-partial differential equation with weakly singular by using Riesz derivative in the space direction, whereas in \cite{saadatmandi2011tau} authors used the tau approach for solving space fractional diffusion equation by using Caputo derivative in space. In this work, we consider the following DOT-SFWSIPDE \cite{abbaszadeh2020crank, christensen2012mechanics, miller1978integrodifferential, renardy1989mathematical} in 1D \& 2D, using Caputo derivative in both directions define as,\\
\begin{align}\label{Eqn:01.1}
\int_{\alpha_1}^{\alpha_2}\rho(\alpha)\frac{\partial^{\alpha}\mathbb{U}(\varkappa,\varrho)}{\partial \varrho^{\alpha}}d\alpha+\mathbb{U}(\varkappa,\varrho)=\mathcal{K^*}\int_{\beta_1}^{\beta_2}\rho(\beta)\frac{\partial^{\beta}\mathbb{U}(\varkappa,\varrho)}{\partial \varkappa^{\beta}}d\beta&+\int_0^\varrho(\varrho-\xi)^{-\frac{1}{2}}\left[\frac{\partial^2\mathbb{U}(\varkappa,\xi)}{\partial \varkappa^2}\right]d\xi\nonumber\\
&+f(\varkappa,\varrho),
\end{align}
where, $\mathcal{K^*}$ is viscosity constant and $(\varkappa,\varrho)\in\Omega,~~ \alpha_1=0, \alpha_2=1, \beta_1=1, \beta_2=2$ and $\Omega=[0,1]\times[0,T].$

The above  equation \ref{Eqn:01.1} is  with the initial condition
\begin{eqnarray}\label{Eqn:01.2}
\mathbb{U}(\varkappa,0)=\nu(\varkappa) ,~~  0\textless \varkappa \textless 1,
\end{eqnarray}
and Dirichlet boundary conditions 
\begin{eqnarray}\label{Eqn:01.3}
\mathbb{U}(0,\varrho)=\mathfrak{p_1}(\varrho),~~~0\textless \varrho\textless T,
\end{eqnarray}
\begin{eqnarray}\label{Eqn:01.4}
\mathbb{U}(0,\varrho)=\mathfrak{p_2}(\varrho), ~~0\textless \varrho\textless T.
\end{eqnarray}

The 2D form of the above problem is defined as
\begin{eqnarray}\label{Eqn:1.1}
\int_{\alpha_1}^{\alpha_2}\rho(\alpha)\frac{\partial^{\alpha}\mathbb{U}(\varkappa,\eta,\varrho)}{\partial \varrho^{\alpha}}d\alpha+\mathbb{U}(\varkappa,\eta,\varrho)=\mathcal{K^*}\int_{\beta_1}^{\beta_2}\rho(\beta)\left[\frac{\partial^{\beta}\mathbb{U}(\varkappa,\eta,\varrho)}{\partial \varkappa^{\beta}}+\frac{\partial^{\beta}\mathbb{U}(\varkappa,\eta,\varrho)}{\partial \eta^{\beta}}\right]d\beta\nonumber\\
+\int_0^\varrho(\varrho-\xi)^{-\frac{1}{2}}\left[\frac{\partial^2\mathbb{U}(\varkappa,\eta,\xi)}{\partial \varkappa^2}+\frac{\partial^2\mathbb{U}(\varkappa,\eta,\xi)}{\partial \eta^2}\right]d\xi+f(\varkappa,\eta,\varrho),
\end{eqnarray}
where, $\mathcal{K^*}$ is viscosity constant and $(\varkappa,\eta,\varrho)\in \Omega,\alpha_1=0, \alpha_2=1, \beta_1=1, \beta_2=2$ and $\Omega=[0,1]\times[0,1]\times[0,T]$.

The initial condition for the above  equation \ref{Eqn:1.1} is 
\begin{eqnarray}\label{Eqn:1.2}
\mathbb{U}(\varkappa,\eta,0)=\nu(\varkappa,\eta) ,~~  0\textless \varkappa \textless 1~and~0\textless\eta\textless 1
\end{eqnarray}
and the boundary conditions are 
\begin{eqnarray}\label{Eqn:1.3}
\mathbb{U}(0,\eta,\varrho)=p_1(\eta,\varrho),~~~0\textless\eta\textless 1 ~and~0\textless \varrho\textless T,
\end{eqnarray}	
\begin{eqnarray}\label{Eqn:1.4}
\mathbb{U}(1,\eta,\varrho)=p_2(\eta,\varrho), ~~0\textless\eta\textless 1~and~0\textless \varrho\textless T.
\end{eqnarray}\label{Eqn:1.5}
\begin{eqnarray}
\mathbb{U}(\varkappa,0,\varrho)=q_1(\varkappa,\varrho),~~~0\textless\varkappa\textless 1 ~and~0\textless \varrho\textless T,
\end{eqnarray}	
\begin{eqnarray}\label{Eqn:1.6}
\mathbb{U}(\varkappa,1,\varrho)=q_2(\eta,\varrho), ~~0\textless\varkappa\textless 1~and~0\textless \varrho\textless T.
\end{eqnarray}
Here, $\rho(\alpha),~\rho(\beta)$ are the weight functions that satisfy the following criteria\cite{gorenflo2013fundamental}
 \begin{equation*}
 \rho(\alpha)\geq 0, ~~\int_{\alpha_1}^{\alpha_2}\rho(\alpha)d\alpha=\lambda_1\textgreater 0~and~\rho(\beta)\geq 0, ~~\int_{\beta_1}^{\beta_2}\rho(\beta)d\beta=\lambda_2\textgreater 0.
 \end{equation*}
 The existence and uniqueness of the solution of DOFDE can be seen in \cite{li2017analyticity,morgado2017numerical}. The application of considered problem \ref{Eqn:01.1}-\ref{Eqn:01.4} \& \ref{Eqn:1.1}-\ref{Eqn:1.6} can be found in the  modeling of physical phenomena involving viscoelastic model. Based on the current literature and referring to the data, there is no computational approach available for solving DOT-SFIPDE centred on LW. The proposed technique based on LW operational matrices is employed in this article, to solve this newly created model \ref{Eqn:01.1}-\ref{Eqn:01.4} \& \ref{Eqn:1.1}-\ref{Eqn:1.6} in the sense of fractional Caputo derivative. 
 
 The operational matrices have been proved to be an effective tool for solving FDEs. Saadatmandi and Dehghan developed the operational matrix of fractional derivative (OMFD) for shifted Legendre polynomials (SLPs) in 2010 \cite{saadatmandi2010new}. In  \cite{li2010haar}, Zhao et al. constructed the OMFD using Haar wavelet. In \cite{keshavarz2014bernoulli} authors developed the  Bernoulli wavelets OMFD. In \cite{bhrawy2015review}, Taha et al. invented the Laguerre polynomials OMFD. Pourbabaee and Saadatmandi \cite{pourbabaee2019novel} recently devised a helpful technique  based on Legendre polynomials OMFD for finding the approximated solution of DOFDE. Readers can see \cite{singh2017numerical,singh2018application,singh2018convergence}, to learn more about operational matrix approaches .
 
  Wavelets are a special type of orthogonal functions that have become very useful and effective tools in computational science. Wavelet methods have recently received increased recognition for numerically solving integral and differential equations; they were first utilized to discuss the solution of differential equations in the early 1990 \cite{chen1997haar}. Many papers have been recently published that use LWs to include numerical solutions of  fractional differential and integro-differential equations (IDEs). The LW operational matrix approach is used to solve the nonlinear Volterra IDEs \cite{sahu2015legendre}. To solve the Dirichlet boundary value problem for fractional partial differential equation, LWs were used \cite{heydari2014legendre}. In \cite{meng2015legendre} Meng et.
  al. used LW to evaluate the solution of linear and nonlinear fractional IDEs.
  
It is observed that  majority of papers that use the LWs approach to obtain numerical solution of FDEs use a LWs operational matrix. As a result, we use the LWs operational matrix approach to solve linear time-space DOF integro-differential equations with weakly singular kernels in this article.
The operational matrix approach is also proven to be an effective and resilient numerical methodology for solving DOFDE, as shown in  \cite{pourbabaee2019novel}. The goal of this article is to construct the DOF derivative matrix based on LWs. The motivation for using a wavelets-based approach is convenient. There are two approaches to improve the accuracy of the solution in such methods: raising the  level of resolution of wavelets family and increasing the number of wavelet basis functions. Furthermore, because LWs are made up of orthogonal polynomials, they have indefinitely differentiable functions \& small compact support. Moreover, the LW  operational matrices are sparse, reducing calculation time. The Legendre-Gauss quadrature (LGQ) rule and the tau technique are used to solve the DOT-SFWSIPDE using such matrices. Know more information about wavelet and DOFDE readers can be see \cite{mehra2018wavelets,kumar2021computational,behera2018adaptive,gao2017temporal,alikhanov2015numerical}. The goal of this technique is to have successful experimental tests that are less computationally expensive. 

The remainder of the paper is structured as follows: Fractional derivatives, the distributed differential operator, LWs, and their approximation characteristics are all briefly defined in Section \ref{sec:2}. This section also includes the Gauss-Legendre quadrature integration formula. The operational matrices are covered in the reference section \ref{sec:3}.  Operational matrices of derivative are produced in this section for both integer and distributed order LWs. In section \ref{sec:4}, the suggested technique is implemented together with a numerical algorithm to solve DOFIDEs. Section \ref{sec:5}  discusses the error bound and convergence analysis for the described scheme. The suggested method's error estimation is described in section \ref{sec:6}. Finally, in section \ref{sec:7}, numerical tests of a viscoelastic model regulated by DOT-SFWSIPDE are carried out.
\\

\section{Some basic definitions}\label{sec:2}
\begin{Definition}(\textbf{Fractional Caputo derivative}):
	 The fractional derivative of order $\alpha>0$ in Caputo sense  is devoted as  \cite{podlubny1998fractional}
	\begin{equation}\label{Eqn:2.1}
	D_\varrho^{\alpha}\mathbb{U}(\varkappa,\varrho)=
	\left\{ \begin{array}{l}\frac{1}{\Gamma(n-\alpha)}\int_0^{\varrho}\frac{1}{(\varrho-\tau)^{\alpha+1-n}}\frac{\partial^n \mathbb{U}(\varkappa,\tau)}{\partial \tau^n}d\tau,~~~  n-1<\alpha<n
	\vspace{0.2cm}\\ \frac{\partial^n \mathbb{U}(\varkappa,\varrho)}{\partial \varrho^n},~~~~~~~~~~~~~~~~~~~~~~~~~~~~~~ \alpha=n\in\mathbb{N}.
	\end{array}\right.
	\end{equation}
	
	Here, $\Gamma(.)$ represents the Gamma function. This Caputo operator have some basic properties:
\end{Definition}

\begin{itemize}
	\item $D_\varrho^{\alpha} \mathcal{J}=0$, here, $\mathcal{J}$ denotes arbitrary constant.
\item The Caputo derivative of $u(\varrho)=\varrho^n, n\in\mathbb{Z}^+$ is given as:
\begin{equation*}
 D_\varrho^{\alpha}\varrho^n=
\left\{ \begin{array}{l} 0,~~~~~~~~~~~~~~~~~~~~~~ n<\lceil\alpha\rceil,
\vspace{0.2cm}\\ \frac{\Gamma(n+1)}{\Gamma(n+1-\alpha)}\varrho^{n-\alpha},~~ n\geq\lceil\alpha\rceil,
\end{array}\right.
\end{equation*}
\end{itemize}

 where,  $\lceil\cdot\rceil$ represent  the ceiling function.
\begin{itemize}
	\item $\ D_\varrho^{\alpha}$ fulfills the linearity property, i.e.
\begin{equation*}
D_\varrho^{\alpha}\left(\sum_{j=1}^sb_ju_j(\varrho)\right)=\sum_{j=1}^sb_jD_\varrho^{\alpha}u_j(\varrho).
\end{equation*}
Here, $b_j$ denotes an arbitrary constants for $j=1,2,\cdots,s, ~s\in\mathcal{Z}^+$
\end{itemize}
\begin{Definition}(\textbf{Distributed order fractional derivative}):
	The $D_\varrho^{\rho(\alpha)}$ DOF derivative  is defined as\cite{jiao2012distributed}
	\begin{eqnarray}\label{2MEqn2}
	D_\varrho^{\rho(\alpha)}u(\varrho)=\int_{\alpha_1}^{\alpha_2}{\rho(\alpha)}D_\varrho^{\alpha}u(\varrho)d\alpha.
	\end{eqnarray}
	Here, $\rho(\alpha)$ defines the  distribution weight function of order $\alpha$ and $\alpha\in[\alpha_1,\alpha_2]$, where $\alpha_1$ and $\alpha_2$ are non-negative real numbers.
	The DOFD operator has the following properties:
\end{Definition}
\begin{itemize}
\item$D_\varrho^{\rho(\alpha)}\mathcal{J}=0$, where, $\mathcal{J}$ is any arbitrary constant.
\item$D_\varrho^{\rho(\alpha)}$ is a linear operator, i.e,
\begin{eqnarray}
D_\varrho^{\rho(\alpha)}\left(\sum_{j=1}^sb_jD_\varrho^{\rho(\alpha)}u_j(\varrho)\right)=\sum_{j=1}^sb_jD_{\varrho}^{\rho(\alpha)}u_j(\varrho),
\end{eqnarray}
where, ${b_j}$ are arbitrary constants for $j=1,2,\cdots,s, `s\in\mathcal{Z}^+$ 
\end{itemize}
\begin{itemize}
\item  If $\rho(\alpha)=\delta(\alpha-\mu)$, where, $\alpha_1<\mu<\alpha_2$ and  $\delta$ is delta Dirac function. Then we have 
\begin{eqnarray}
D_\varrho^{\rho(\alpha)}u(\varrho)=\int_{\alpha_1}^{\alpha_2}\textbf{$\delta$} (\alpha-\mu)D_\varrho^{\alpha}u(\varrho)d\alpha=D_\varrho^{\mu}u(\varrho).
\end{eqnarray}
In other words, we obtain a fractional derivative of order $\mu.$
\end{itemize}
\begin{Definition}(\textbf{Legendre wavelets}):
	Legendre wavelets $\Psi_{\mathfrak{h},\mathfrak{g}}(\varrho)=\Psi(\mathfrak{R},\hat{\mathfrak{h}},\mathfrak{g},\varrho)$ have four arguments: $\hat{\mathfrak{h}}=2\mathfrak{h}-1, \mathfrak{h}=1,2,3,\cdots,2^{\mathfrak{R}-1}$, $\mathfrak{R}\in Z^+$, $\mathfrak{g}$ is degree of Legendre polynomials and $\varrho$ denotes the time of normalization. Then the one dimension LWs definition  over [0,1] are described as \cite{singh2018application}:
	
	\begin{eqnarray}\label{2MEqn5}
	\Psi_{\mathfrak{h},\mathfrak{g}}(\varrho)= \left\{ \begin{array}{l}\sqrt{\displaystyle(\mathfrak{g}+\frac{1}{2})}2^{\frac{\mathfrak{R}}{2}}p_{\mathfrak{g}}(2^\mathfrak{R}\varrho-2\mathfrak{h}+1),~~  \frac{\mathfrak{h}-1}{2^{\mathfrak{R}-1}}\leq \varrho \leq\frac{\mathfrak{h}}{2^{\mathfrak{R}-1}},
	\vspace{0.2cm}\\0,~~~~~elsewhere.
	\end{array}\right.
	\end{eqnarray} 
	
	Where, $\mathfrak{g}=0,1,2,\cdots,\Lambda-1, \mathfrak{h}=1,2,3,\cdots,2^{\mathfrak{R}-1}$. 
\end{Definition}
 
  \begin{Remark}\label{sec:2.4}
  Two-dimensional LWs are represented as follows:\cite{singh2018application}:
  \begin{eqnarray}
  \Psi_{\mathfrak{h},\mathfrak{g},\mathfrak{h}',\mathfrak{g}'}(\varkappa,\varrho)=\left\{ \begin{array}{l}\Psi_{\mathfrak{h},\mathfrak{g}}(\varkappa)\Psi_{\mathfrak{h}',\mathfrak{g}'}(\varrho),~~  \frac{\mathfrak{h}-1}{2^{\mathfrak{R}-1}}\leq \varkappa\leq\frac{\mathfrak{h}}{2^{\mathfrak{R}-1}},~\frac{\mathfrak{h}'-1}{2^{\mathfrak{R}'-1}}\leq \varrho\leq\frac{\mathfrak{h}'}{2^{\mathfrak{R}'-1}},
  \vspace{0.2cm}\\0, ~~elsewhere.
  \end{array}\right.
  \end{eqnarray} 
  \end{Remark}

  \begin{Definition}(\textbf{Function approximation}):
%
  	  	
  	The function $f(\varkappa,\varrho)$ defined over $L^2(\Omega=[0,1]\times[0,T])$ can be written as the sum of LW infinite series such as
  	\begin{eqnarray}
  	f(\varkappa,\varrho)=\sum_{\mathfrak{h}=1}^{\infty}\sum_{\mathfrak{g}=0}^{\infty}\sum_{\mathfrak{h}'=1}^{\infty}\sum_{\mathfrak{g}'=0}^{\infty} f_{\mathfrak{h}\mathfrak{g} \mathfrak{h}'\mathfrak{g}'}\varphi_{\mathfrak{h}\mathfrak{g} \mathfrak{h}'\mathfrak{g}'}(\varkappa,\varrho).
  	\end{eqnarray}
  	The truncation of the above series leads to 
  	\begin{eqnarray}
  	f(\varkappa,\varrho)\approx\sum_{\mathfrak{h}=1}^{2^{\mathfrak{R}-1}}\sum_{\mathfrak{g}=0}^{{\Lambda-1}}\sum_{\mathfrak{h}'=1}^{2^{\mathfrak{R}'-1}}\sum_{\mathfrak{g}'=0}^{\Lambda'-1} f_{\mathfrak{h}\mathfrak{g} \mathfrak{h}'\mathfrak{g}'}\varphi_{\mathfrak{h}\mathfrak{g} \mathfrak{h}'\mathfrak{g}'}(\varkappa,\varrho)=\Psi^T(\varrho)\mathcal{F}\Psi(\varkappa),
  	\end{eqnarray}
  	where, $f_{\mathfrak{h}\mathfrak{g} \mathfrak{h}'\mathfrak{g}'}=\left<\left<f,\Psi(\varrho)\right>,\Psi(\varkappa)\right>$ and $<.,.>$ represents the inner product and $\mathcal{F}$  is $2^{\mathfrak{R}-1}\Lambda\times2^{\mathfrak{R}'-1}\Lambda'$ vector and $\Psi(\varrho)$,$\Psi(\varkappa)$ are $2^{\mathfrak{R}-1}\Lambda\times1,2^{\mathfrak{R}'-1}\Lambda'\times1$ vectors, respectively.
  	\begin{eqnarray}
  	\Psi(\varrho)=\begin{bmatrix}
  	\varphi_{1,0},\varphi_{1,1},\cdots, \varphi_{1,\Lambda-1},\varphi_{2,0},\cdots, \varphi_{2,\Lambda-1},\cdots, \varphi_{2^{\mathfrak{R}-1},0},\cdots, \varphi_{2^{\mathfrak{R}-1},\Lambda-1}
  	\end{bmatrix}^T.
  	\end{eqnarray}
 
  	Similarly, 
  	\begin{equation}
  	f(\varkappa,\eta,\varrho)=\sum_{\mathfrak{h}=1}^{\infty}\sum_{\mathfrak{g}=0}^{\infty}\sum_{\mathfrak{h}'=1}^{\infty}\sum_{\mathfrak{g}'=0}^{\infty}\sum_{\mathfrak{h}''=0}^\infty\sum_{\mathfrak{g}''=0}^\infty f_{\mathfrak{h}\mathfrak{g} \mathfrak{h}'\mathfrak{g}' \mathfrak{h}''\mathfrak{g}''}\varphi_{\mathfrak{h}\mathfrak{g} \mathfrak{h}'\mathfrak{g}' \mathfrak{h}''\mathfrak{g}''}(\varkappa,\eta,\varrho).
  	\end{equation}
  	
  	The truncation of the above series leads to 
  	\begin{eqnarray}
  	f(\varkappa,\eta,\varrho)\approx\sum_{\mathfrak{h}=1}^{2^{\mathfrak{R}-1}}\sum_{\mathfrak{g}=0}^{{\Lambda-1}}\sum_{\mathfrak{h}'=1}^{2^{\mathfrak{R}'-1}}\sum_{\mathfrak{g}'=0}^{\Lambda'-1}\sum_{\mathfrak{h}''=1}^{2^{\mathfrak{R}''-1}}\sum_{\mathfrak{g}''=0}^{{\Lambda''-1}} f_{\mathfrak{h}\mathfrak{g} \mathfrak{h}'\mathfrak{g}'\mathfrak{h}''\mathfrak{g}''}\varphi_{\mathfrak{h}\mathfrak{g} \mathfrak{h}'\mathfrak{g}'\mathfrak{h}''\mathfrak{g}''}(\varkappa,\eta,\varrho)=\Psi^T(\varrho)\mathcal{F}\Psi(\varkappa,\eta),
  	\end{eqnarray}
  	where, the function $f(\varkappa,\eta,\varrho)$ defined over $L^2(\Omega=[0,1]\times[0,1]\times[0,T])$,  $\Psi(\varkappa,\eta)=\Psi(\varkappa)\otimes \Psi(\eta)$, $\otimes$ denotes the Kronecker product and $f_{\mathfrak{h}\mathfrak{g} \mathfrak{h}'\mathfrak{g}'\mathfrak{h}''\mathfrak{g}''}$ =$\left<\left<\left<f,\Psi(\varrho)\right>, \Psi(\varkappa)\right>,\Psi(\eta)\right>$. $\mathcal{F}$ is $2^{\mathfrak{R}-1}\Lambda\times (2^{\mathfrak{R}'-1}\Lambda')(2^{\mathfrak{R}''-1}\Lambda'')$ vector and $\Psi(\varrho),\Psi(\varkappa)$ and $\Psi(\eta)$ are $2^{\mathfrak{R}-1}\Lambda\times 1,2^{\mathfrak{R}'-1}\Lambda'\times1, 2^{\mathfrak{R}''-1}\Lambda''\times 1$ vectors, respectively.
  	
  \end{Definition}

    \begin{Definition}(\textbf{Legendre-Gauss quadrature (LGQ) formula for Numerical integration}):
    	 Let ${\{\tau_s}\}_{s=1}^P$ denotes the collection of $P$ distinct roots of Legendre polynomial of degree $P$, where, $P\in \mathbb{Z} ^+$. The $P$-point LGQ formula approximates the function  integral over the interval $(\alpha_1,\alpha_2)$ as\cite{hildebrand1987introduction}
    	\begin{eqnarray}\label{2MEqn16}
    	\int_{\alpha_1}^{\alpha_2}u(\varrho)d\varrho\approx\sum_{q=1}^P w_qu(\sigma_q),
    	\end{eqnarray}
    	where,
    	\begin{eqnarray}
    	w_s=\frac{\alpha_2-\alpha_1}{(1-\tau_s^2)(L'_P(\tau_s))^2},~  \sigma_s=\frac{\alpha_2-\alpha_1}{2}\tau_s+\frac{\alpha_2+\alpha_1}{2}, ~ s=1,2,\cdots,P.
    	\end{eqnarray}
    	Here, ${\{w_s}\}_{s=1}^P$ and ${\{\sigma_s}\}_{s=1}^P$ are LGQ weights and nodes, respectively. The LGQ formula is correct upto for all polynomials, of degree atmost $2P-1$.
    \end{Definition}
 
  \section{ Construction of operational matrices }\label{sec:3}
  \subsection{Derivative operational matrix for integer order }\label{sec:3.1}
  \begin{Theorem}
  	The derivative of m-degree shifted Legendre polynomial $p_m(\varkappa)$ defined over [0, 1] is given as:
  	\begin{equation*}
  	p_m'(\varkappa)=2\sum_{k=0, k+m~  odd}^{m-1}(2k+1)p_k(\varkappa)
  	\end{equation*}
  	\begin{proof}
  		Given in reference \cite{mohammadi2011new}. 
  	\end{proof}
  \end{Theorem}
    \begin{Theorem}
  	Suppose $\Psi(\varkappa)$ denotes the LW vector. The derivative of $\Psi(\varkappa)$ can be determined as: 
  	\begin{eqnarray}
  	\frac{d\Psi(\varkappa)}{d\varrho}=D\Psi(\varkappa),
  	\end{eqnarray}  
  	here, $D$ denotes derivative operational matrix for LW   of order  $2^{\mathfrak{R}-1}(\mathfrak{g}+1)$, described as follows:
  	\begin{eqnarray*}
  	D=
  	\begin{bmatrix}
  	H & 0 & 0 & \cdots & 0\\
  	0 & H & 0 & \cdots & 0\\
  	\vdots & \vdots & \vdots & \ddots & \vdots\\
  	0 & 0 & \cdots & 0 & H
  	\end{bmatrix}.
  	\end{eqnarray*}
  	$H$ represents matrix of order $(\mathfrak{g}+1)\times(\mathfrak{g}+1)$ and its $(\mathfrak{p},\mathfrak{q})$th component is described as follows:
  \begin{eqnarray}
  H_{\mathfrak{p},\mathfrak{q}}= \left\{ \begin{array}{l} 2^{\mathfrak{R}}\sqrt{\displaystyle(2\mathfrak{p}-1)(2\mathfrak{q}-1)}~~\mathfrak{p}=2,3,\cdots,(\mathfrak{g}+1),~ \mathfrak{q}=1,2, \cdots, \mathfrak{p}-1 ~ and~(\mathfrak{p}+\mathfrak{q})~odd, 
  \vspace{0.2cm}\\0,~~otherwise.
  \end{array}\right.
  \end{eqnarray} 
  \end{Theorem}
  In general,
  \begin{eqnarray}
  \frac{d^{n_1}\Psi(\varkappa)}{d\varrho^{n_1}}=(D)^{n_1}\Psi(\varkappa),~~n_1=1,2,3, \cdots
  \end{eqnarray}
  \begin{proof}
  Given in reference\cite{mohammadi2011new}
  \end{proof}
  	\subsection{Construction of DOF matrix}\label{sec:3.3}
  Let $D_\varrho^{\rho(\alpha)}$ be the DOF derivative w.r.t time component defined in equation \ref{2MEqn2}. Then we have
  	\begin{eqnarray}
  	D_\varrho^{\rho(\alpha)}\varrho^n=\int_{\alpha_1}^{\alpha_2}\rho(\alpha) D_\varrho^{\alpha}\varrho^n d\alpha
  	=\int_{\alpha_1}^{\alpha_2}\rho(\alpha)\frac{\Gamma(n+1)}{\Gamma(n+1-\alpha)}\varrho^{n-\alpha}d\alpha , ~~ n\in\mathbb{N} ,k\geq\lceil\alpha_2\rceil.
  	\end{eqnarray}
   The LGQ rule approximates the above equation as follows:
   \begin{equation}
   	D_\varrho^{\rho(\alpha)}\varrho^n\approx\sum_{s=1}^Pw_s\frac{\rho(\sigma_s)n!}{\Gamma(n+1-\sigma_s)}.
   	\end{equation}
   	For the general class
   	\begin{align}
   	D_\varrho^{\rho(\alpha)}\Psi(\varrho)
   	=\int_{\alpha_1}^{\alpha_2}\rho(\alpha)D_\varrho^{(\alpha)}\Psi(\varrho)d\alpha
        &=
   		\int_{\alpha_1}^{\alpha_2}\rho(\alpha)
   		\begin{bmatrix}
   		D_\varrho^{(\alpha)}\varphi_{10}(\varrho)\\
   		D_\varrho^{(\alpha)}\varphi_{11}(\varrho)\\
   		\vdots\\
   		D_\varrho^{(\alpha)}\varphi_{1\mathfrak{g}}(\varrho)\\
   		D_\varrho^{(\alpha)}\varphi_{20}(\varrho)\\
   		\vdots\\
   		D_\varrho^{(\alpha)}\varphi_{2\mathfrak{g}}(\varrho)\\
   		\vdots\\
   		D_\varrho^{(\alpha)}\varphi_{2^{\mathfrak{R}-1}0}(\varrho)\\
   		\vdots\\
   		D_\varrho^{(\alpha)}\varphi_{2^{\mathfrak{R}-1}\mathfrak{g}}(\varrho)
   		\end{bmatrix}d\gamma\nonumber
   		\end{align}
   		\begin{align}
   		=
   		\int_{\alpha_1}^{\alpha_2}\rho(\alpha)\begin{bmatrix}
   		b_{10}(\varrho,\alpha)\\
   		b_{11}(\varrho,\alpha)\\
   		\vdots\\
   		b_{1\mathfrak{g}}(\varrho,\alpha)\\
   		b_{20}(\varrho,\alpha)\\
   		\vdots\\
   		b_{2\mathfrak{g}}(\varrho,\alpha)\\
   		\vdots\\
   		b_{2^{\mathfrak{R}-1}0}(\varrho,\alpha)\\
   		\vdots\\
   		b_{2^{\mathfrak{R}-1}\mathfrak{g}}(\varrho,\alpha)\\		
   		\end{bmatrix}d\alpha
   		=
   		\begin{bmatrix}
   		\int_{\alpha_1}^{\alpha_2}\mathcal{H}_{10}(\varrho,\alpha)d\alpha\\
   		\int_{\alpha_1}^{\alpha_2}\mathcal{H}_{11}(\varrho,\alpha)d\alpha\\
   		\vdots\\
   		\int_{\alpha_1}^{\alpha_2}\mathcal{H}_{1\mathfrak{g}}(\varrho,\alpha)d\alpha\\
   		\int_{\alpha_1}^{\alpha_2}\mathcal{H}_{20}(\varrho,\alpha)d\alpha\\
   		\vdots\\
   		\int_{\alpha_1}^{\alpha_2}\mathcal{H}_{2\mathfrak{g}}(\varrho,\alpha)d\alpha\\
   		\vdots\\
   		\int_{\alpha_1}^{\alpha_2}\mathcal{H}_{2^{\mathfrak{R}-1}0}(\varrho,\alpha)d\alpha\\
   		\vdots\\
   		\int_{\alpha_1}^{\alpha_2}\mathcal{H}_{2^{\mathfrak{R}-1}\mathfrak{g}}(\varrho,\alpha)d\alpha\nonumber
   		\end{bmatrix}.
   	\end{align}
   	By using the LGQ rule for numerical integration, one can write
   		\begin{align}
   	   		D_\varrho^{\rho(\alpha)}\Psi(\varrho) 
   	   		&\approx
   		\begin{bmatrix}
   		\sum_{s=1}^Pw_s\mathcal{H}_{10}(\varrho,\sigma_q)\\
   		\sum_{s=1}^Pw_s\mathcal{H}_{11}(\varrho,\sigma_q)\\
   		\vdots\\
   		\sum_{s=1}^Pw_s\mathcal{H}_{1\mathfrak{g}}(\varrho,\sigma_q)\\
   		\sum_{s=1}^Pw_s\mathcal{H}_{20}(\varrho,\sigma_q)\\
   		\vdots\\
   		\sum_{s=1}^Pw_s\mathcal{H}_{2\mathfrak{g}}(\varrho,\sigma_q)\\
   		\vdots\\
   		\sum_{s=1}^Pw_s\mathcal{H}_{2^{\mathfrak{R}-1}0}(\varrho,\sigma_q)\\
   		\vdots\\
   		\sum_{s=1}^Pw_s\mathcal{H}_{2^{\mathfrak{R}-1}\mathfrak{g}}(\varrho,\sigma_q)
   		\end{bmatrix}
   		   		=
   		\begin{bmatrix}
   		\mathcal{Q}_{10}(\varrho)\\
   		\mathcal{Q}_{11}(\varrho)\\
   		\vdots\\
   		\mathcal{Q}_{1\mathfrak{g}}(\varrho)\\
   	    \mathcal{Q}_{20}(\varrho)\\
   	     \vdots\\
   	     \mathcal{Q}_{2\mathfrak{g}}(\varrho)\\
   	     \vdots\\
   	     \mathcal{Q}_{2^{\mathfrak{R}-1}0}(\varrho)\\
   	     \vdots\\
   	     \mathcal{Q}_{2^{\mathfrak{R}-1}\mathfrak{g}}(\varrho)
   		\end{bmatrix}\nonumber\\
   		&\approx
   		\hat{D}^{(\alpha_1,\alpha_2,\rho(\alpha))}\Psi(\varrho).
   		\end{align}
   		Thus we obtain
   		\begin{eqnarray}
   		   		D_\varrho^{\rho(\alpha)}\Psi(\varrho)\approx
   		   		\hat{D}^{(\alpha_1,\alpha_2,\rho(\alpha))}\Psi(\varrho),
   		\end{eqnarray}
   		where, the matrix $\hat{D}^{(\alpha_1,\alpha_2,\rho(\alpha))}$ defined as:
   	
   		\begin{equation*}
   		\setcounter{MaxMatrixCols}{20}
   		\begin{bmatrix}
   		d_{1010} & d_{1011}& \cdots& d_{101\mathfrak{g}} &d_{1020} &\cdots & d_{102\mathfrak{g}}&\cdots&d_{102^{\mathfrak{R}-1}0}&\cdots&d_{102^{\mathfrak{R}-1}\mathfrak{g}}\\
   		d_{1110} & d_{1111}& \cdots& d_{111\mathfrak{g}} &d_{1120} &\cdots & d_{112\mathfrak{g}}&\cdots&d_{112^{\mathfrak{R}-1}0}&\cdots&d_{112^{\mathfrak{R}-1}\mathfrak{g}}\\
   		\vdots & \vdots &  & \vdots & \vdots &  &\vdots & &\vdots&  &\vdots\\
   		d_{1\mathfrak{g}10} & d_{1\mathfrak{g}11}& \cdots& d_{1\mathfrak{g}1\mathfrak{g}} &d_{1\mathfrak{g}20} &\cdots & d_{1\mathfrak{g}2\mathfrak{g}}&\cdots&d_{1\mathfrak{g}2^{\mathfrak{R}-1}0}&\cdots&d_{1\mathfrak{g}2^{\mathfrak{R}-1}\mathfrak{g}}\\
   			d_{2010} & d_{2011}& \cdots& d_{201\mathfrak{g}} &d_{2020} &\cdots & d_{202\mathfrak{g}}&\cdots&d_{202^{\mathfrak{R}-1}0}&\cdots&d_{202^{\mathfrak{R}-1}\mathfrak{g}}\\
   		\vdots & \vdots &   & \vdots & \vdots &  &\vdots &  &\vdots&  &\vdots\\
   				d_{2\mathfrak{g}10} & d_{2\mathfrak{g}11}& \cdots& d_{2\mathfrak{g}1\mathfrak{g}} &d_{2\mathfrak{g}20} &\cdots & d_{2\mathfrak{g}2\mathfrak{g}}&\cdots&d_{2\mathfrak{g}2^{\mathfrak{R}-1}0}&\cdots&d_{2\mathfrak{g}2^{\mathfrak{R}-1}\mathfrak{g}}\\
   				\vdots & \vdots &   & \vdots & \vdots &  &\vdots &  &\vdots&  &\vdots\\
   					d_{2^{\mathfrak{R}-1}010} & d_{2^{\mathfrak{R}-1}011}& \cdots& d_{2^{\mathfrak{R}-1}01\mathfrak{g}} &d_{2^{\mathfrak{R}-1}020} &\cdots & d_{2^{\mathfrak{R}-1}02\mathfrak{g}}&\cdots&d_{2^{\mathfrak{R}-1}02^{\mathfrak{R}-1}0}&\cdots&d_{2^{\mathfrak{R}-1}02^{\mathfrak{R}-1}\mathfrak{g}}\\
   					\vdots & \vdots &   & \vdots & \vdots &  &\vdots &  &\vdots&  &\vdots\\
   						d_{2^{\mathfrak{R}-1}\mathfrak{g}10} & d_{2^{\mathfrak{R}-1}\mathfrak{g}11}& \cdots& d_{2^{\mathfrak{R}-1}\mathfrak{g}1\mathfrak{g}} &d_{2^{\mathfrak{R}-1}\mathfrak{g}20} &\cdots & d_{2^{\mathfrak{R}-1}\mathfrak{g}2\mathfrak{g}}&\cdots&d_{2^{\mathfrak{R}-1}\mathfrak{g}2^{\mathfrak{R}-1}0}&\cdots&d_{2^{\mathfrak{R}-1}\mathfrak{g}2^{\mathfrak{R}-1}\mathfrak{g}}
   		\end{bmatrix},
   	   		\end{equation*}
   	   	whose entries can be calculated as follows
   	   	\begin{equation*}
   	   	d_{ijkl}=\left<\mathcal{Q}_{ij}(\varrho),\varphi_{kl}(\varrho)\right>,~~ i,k=1,2,\cdots,2^{\mathfrak{R}-1},~j,l=0,1,2,\cdots\mathfrak{g}~and~<.,.> denotes~ the~ inner ~product.
   	   	\end{equation*}
   	   		
   		The above defined matrix $\hat{D}^{(\alpha_1,\alpha_2,\rho(\alpha))}$ of order ${2^{\mathfrak{R}-1}(\mathfrak{g}+1\times \mathfrak{g}+1)}$  is represents the DOF operational matrix. 
   		 
   		 \begin{Remark}\label{Rem3.1}
   		 Similarly, one can construct the DOF operational matrix namely: $\hat{D}_\varkappa^{(\beta_1,\beta_2,\rho(\beta))}$, $\hat{D}_\eta^{(\beta_1,\beta_2,\rho(\beta))}$, for space direction.
   		 \end{Remark}
   		 		
   		\section{Numerical method}\label{sec:4}
   		In this section we discuss numerical procedure to solve one and two dimensional DOT-SFIDEs.
   		
   		\subsection{1-D distributed order time-space fractional weakly singular integero differential equation}\label{sec:4.2}
   	   	We consider the following DOT--SFWSIPDE  of the form:
   		\begin{align}\label{Eqn:4.1}
   		\int_{\alpha_1}^{\alpha_2}\rho(\alpha)\frac{\partial^{\alpha}\mathbb{U}(\varkappa,\varrho)}{\partial \varrho^{\alpha}}d\alpha+\mathbb{U}(\varkappa,\varrho)=\mathcal{K^*}\int_{\beta_1}^{\beta_2}\rho(\beta)\frac{\partial^{\beta}\mathbb{U}(\varkappa,\varrho)}{\partial \varkappa^{\beta}}d\beta&+\int_0^\varrho(\varrho-\xi)^{-\frac{1}{2}}\left[\frac{\partial^2\mathbb{U}(\varkappa,\xi)}{\partial \varkappa^2}\right]d\xi\nonumber\\
   		&+f(\varkappa,\varrho),
   		\end{align}
   		where, $\mathcal{K^*}$ is viscosity constant and $(\varkappa,\varrho)\in \Omega, \alpha_1=0, \alpha_2=1, \beta_1=1, \beta_2=2$ and $\Omega=[0,1]\times[0,T].$
   		
   		The above  equation \ref{Eqn:4.1} is endowed with the initial condition (IC)
   		\begin{eqnarray}\label{Eqn:4.2}
   		\mathbb{U}(\varkappa,0)=\nu(\varkappa) ,~~  0\textless \varkappa \textless 1,
   		\end{eqnarray}
   		and Dirichlet boundary conditions (BCs) 
   		\begin{eqnarray}\label{Eqn:4.3}
   		\mathbb{U}(0,\varrho)=\mathfrak{p_1}(\varrho),~~~0\textless \varrho\textless T,
   		\end{eqnarray}
   		\begin{eqnarray}\label{Eqn:4.4}
   		\mathbb{U}(0,\varrho)=\mathfrak{p_2}(\varrho), ~~0\textless \varrho\textless T.
   		\end{eqnarray}
   		Consider the approximation of the known and unknown function as
   		\begin{eqnarray}\label{Eqn:4.5}
   		f(\varkappa,\varrho)\approx\Psi^T(\varrho)\mathcal{F}\Psi(\varkappa),
   		\end{eqnarray}
   		\begin{eqnarray}\label{Eqn:4.6}
   		\mathbb{U}(\varkappa,\varrho)\approx\Psi^T(\varrho)\mathcal{A}\Psi(\varkappa),
   		\end{eqnarray}
   		where, the matrix $\mathcal{F}$ is known and $\mathcal{A}=[a_{ij}]$ denotes the  unknown matrix that must be evaluated. The left hand side (L.H.S) of \ref{Eqn:4.1}, by using the approximation of $\mathbb{U}(\varkappa, \varrho)$ can be written as
   		\begin{align}\label{Eqn:7}
   		\int_{\alpha_1}^{\alpha_2}\rho(\alpha)\frac{\partial^{\alpha}\mathbb{U}(\varkappa,\varrho)} {\partial \varrho^{\alpha}}d\alpha
   		&\approx
   		\left(\int_{\alpha_1}^{\alpha_2}\rho(\alpha)(D_\varrho^{\alpha}\Psi^T(\varrho))d\alpha\right)\mathcal{A}\Psi(\varkappa)\nonumber\\
   		&\approx\left(D_\varrho^{\rho(\alpha)}\Psi^T(\varrho)\right)\mathcal{A}\Psi(\varkappa)\nonumber\\
   		&\approx
   		\Psi^T(\varrho)\left(\hat{D}^{(\alpha_1,\alpha_2,\rho(\alpha))}\right)^T\mathcal{A}\Psi(\varkappa).
   		\end{align}
   		Here, $\hat{D}^{(\alpha_1,\alpha_2,\rho(\alpha))}$ denotes the time-DOF operational matrix.
   		
   		Now, the R.H.S of equation \ref{Eqn:4.1}, with the help of the approximation of  $\mathbb{U}(\varkappa,\varrho)$ can be described as
   		\begin{align}\label{Eqn:4.8}
   		\int_{\beta_1}^{\beta_2}\rho(\beta)\frac{\partial^{\beta}\mathbb{U}(\varkappa,\varrho)} {\partial \varkappa^{\beta}}d\beta
   		&\approx
   		\Psi^T(\varrho) \mathcal{A}\left(\int_{\beta_1}^{\beta_2}\rho(\beta)(D_\varkappa^{\beta}\Psi(\varkappa))d\beta\right)\nonumber\\
   		&\approx\Psi^T(\varrho)\mathcal{A}\left(D_\varkappa^{\rho(\beta)}\Psi(\varkappa)\right)\nonumber\\
   		&\approx
   		\Psi^T(\varrho)\mathcal{A}\left(\hat{D}_\varkappa^{(\beta_1,\beta_2,\rho(\beta))}\right)\Psi(\varkappa).
   		\end{align}
   		Here, $\hat{D}_\varkappa^{(\beta_1,\beta_2,\rho(\beta))}$ denotes the space-DOF operational matrix.
   		
   	Now, approximation of the second term of R.H.S of equation \ref{Eqn:4.1} by using the derivative operational matrix of integer order can be written as
   		\begin{eqnarray}\label{Eqn:4.9}
   		\frac{\partial^2\mathbb{U}(\varkappa,\varrho)}{\partial \varkappa^2}\approx\Psi^T(\varrho)\mathcal{A}\left(\frac{d^2}{d\varkappa^2}\Psi(\varkappa)\right)=\Psi^T(\varrho)\mathcal{A}D^{(2)}\Psi(\varkappa).
   		\end{eqnarray}
   		\begin{eqnarray}\label{Eqn:4.10}
   		\int_0^\varrho(\varrho-\xi)^{-\frac{1}{2}}\left[\frac{\partial^2\mathbb{U}(\varkappa,\xi)}{\partial \varkappa^2}\right]d\xi\approx&\int_0^\varrho(\varrho-\xi)^{-\frac{1}{2}}\left(\Psi^T(\xi)\mathcal{A}D^{(2)}\Psi(\varkappa)\right)d\xi\nonumber\\
   		\approx&\left[\int_0^\varrho\frac{\Psi^T(\xi)}{(\xi-\varrho)^{\frac{1}{2}}}d\xi\right]\mathcal{A}D^{(2)}\Psi(\varkappa)\nonumber\\
   		\approx&\Psi^T(\varrho){P^*}^TAD^{(2)}\Psi(\varkappa).
   		\end{eqnarray}
   		By the use of orthogonal property of shifted Legendre polynomials \cite{mohammadi2011new} one can write
   		\begin{equation}\label{Eqn:4.11}
   		(\Psi_{\mathfrak{g},\Lambda}(\varkappa),\Psi_{\mathfrak{g},\Lambda'}^T(\varkappa))=I_{\mathfrak{g}}^{(\Lambda,\Lambda')}=\left(h_j^{\mathfrak{g}}\delta_{ij}\right)_{0\leq i\leq \Lambda~0\leq j\leq \Lambda'}.
   		\end{equation}
   		Grouping equations \ref{Eqn:4.5}-\ref{Eqn:4.10}, the residual term for the equation \ref{Eqn:4.1} is 
   		\begin{eqnarray}\label{Eqn:4.12}
   		Res_{\mathfrak{p},\mathfrak{q}}(\varkappa,\varrho)
   		&\approx&
   		\Psi^T(\varrho)\left((\hat{D}^{(\alpha_1,\alpha_2,\rho(\alpha))})^T\mathcal{A}+\mathcal{A}-\mathcal{K^*}\mathcal{A}(\hat{D}^{(\beta_1,\beta_2,\rho(\beta))})-{P^*}^TAD^{(2)}-\mathcal{F}\right)\Psi(\varkappa)\nonumber\\
   		&=&
   		\Psi^T(\varrho)\mathcal{G}\Psi(\varkappa),
   		\end{eqnarray}
   		where,
   		\begin{equation*}
   		\mathcal{G}=\left((\hat{D}^{(\alpha_1,\alpha_2,\rho(\alpha))})^T\mathcal{A}+\mathcal{A}-\mathcal{K^*}\mathcal{A}(\hat{D}^{(\beta_1,\beta_2,\rho(\beta))})-{P^*}^TAD^{(2)}-\mathcal{F}\right).
   		\end{equation*}
   		
   		Now, the standard tau method \cite{zaky2020multi} is used to construct the following $\mathfrak{g}(\mathfrak{g}-1)$ linear algebraic equations
   		\begin{eqnarray}\label{Eqn:4.13}
   		I_T^{(\Lambda-1,\Lambda)}\left((\hat{D}^{(\alpha_1,\alpha_2,\rho(\alpha))})^T\mathcal{A}+\mathcal{A}-\mathcal{K^*}\mathcal{A}(\hat{D}^{(\beta_1,\beta_2,\rho(\beta))})-{P^*}^TAD^{(2)}-\mathcal{F}\right)I_l^{(\Lambda',\Lambda'-2)}=0.
   		\end{eqnarray}
   		The IC \ref{Eqn:4.2} and BCs \ref{Eqn:4.3}-\ref{Eqn:4.4}, with the help of equation \ref{Eqn:4.6} can be utilised to obtain
   		\begin{eqnarray}\label{Eqn:4.14}
   		\Psi^T(0)\mathcal{A}\Psi(\varkappa) =\nu(\varkappa),
   		\end{eqnarray}
   		\begin{eqnarray}\label{Eqn:4.15}
   		\Psi^T(\varrho)\mathcal{A}\Psi(0)=\mathfrak{p}(\varrho),
   		\end{eqnarray}
   		\begin{eqnarray}\label{Eqn:4.16}
   		\Psi^T(\varrho)\mathcal{A}\Psi(1)=\mathfrak{q}(\varrho).
   		\end{eqnarray}
   		Equation \ref{Eqn:4.14}, acquired with the help of IC and equations \ref{Eqn:4.15}-\ref{Eqn:4.16} are obtained through BCs. We collocate IC \ref{Eqn:4.14} at  $\mathfrak{g}+1$ \& BCs \ref{Eqn:4.15}-\ref{Eqn:4.16} at $\mathfrak{g}$ points. Equations \ref{Eqn:4.13}-\ref{Eqn:4.16} constitute a linear algebraic set of $(\mathfrak{g}+1)^2$ equations which are solved for the unknowns $a_{ij},~i,j=0, \cdots, \mathfrak{g}$. Here we chose the roots of shifted Legendre polynomials as a collocation points.

   		\begin{algorithm}[H]\label{algo1}
   			
   			\SetAlgoLined
   			
   			\KwIn{ The constant $ \mathfrak{R}\in \mathbb{N} $ and $\mathfrak{g}\in \mathbb{N} ~or~ N \in \mathbb{N}, ~~f(\varkappa,\varrho) : L^2(\Omega) \to \mathbb{R}.$}

   			\KwOut{ The approximate solutions $\mathbb{U}(\varkappa,\varrho) \approx \Psi^T(\varrho)\mathcal{A}\Psi(\varkappa)$ for numerical solution of DOT-SFWSIPDE(\ref{Eqn:4.1}-\ref{Eqn:4.4}) by using of operational matrix method.}
   			
   			\For{\normalfont Numerical solution of DOT-SFWSIPDE \ref{Eqn:4.1}-\ref{Eqn:4.4} by using operational matrix method}
   			{
   				\textbf{Step-2.1} Generate the basis function $\varphi_{i}(\varkappa),\varphi_j(\varrho)$; $i,j=0,\dots,\mathfrak{g}$, by using LWs as given in section \ref{sec:2}.       \\
   				\textbf{Step-2.2} Approximate the unknown function $\mathbb{U}(\varkappa,\varrho)$  as given in equation \ref{Eqn:4.6} to get unknown vector $\mathcal{A}$. \\
   				\textbf{Step-2.3} Approximate the term $f(\varkappa,\varrho)$ as $f(\varkappa,\varrho)\approx\Psi^T(\varrho)\mathcal{F}\Psi(\varkappa)$ and obtain known vector $\mathcal{F}$. \\
   				\textbf{Step-2.4} Approximate the distributed order time-fractional and space-fractional operational matrix using section \ref{sec:3.3} as 	$D_\varrho^{\rho(\alpha)}Z(\varkappa,\varrho)\approx \Psi^T(\varrho)\left({\hat{D}}^{(\alpha_1,\alpha_2,\rho(\alpha))}\right)^T\mathcal{A}\Psi(\varkappa)$ and $D_\varkappa^{\rho(\beta)}Z(\varkappa,\varrho)\approx \Psi(\varrho)^T\mathcal{A}\left({\hat{D}}^{(\beta_1,\beta_2,\rho(\beta))}\right)\Psi(\varkappa)$, respectively.  \\
   				\textbf{Step-2.5} Approximate the singular integral operational matrix using section \ref{sec:3.1} and equations \ref{Eqn:4.9}-\ref{Eqn:4.12} as	$\Psi^T(\varrho){P^*}^TAD^{(2)}\Psi(\varkappa)$.\\
   				
   				\textbf{Step-2.6} Compute the residual function $Res_{\mathfrak{p},\mathfrak{q}}(\varkappa,\varrho)$ using equations \ref{Eqn:4.5}-\ref{Eqn:4.10} for equation \ref{Eqn:4.1} we get equation \ref{Eqn:4.12} as follows $Res_{\mathfrak{p},\mathfrak{q}}(\varkappa,\varrho)
   				\approx
   				\Psi^T(\varrho)\left((\hat{D}^{(\alpha_1,\alpha_2,\rho(\alpha))})^T\mathcal{A}+\mathcal{A}-\mathcal{K^*}\mathcal{A}(\hat{D}^{(\beta_1,\beta_2,\rho(\beta))})-{P^*}^TAD^{(2)}-\mathcal{F}\right)\Psi(\varkappa)\nonumber.
   				$  \\
   				\textbf{Step-2.7} Apply standard tau method (use equation \ref{Eqn:4.13}) to create $\mathfrak{g}(\mathfrak{g}-1)$ system of linear algebraic equations.\\
   				\textbf{Step-2.8} Use the initial condition (use equation \ref{Eqn:4.14}) to construct $\mathfrak{g}+1$ system of linear algebraic equations with the help of collocation points.\\
   				\textbf{Step-2.9} Use the boundary conditions (equation \ref{Eqn:4.15} and \ref{Eqn:4.16}) to create $\mathfrak{g}+\mathfrak{g}=2\mathfrak{g}$  linear algebraic system of equations with the help of collocation points. \\ 
   				\textbf{Step-2.10} For getting unknown vector $\mathcal{A}$, to solve the system of $(\mathfrak{g}+1)^2$ algebraic linear equations which is evaluated in step (2.6)-(2.8).\\
   				\textbf{Step-2.11} Put the value of $\mathcal{A}$ in step (2.2) and we get the estimated solution $\mathbb{U}(\varkappa,\varrho)$.\\
   			}
   			\caption{To evaluate the numerical solution of one dimensional distributed order time-space fractional integero-partial differential equation (\ref{Eqn:4.1})-(\ref{Eqn:4.4})}
   			
   		\end{algorithm}

   		\subsection{2-D distributed order time-space fractional weakly singular integero differential equation}\label{sec:4.1}
   		In order to describe the numerical method for solving two dimensional DOT-SFIPDEs, we consider the following DOT-SFWSIDE of the form:
   			\begin{eqnarray}\label{Eqn:4.17}
   			\int_{\alpha_1}^{\alpha_2}\rho(\alpha)\frac{\partial^{\alpha}\mathbb{U}(\varkappa,\eta,\varrho)}{\partial \varrho^{\alpha}}d\alpha+\mathbb{U}(\varkappa,\eta,\varrho)=\mathcal{K^*}\int_{\beta_1}^{\beta_2}\rho(\beta)\left[\frac{\partial^{\beta}\mathbb{U}(\varkappa,\eta,\varrho)}{\partial \varkappa^{\beta}}+\frac{\partial^{\beta}\mathbb{U}(\varkappa,\eta,\varrho)}{\partial \eta^{\beta}}\right]d\beta\nonumber\\
   			+\int_0^\varrho(\varrho-\xi)^{-\frac{1}{2}}\left[\frac{\partial^2\mathbb{U}(\varkappa,\eta,\xi)}{\partial \varkappa^2}+\frac{\partial^2\mathbb{U}(\varkappa,\eta,\xi)}{\partial \eta^2}\right]d\xi+f(\varkappa,\eta,\varrho),
   			\end{eqnarray}
   			Where, $\mathcal{K^*}$ is viscosity constant and $(\varkappa,\eta,\varrho)\in \Omega , \alpha_1=0, \alpha_2=1, \beta_1=1, \beta_2=2$ and $\Omega=[0,1]\times[0,1]\times[0,T].$
   			
   			The above  equation \ref{Eqn:4.17} is acquired with the initial condition (IC)
   			\begin{eqnarray}\label{Eqn:4.18}
   			\mathbb{U}(\varkappa,\eta,0)=\nu(\varkappa,\eta) ,~~  0\textless \varkappa \textless 1~and~0\textless\eta\textless 1,
   			\end{eqnarray}
   			and the boundary conditions (BCs) 
   			\begin{eqnarray}\label{Eqn:4.19}
   			\mathbb{U}(0,\eta,\varrho)=p_1(\eta,\varrho),~~~0\textless\eta\textless 1 ~and~0\textless \varrho\textless T,
   			\end{eqnarray}	
   			\begin{eqnarray}\label{Eqn:4.20}
   			\mathbb{U}(1,\eta,\varrho)=p_2(\eta,\varrho), ~~0\textless\eta\textless 1~and~0\textless \varrho\textless T,
   			\end{eqnarray}
   			\begin{eqnarray}\label{Eqn:4.21}
   			\mathbb{U}(\varkappa,0,\varrho)=q_1(\varkappa,\varrho),~~~0\textless\varkappa\textless 1 ~and~0\textless \varrho\textless T,
   			\end{eqnarray}	
   			\begin{eqnarray}\label{Eqn:4.22}
   			\mathbb{U}(\varkappa,1,\varrho)=q_2(\eta,\varrho), ~~0\textless\varkappa\textless 1~and~0\textless \varrho\textless T.
   			\end{eqnarray}
   			Before discussing the method, we need to give brief about Kronecker product of two matrices \cite{zaky2020multi}.
   			
   			If
   			\begin{equation*}
   			A=\begin{bmatrix}
   			a_{00} & a_{01}  & \cdots & a_{0\mathfrak{g}}\\
   			a_{10} & a_{11}  & \cdots & a_{1\mathfrak{g}}\\
   			a_{20} & a_{21}  & \cdots & a_{2\mathfrak{g}}\\
   			\vdots & \vdots & \ddots & \vdots\\
   			a_{\mathfrak{g}0} & a_{\mathfrak{g}1} & \cdots & a_{\mathfrak{g}\mathfrak{g}}
   			\end{bmatrix}_{(\mathfrak{g}+1)\times (\mathfrak{g}+1)}
   			B=\begin{bmatrix}
   			b_{00} & b_{01}  & \cdots & b_{0\mathfrak{g}}\\
   			b_{10} & b_{11}  & \cdots & b_{1\mathfrak{g}}\\
   			b_{20} & b_{21}  & \cdots & b_{2\mathfrak{g}}\\
   			\vdots & \vdots & \ddots & \vdots\\
   			b_{\mathfrak{g}0} & b_{\mathfrak{g}1} & \cdots & b_{\mathfrak{g}\mathfrak{g}}
   			\end{bmatrix}_{(\mathfrak{g}+1)\times (\mathfrak{g}+1)}
   			\end{equation*}
   			Then
   			\begin{equation*}
   			A\otimes B=	\begin{bmatrix}
   			a_{00}B & a_{01}B  & \cdots & a_{0\mathfrak{g}}B\\
   			a_{10}B & a_{11}B  & \cdots & a_{1\mathfrak{g}}B\\
   			a_{20}B & a_{21}B  & \cdots & a_{2\mathfrak{g}}B\\
   			\vdots & \vdots & \ddots & \vdots\\
   			a_{\mathfrak{g}0}B & a_{\mathfrak{g}1}B & \cdots & a_{\mathfrak{g}\mathfrak{g}}B
   			\end{bmatrix}_{(\mathfrak{g}+1)^2\times (\mathfrak{g}+1)^2}.
   			\end{equation*}
   			\begin{itemize}
   				
   				\item If A and B are lower(upper) triangular, then $A\otimes B$ is also lower(upper) trianular.
   				\item If A and B are band matrices, then $A\otimes B$ is also a band matrix.
   			\end{itemize}
   	Consider the approximation of the known and unknown function as
   			\begin{eqnarray}\label{Eqn:4.23}
   			f(\varkappa,\eta,\varrho)\approx\Psi^T(\varrho)\mathcal{F}\Psi(\varkappa,\eta),
   			\end{eqnarray}
   			\begin{eqnarray}\label{Eqn:4.24}
   			\mathbb{U}(\varkappa,\eta,\varrho)\approx\Psi^T(\varrho)\mathcal{A}\Psi(\varkappa,\eta),
   			\end{eqnarray}
   			where,
   			\begin{equation*}
   				\mathcal{F}=\begin{bmatrix}
   				f_{00} & f_{01} & f_{02} & \cdots & f_{0\mathfrak{g}^2}\\
   				f_{10} & f_{11} & f_{12} & \cdots & f_{1\mathfrak{g}^2}\\
   				f_{20} & f_{21} & f_{22} & \cdots & f_{2\mathfrak{g}^2}\\
   				\vdots & \vdots & \vdots & \ddots & \vdots\\
   				f_{\mathfrak{g}-1~0} & f_{\mathfrak{g}-1~1} & f_{\mathfrak{g}-1~2} & \cdots & f_{\mathfrak{g}-1~\mathfrak{g}^2}\\
   				f_{\mathfrak{g}0} & f_{\mathfrak{g}1} & f_{\mathfrak{g}2} & \cdots & f_{\mathfrak{g}\mathfrak{g}^2}
   				\end{bmatrix}_{(\mathfrak{g}+1)\times (\mathfrak{g}+1)^2},
   				\end{equation*}
   				\begin{equation*}
   				\mathcal{A}=\begin{bmatrix}
   				a_{00} & a_{01} & a_{02} & \cdots & a_{0\mathfrak{g}^2}\\
   				a_{10} & a_{11} & a_{12} & \cdots & a_{1\mathfrak{g}^2}\\
   				a_{20} & a_{21} & a_{22} & \cdots & a_{2\mathfrak{g}^2}\\
   				\vdots & \vdots & \vdots & \ddots & \vdots\\
   				a_{\mathfrak{g}-1~0} & a_{\mathfrak{g}-1~1} & a_{\mathfrak{g}-1~2} & \cdots & a_{\mathfrak{g}-1~\mathfrak{g}^2}\\
   				a_{\mathfrak{g}0} & a_{\mathfrak{g}1} & a_{\mathfrak{g}2} & \cdots & a_{\mathfrak{g}\mathfrak{g}^2}
   				\end{bmatrix}_{(\mathfrak{g}+1)\times (\mathfrak{g}+1)^2},
   				\end{equation*}
   				where, $\Psi(\varkappa,\eta)=(\Psi(\varkappa)\otimes\Psi(\eta))$, the matrix $\mathcal{F}$ is known and $\mathcal{A}=[a_{ij}]$ denotes the  unknown matrix that must be evaluated. The left hand side (L.H.S) of \ref{Eqn:4.17}, by using the approximation of $\mathbb{U}(\varkappa, \eta,\varrho)$ can be written as
   			\begin{align}\label{Eqn:4.25}
   			\int_{\alpha_1}^{\alpha_2}\rho(\alpha)\frac{\partial^{\alpha}\mathbb{U}(\varkappa,\eta,\varrho)} {\partial \varrho^{\alpha}}d\alpha
   			&\approx
   			\left(\int_{\alpha_1}^{\alpha_2}\rho(\alpha)(D_\varrho^{\alpha}\Psi^T(\varrho))d\alpha\right)\mathcal{A}\Psi(\varkappa,\eta)\nonumber\\
   			&\approx\left(D_\varrho^{\rho(\alpha)}\Psi^T(\varrho)\right)\mathcal{A}\Psi(\varkappa,\eta)\nonumber\\
   			&\approx
   			\Psi^T(\varrho)\left(\hat{D}^{(\alpha_1,\alpha_2,\rho(\alpha))}\right)^T\mathcal{A}\Psi(\varkappa,\eta).
   			\end{align}
   			Here, $\hat{D}^{(\alpha_1,\alpha_2,\rho(\alpha))}$ is the time-DOF operational matrix.
   			
   			Now, the R.H.S of equation \ref{Eqn:4.17}, with the help of the approximation of  $\mathbb{U}(\varkappa,\eta,\varrho)$ can be described as 
   			 \begin{align}\label{Eqn:4.26}
   			 \int_{\beta_1}^{\beta_2}\rho(\beta)\frac{\partial^{\beta}\mathbb{U}(\varkappa,\eta,\varrho)} {\partial \varkappa^{\beta}}d\beta
   			 &\approx
   			\Psi^T(\varrho) \mathcal{A}\left(\int_{\beta_1}^{\beta_2}\rho(\beta)(D_\varkappa^{\beta}\Psi(\varkappa,\eta))d\beta\right)\nonumber\\
   			 &\approx\Psi^T(\varrho)\mathcal{A}\left(\left(D_\varkappa^{\rho(\beta)}\Psi(\varkappa)\right)\otimes\Psi(\eta)\right)\nonumber\\
   			 &\approx
   			 \Psi^T(\varrho)\mathcal{A}\left(\left(\hat{D}_\varkappa^{(\beta_1,\beta_2,\rho(\beta))}\Psi(\varkappa)\right)\otimes\Psi(\eta)\right),
   			 \end{align}
   			 and
   			 \begin{align}\label{Eqn:4.27}
   			 \int_{\beta_1}^{\beta_2}\rho(\beta)\frac{\partial^{\beta}\mathbb{U}(\varkappa,\eta,\varrho)} {\partial \eta^{\beta}}d\beta
   			 &\approx
   			 \Psi^T(\varrho) \mathcal{A}\left(\int_{\beta_1}^{\beta_2}\rho(\beta)(D_\eta^{\beta}\Psi(\varkappa,\eta))d\beta\right)\nonumber\\
   			 &\approx\Psi^T(\varrho)\mathcal{A}\left(\Psi(\varkappa)\otimes\left(D_\eta^{\rho(\beta)}\Psi(\eta)\right)\right)\nonumber\\
   			 &\approx
   			 \Psi^T(\varrho)\mathcal{A}\left(\Psi(\varkappa)\otimes\left(\hat{D}_\eta^{(\beta_1,\beta_2,\rho(\beta))}\Psi(\eta)\right)\right).
   			 \end{align}
   			 Here, $\hat{D}_\varkappa^{(\beta_1,\beta_2,\rho(\beta))},~\hat{D}_\eta^{(\beta_1,\beta_2,\rho(\beta))}$ are the space-DOF operational matrices.\\
   			 Now, approximation of the second term of R.H.S of equation \ref{Eqn:4.17} with the help of derivative operational matrix of integer order can be written as
   			\begin{eqnarray}\label{Eqn:4.28}
   			\frac{\partial^2\mathbb{U}(\varkappa,\eta,\varrho)}{\partial \varkappa^2}\approx\Psi^T(\varrho)\mathcal{A}\left(\left(\frac{d^2}{d\varkappa^2}\Psi(\varkappa)\right)\otimes\Psi(\eta)\right)\approx\Psi^T(\varrho)\mathcal{A}\left((D_\varkappa^{(2)}\Psi(\varkappa))\otimes\Psi(\eta)\right).
   			\end{eqnarray}
   			\begin{eqnarray}\label{Eqn:4.29}
   			\frac{\partial^2\mathbb{U}(\varkappa,\eta,\varrho)}{\partial \eta^2}\approx\Psi^T(\varrho)\mathcal{A}\left(\Psi(\varkappa)\otimes\left(\frac{d^2}{d\varkappa^2}\Psi(\eta)\right)\right)\approx\Psi^T(\varrho)\mathcal{A}\left(\Psi(\varkappa)\otimes(D_\eta^{(2)}\Psi(\eta))\right).
   			\end{eqnarray}
   			\begin{eqnarray}\label{Eqn:4.30}
   			\int_0^\varrho(\varrho-\xi)^{-\frac{1}{2}}\left[\frac{\partial^2\mathbb{U}(\varkappa,\eta,\xi)}{\partial \varkappa^2}\right]d\xi\approx&\int_0^\varrho(\varrho-\xi)^{-\frac{1}{2}}\left(\Psi^T(\xi)\mathcal{A}\left((D_\varkappa^{(2)}\Psi(\varkappa))\otimes\Psi(\eta)\right)\right)d\xi\nonumber\\
   			\approx&\left[\int_0^\varrho\frac{\Psi^T(\xi)}{(\varrho-\xi)^{\frac{1}{2}}}d\xi\right]\mathcal{A}\left((D_\varkappa^{(2)}\Psi(\varkappa))\otimes\Psi(\eta)\right)\nonumber\\
   			\approx&\Psi^T(\varrho){P_1^*}^T\mathcal{A}\left((D_\varkappa^{(2)}\Psi(\varkappa))\otimes\Psi(\eta)\right).
   			\end{eqnarray}
   			\begin{eqnarray}\label{Eqn:4.31}
   			\int_0^\varrho(\varrho-\xi)^{-\frac{1}{2}}\left[\frac{\partial^2\mathbb{U}(\varkappa,\eta,\xi)}{\partial \eta^2}\right]d\xi\approx&\int_0^\varrho(\varrho-\xi)^{-\frac{1}{2}}\left(\Psi^T(\xi)\mathcal{A}\left(\Psi(\varkappa)\otimes(D_\eta^{(2)}\Psi(\eta))\right)\right)d\xi\nonumber\\
   			\approx&\left[\int_0^\varrho\frac{\Psi^T(\xi)}{(\varrho-\xi)^{\frac{1}{2}}}d\xi\right]\mathcal{A}\left(\Psi(\varkappa)\otimes(D_\eta^{(2)}\Psi(\eta))\right)\nonumber\\
   			\approx&\Psi^T(\varrho){P_2^*}^T\mathcal{A}\left(\Psi(\varkappa)\otimes(D_\eta^{(2)}\Psi(\eta))\right).
   			\end{eqnarray}
   		substituting all these approximations in equation \ref{Eqn:4.17}
   		\begin{eqnarray}
   		&\Psi^T(\varrho)\left(\hat{D}^{(\alpha_1,\alpha_2,\rho(\alpha))}\right)^T\mathcal{A}\Psi(\varkappa,\eta)+\Psi^T\mathcal{A}(\Psi(\varkappa)\otimes\Psi(\eta))-\mathcal{K^*}\Psi^T(\varrho)\mathcal{A}\left(\left(\hat{D}_\varkappa^{(\beta_1,\beta_2,\rho(\beta))}\Psi(\varkappa)\right)\otimes\Psi(\eta)\right)\nonumber\\
   		&-\mathcal{K^*}\Psi^T(\varrho)\mathcal{A}\left(\Psi(\varkappa)\otimes\left(\hat{D}_\eta^{(\beta_1,\beta_2,\rho(\beta))}\Psi(\eta)\right)\right)-\Psi^T(\varrho){P_1^*}^T\mathcal{A}\left((D_\varkappa^{(2)}\Psi(\varkappa))\otimes\Psi(\eta)\right)\nonumber\\
   		&-\Psi^T(\varrho){P_2^*}^T\mathcal{A}\left(\Psi(\varkappa)\otimes(D_\eta^{(2)}\Psi(\eta))\right)-\Psi^T\mathcal{F}(\Psi^T(\varkappa)\otimes\Psi^T(\eta))=0.
   		\end{eqnarray}
   		Using the tau method \cite{zaky2020multi} with LW operational matrix, we generate  $(\Lambda-1)\times(\Lambda'-2)\times(\Lambda''-2)$ linear algebraic equations.
   		\begin{eqnarray}\label{Eqn:4.33}
   		I_T^{(\Lambda-1,\Lambda)}(\varrho)\left(\hat{D}^{(\alpha_1,\alpha_2,\rho(\alpha))}\right)^T\mathcal{A}(I_l^{(\Lambda',\Lambda''-1)}\otimes I_h^{(\Lambda-1,\Lambda)})+I_T^{(\Lambda-1,\Lambda)}\mathcal{A}(I_l^{(\Lambda'-1,\Lambda')}\otimes I_h^{(\Lambda''-1,\Lambda'')})\nonumber
   		\end{eqnarray}
   		\begin{eqnarray}
   		-\mathcal{K^*}I_T^{(\Lambda-1,\Lambda)}\mathcal{A}\left(\left(\hat{D}_\varkappa^{(\beta_1,\beta_2,\rho(\beta))}I_l^{(\Lambda'-1,\Lambda')}\right)\otimes I_h^{(\Lambda''-1,\Lambda'')}\right)\nonumber
   		\end{eqnarray}
   		\begin{eqnarray}
   		-\mathcal{K^*}I_T^{(\Lambda-1,\Lambda)}\mathcal{A}\left(I_l^{(\Lambda'-1,\Lambda')}\otimes\left(\hat{D}_\eta^{(\beta_1,\beta_2,\rho(\beta))}I_h^{(\Lambda''-1,\Lambda'')}\right)\right)\nonumber
   		\end{eqnarray}
   		\begin{eqnarray}
   		-I_T^{(\Lambda-1,\Lambda)}{P_1^*}^T\mathcal{A}\left((D_\varkappa^{(2)}I_l^{(\Lambda'-1,\Lambda')})\otimes I_h^{(\Lambda''-1,\Lambda'')}\right)\nonumber
   		\end{eqnarray}
   		\begin{eqnarray}
   		-I_T^{(\Lambda-1,\Lambda)}{P_2^*}^T\mathcal{A}\left(I_l^{(\Lambda'-1,\Lambda')}\otimes(D_\eta^{(2)}I_h^{(\Lambda''-1,\Lambda'')})\right)\nonumber
   		\end{eqnarray}
   		\begin{eqnarray}
   		-I_T^{(\Lambda-1,\Lambda)}\mathcal{F}(I_l^{(\Lambda'-1,\Lambda')}\otimes I_h^{(\Lambda''-1,\Lambda'')})=0
   		\end{eqnarray}
   		The IC \ref{Eqn:4.18} and BCs \ref{Eqn:4.19}-\ref{Eqn:4.22}, with the help of equation \ref{Eqn:4.24} can be utilised to obtain
   			\begin{eqnarray}\label{Eqn:4.34}
   			\Psi^T(0)\mathcal{A}(\Psi(\varkappa)\otimes\Psi(\eta))=\nu(\varkappa,\eta). 
   			\end{eqnarray} 		
   			\begin{eqnarray}\label{Eqn:4.35}
   			\Psi^T(\varrho)\mathcal{A}(\Psi(0)\otimes\Psi(\eta))=p_1(\eta,\varrho).
   			\end{eqnarray} 				
   			\begin{eqnarray}\label{Eqn:4.36}
   			\Psi^T(\varrho)\mathcal{A}(\Psi(1)\otimes\Psi(\eta))=p_2(\eta,\varrho).
   			\end{eqnarray}		
   			\begin{eqnarray}\label{Eqn:4.37}
   			\Psi^T(\varrho)\mathcal{A}(\Psi(\varkappa)\otimes\Psi(0))=q_1(\varkappa,\varrho).
   			\end{eqnarray}		
   			\begin{eqnarray}\label{Eqn:4.38}
   		\Psi^T(\varrho)\mathcal{A}(\Psi(\varkappa)\otimes\Psi(1))=q_2(\varkappa,\varrho).
   			\end{eqnarray}
   			Equation \ref{Eqn:4.34}, acquired with the help of IC and equations \ref{Eqn:4.35}-\ref{Eqn:4.38} are obtained through BCs. We collocate IC \ref{Eqn:4.34} at  $(\mathfrak{g+1})^2$ \& BCs \ref{Eqn:4.35}-\ref{Eqn:4.38} at $4(\mathfrak{g})^2$ points. Equations \ref{Eqn:4.33}-\ref{Eqn:4.38} constitute linear algebraic set of $(\mathfrak{g}+1)^3$ equations which are solved for the unknowns $a_{ij},~i=0,1 \cdots, \mathfrak{g}~\&~j=0,1,\cdots,\mathfrak{g}^2$.
   			Here we chose the roots of shifted Legendre polynomials as a collocation points.

   			\begin{algorithm}[H]\label{algo2}
   				
   				\SetAlgoLined
   				
   				\KwIn{ The constant $ K\in \mathbb{N} $ and $\mathfrak{g} \in \mathbb{N} ~or~ N \in \mathbb{N}, ~~f(\varkappa,\eta,\varrho):L^2(\Omega=[0,1]\times[0,1]\times[0,1]) \to \mathbb{R}.$}

   				\KwOut{ The approximate solutions $\mathbb{U}(\varkappa,\eta,\varrho) \approx \Psi^T(\varrho)\mathcal{A}\Psi(\varkappa,\eta)$ for numerical solution of DOT-SFWSIPDE \ref{Eqn:4.17}-\ref{Eqn:4.22} by using of  operational matrix method.}
   				
   				\For{\normalfont Numeical solution of DOT-SFWSIPDE \ref{Eqn:4.17}-\ref{Eqn:4.22} by using of operational matrix method}
   				{
   					\textbf{Step-2.1} Generate the basis function $\varphi_{i}(\varkappa),\varphi_j(\varrho),\varphi_k(\eta)$; $i,j,k=0,\dots,\mathfrak{g}$ by using LWs as given in section \ref{sec:2}.       \\
   					\textbf{Step-2.2} Approximate the unknown function $\mathbb{U}(\varkappa,\eta,\varrho)$  as given in equation \ref{Eqn:4.24} to get unknown vector $\mathcal{A}$. \\
   					\textbf{Step-2.3} Approximate the source term $f(\varkappa,\eta,\varrho)$ as $f(\varkappa,\eta,\varrho)\approx\Psi^T(\varrho)\mathcal{F}\Psi(\varkappa,\eta)$ to get the known vector say $\mathcal{F}$. \\
   					\textbf{Step-2.4} Approximate the distributed order time-fractional and space-fractional operational matrix using section \ref{sec:3.3}	$D_\varrho^{\rho(\alpha)}\mathbb{U}(\varkappa,\eta,\varrho)\approx \Psi^T(\varrho)\left({\hat{D}}^{(\alpha_1,\alpha_2,\rho(\alpha))}\right)^T\mathcal{A}(\Psi(\varkappa)\otimes\Psi(\eta))$, $D_\varkappa^{\rho(\beta)}\mathbb{U}(\varkappa,\eta,\varrho)\approx \Psi(\varrho)^T\mathcal{A}\left({\hat{D}}^{(\beta_1,\beta_2,\rho(\beta))}\Psi(\varkappa)\otimes\Psi(\eta)\right)$ and
   					$D_\eta^{\rho(\beta)}\mathbb{U}(\varkappa,\eta,\varrho)\approx \Psi(\varrho)^T\mathcal{A}\Psi(\varkappa)\otimes\left({\hat{D}}^{(\beta_1,\beta_2,\rho(\beta))}\Psi(\eta)\right)$ respectively.\\
   					\textbf{Step-2.5} Approximate the singular integral operational matrix using section \ref{sec:3.1} and equation \ref{Eqn:4.30}-\ref{Eqn:4.31} as	$\Psi^T(\varrho){P_1^*}^TA(D^{(2)}\Psi(\varkappa)\otimes\Psi(\eta)$ and $\Psi^T(\varrho){P_2^*}^TA(\Psi(\varkappa)\otimes D^{(2)}\Psi(\eta))$, respectively.\\
   					
   					\textbf{Step-2.6} Compute the matrix $I_T^{(\Lambda-1,\Lambda)}, I_l^{(\Lambda',\Lambda'-2)}$ and $I_h^{(\Lambda'',\Lambda''-2)}$ with the help of equation \ref{Eqn:4.11}.\\
   					\textbf{Step-2.7} Apply standard tau method (use equation \ref{Eqn:4.33}) to create $\mathfrak{g}(\mathfrak{g}-1)^2$ system of linear algebraic equations.\\
   					\textbf{Step-2.8} Use the initial condition (use equation \ref{Eqn:4.34}) to construct $(\mathfrak{g}+1)^2$ system of linear algebraic equations with the help of collocation points.\\
   					\textbf{Step-2.9} Use the boundary conditions (equations \ref{Eqn:4.35}-\ref{Eqn:4.38}) to create $4\mathfrak{g}^2$  linear algebraic system of equations with the help of collocation points. \\ 
   					\textbf{Step-2.10} For getting unknown vector $\mathcal{A}$, to solve the system of $(\mathfrak{g}+1)^3$ linear algebraic equations which is evaluated in step (2.7)-(2.9).\\
   					\textbf{Step-2.11} Put the value of $\mathcal{A}$ in step (2.2) and we get the estimated solution $\mathbb{U}(\varkappa,\eta,\varrho)$.\\
   				}
   				\caption{To evaluate the numerical solution of two dimensional distributed order time-space fractional integero-partial differential equation \ref{Eqn:4.17}-\ref{Eqn:4.22}.}			
   			\end{algorithm}

   		\section{Error bounds and convergence analysis}\label{sec:5}

   		\begin{Theorem}\label{Thm:5.01}
   				Assume $\{f(\varkappa,\varrho)\}_{\mathcal{N}},~ \{v(\varkappa,\varrho)\}_{\mathcal{N}},~  \{w(\varkappa,\varrho)\}_{\mathcal{N}},~
   			 \{w_1(\varkappa,\varrho)\}_{\mathcal{N}},~ \{\mathbb{U}(\varkappa,\varrho)\}_{\mathcal{N}}$ be the approximate solutions  w.r.t. continuous functions $f(\varkappa,\varrho), ~v(\varkappa,\varrho),~ w(\varkappa,\varrho),~ w_1(\varkappa,\varrho),~ \mathbb{U}(\varkappa,\varrho)$, respectively, defined over the domain $\Omega(=[0,1]\times[0,1])$ with  second order bounded mixed derivative, say $\left\lvert \frac{\partial^4f(\varkappa,\varrho)}{\partial \varkappa^2\partial \varrho^2}\right\rvert\leq\mathcal{B},~ \left\lvert \frac{\partial^4v(\varkappa,\varrho)}{\partial \varkappa^2\partial \varrho^2}\right\rvert\leq\mathcal{B}_1, ~\left\lvert \frac{\partial^4w(\varkappa,\varrho)}{\partial \varkappa^2\partial \varrho^2}\right\rvert\leq\mathcal{B}_2, ~\left\lvert \frac{\partial^4w_1(\varkappa,\varrho)}{\partial \varkappa^2\partial \varrho^2}\right\rvert\leq\mathcal{B}_3, ~ \left\lvert\frac{\partial^4\mathbb{U}(\varkappa,\varrho)}{\partial \varkappa^2\partial \varrho^2}\right\rvert\leq\mathcal{B}_0$, for some positive constant $\mathcal{B},~ \mathcal{B}_1,~ \mathcal{B}_2, ~\mathcal{B}_3,~ \mathcal{B}_0$, where, $v=\frac{\partial \mathbb{U}(\varkappa,\varrho)}{\partial\varrho}$,     $w=\frac{\partial^2 \mathbb{U}(\varkappa,\varrho)}{\partial\varkappa^2}~and~w_1=\frac{\partial \mathbb{U}(\varkappa,\varrho)}{\partial\varkappa}$.  Then\\			
   				$(a)$ $H(\varkappa,\varrho)$ can be expressed in terms of LWs infinite series which converges uniformly to the function $H(\varkappa,\varrho)$, that is
   				\begin{equation*}
   				H(\varkappa,\varrho)=\sum_{\mathfrak{h}=1}^\infty\sum_{\mathfrak{g}=1}^\infty\sum_{s'=1}^\infty\sum_{\mathfrak{g}'=1}^\infty C_{\mathfrak{h}\mathfrak{g} \mathfrak{h}'\mathfrak{g}'}\Psi_{\mathfrak{h}\mathfrak{g} \mathfrak{h}'\mathfrak{g}'}(\varkappa,\varrho),
   				\end{equation*}
   				where, $C_{\mathfrak{h}\mathfrak{g} \mathfrak{h}'\mathfrak{g}'}=\left<H(\varkappa,\varrho),\Psi_{\mathfrak{h}\mathfrak{g} \mathfrak{h}'\mathfrak{g}'}\right>_{L^2(\Omega)}$ and
   				$H(\varkappa,\varrho)\in\left\{f(\varkappa,\varrho), \frac{\partial \mathbb{U}}{\partial\varrho}, \frac{\partial^2 \mathbb{U}}{\partial \varkappa^2}, \frac{\partial Z}{\partial \varkappa}, \mathbb{U}(\varkappa,\varrho)\right\}.$\\
   				(b) The bound of error is \begin{equation*}
   				\left||\epsilon_{H}\right||_{L^2}^2\leq {9{\mathcal{B}_*}^2}\sum_{\mathfrak{h}=2^{\mathfrak{R}-1}+1}^\infty\sum_{\mathfrak{g}=\Lambda}^\infty\sum_{\mathfrak{h}'=2^{\mathfrak{R}'-1}+1}^\infty\sum_{\mathfrak{g}'=\Lambda'}^\infty \frac{1}{256(\mathfrak{h}\mathfrak{h}')^5(2\mathfrak{g}-3)^4(2\mathfrak{g}'-3)^4},
   				\end{equation*}
   				where,
   				\begin{equation*}
   				\left||\epsilon_H\right||_{L^2}^2=\int_0^1\int_0^1\lvert H(\varkappa,\varrho)-\sum_{\mathfrak{h}=1}^{2^{\mathfrak{R}-1}}\sum_{\mathfrak{g}=0}^{\Lambda-1}\sum_{\mathfrak{h}'=1}^{2^{\mathfrak{R}'-1}}\sum_{\mathfrak{g}'=0}^{\Lambda'-1}C_{\mathfrak{h}\mathfrak{g} \mathfrak{h}'\mathfrak{g}'}\Psi_{\mathfrak{h}\mathfrak{g} \mathfrak{h}'\mathfrak{g}'}(\varkappa,\varrho)\rvert^2 d\varkappa d\varrho,
   				\end{equation*} 				
   				and $\mathcal{B_*}\in\{\mathcal{B}, \mathcal{B}_1, \mathcal{B}_2, \mathcal{B}_3, \mathcal{B}_0\}$. It should be noted that $\mathcal{B},~\mathcal{B}_1,~\mathcal{B}_2,~\mathcal{B}_3,~\mathcal{B}_0$ correspond to $f,~v,~w,~w_1,~\mathbb{U}$, respectively.
   				\begin{proof}
   					The proof of this theorem is similar to the proof of theorem 1 given in reference \cite{singh2018application}. Consider, equation (22) and applying the inequality $2^\mathfrak{R}\geq2\mathfrak{h}$ \&
   					$2^{\mathfrak{R}'}\geq2{\mathfrak{h}'}$ for $\mathfrak{R},~\mathfrak{R}',~\mathfrak{h},~\mathfrak{h}'\in \mathbb{Z^{+}}$,
   					we have, $\frac{1}{2^{\mathfrak{R}}}\leq\frac{1}{2\mathfrak{h}}$ \& $\frac{1}{2^{\mathfrak{R}'}}\leq\frac{1}{2{\mathfrak{h}'}}$.\\ 
   					Then
   					\begin{equation*}
   					\left||\epsilon_{H}\right||_{L^2}^2\leq {9{\mathcal{B_*}}^2}\sum_{\mathfrak{h}=2^{\mathfrak{R}-1}+1}^\infty\sum_{\mathfrak{g}=\Lambda}^\infty\sum_{\mathfrak{h}'=2^{\mathfrak{R}'-1}+1}^\infty\sum_{\mathfrak{g}'=\Lambda'}^\infty \frac{1}{256(\mathfrak{h}\mathfrak{h}')^5(2\mathfrak{g}-3)^4(2\mathfrak{g}'-3)^4},
   					\end{equation*}
   					where,
   					\begin{equation*}
   					\left||\epsilon_H\right||_{L^2}^2=\int_0^1\int_0^1\lvert H(\varkappa,\varrho)-\sum_{\mathfrak{h}=1}^{2^{K-1}}\sum_{\mathfrak{g}=0}^{\Lambda-1}\sum_{\mathfrak{h}'=1}^{2^{\mathfrak{R}'-1}}\sum_{\mathfrak{g}'=0}^{\Lambda'-1}C_{\mathfrak{h}\mathfrak{g} \mathfrak{h}'\mathfrak{g}'}\Psi_{\mathfrak{h}\mathfrak{g} \mathfrak{h}'\mathfrak{g}'}(\varkappa,\varrho)\rvert^2 d\varkappa d\varrho.
   					\end{equation*}
   					$~~~~~~~~~~~~H(\varkappa,\varrho)\in\left\{f(\varkappa,\varrho), \frac{\partial \mathbb{U}}{\partial\varrho}, \frac{\partial^2 \mathbb{U}}{\partial \varkappa^2}, \frac{\partial \mathbb{U}}{\partial \varkappa}, \mathbb{U}(\varkappa,\varrho)\right\}$ and $\mathcal{B_*}\in\{\mathcal{B}, \mathcal{B}_1, \mathcal{B}_2, \mathcal{B}_3, \mathcal{B}_0\}$.
   				\end{proof}		
   		\end{Theorem}	
 \begin{Theorem}
 	Let $\{J(\varkappa,\varrho)\}_\mathcal{N}$ represents the approximate solution of continuous function $J(\varkappa,\varrho)$ for $\varkappa,\varrho\in[0,1]$ such that $\left|\frac{\partial^6\mathbb{U}(\varkappa,\varrho)}{\partial \varkappa^4\partial\xi^2}\right|\leq\mathcal{B}_2$, where $\mathcal{B}_2$ is a positive constant, then
 	\begin{align*}
 	\|J(\varkappa,\varrho)-(J(\varkappa,\varrho))_\mathcal{N}\|_2^2\leq 36 \mathcal{B}_2\sum_{\mathfrak{h}=2^{\mathfrak{R}-1}+1}^\infty\sum_{\mathfrak{g}=\Lambda}^\infty\sum_{\mathfrak{h}'=2^{\mathfrak{R}'-1}+1}^\infty\sum_{\mathfrak{g}'=\Lambda'}^\infty \frac{1}{256(\mathfrak{h}\mathfrak{h}')^5(2\mathfrak{g}-3)^4(2\mathfrak{g}'-3)^4},
 	\end{align*}
 	where, $J(\varkappa,\varrho)$=$\int_0^\varrho(\varrho-\xi)^{-\frac{1}{2}}\left[\frac{\partial^2\mathbb{U}(\varkappa,\xi)}{\partial \varkappa^2}\right]d\xi$.
 \end{Theorem}
 	\begin{proof}
 		Consider the term $|| J(\varkappa,\varrho)-\{J(\varkappa,\varrho)\}_\mathcal{N}||_2^2$ 
 		\begin{align*}
 		|| J(\varkappa,\varrho)-\{J(\varkappa,\varrho)\}_\mathcal{N}||_2^2= \left\Vert\int_0^\varrho(\varrho-\xi)^{-\frac{1}{2}}\left(\frac{\partial^2 \mathbb{U}(\varkappa,\xi)}{\partial\varkappa^2}-\left(\frac{\partial^2 \mathbb{U}(\varkappa,\xi)}{\partial\varkappa^2}\right)_\mathcal{N}\right)d\xi\right\Vert_2^2
 		\end{align*}
 		Now using theorem \ref{Thm:5.01}, then we have
 		\begin{align*}
 		|| J(\varkappa,\varrho)-\{J(\varkappa,\varrho)\}_\mathcal{N}||_2^2&\leq\left[\int_0^\varrho(\varrho-\xi)^{-\frac{1}{2}}d\xi\right]^2\left\Vert\frac{\partial^2\mathbb{U}(\varkappa,\varrho)}{\partial\varkappa^2}-\left(\frac{\partial^2\mathbb{U}(\varkappa,\varrho)}{\partial\varkappa^2}\right)_\mathcal{N}\right\Vert_2^2\\
 		&\leq \left[9\mathcal{B}_2\sum_{\mathfrak{h}=2^{\mathfrak{R}-1}+1}^\infty\sum_{\mathfrak{g}=\Lambda}^\infty\sum_{\mathfrak{h}'=2^{\mathfrak{R}'-1}+1}^\infty\sum_{\mathfrak{g}'=\Lambda'}^\infty \frac{1}{256(\mathfrak{h}\mathfrak{h}')^5(2\mathfrak{g}-3)^4(2\mathfrak{g}'-3)^4}\right]\\
 		&\times\left[\int_0^\varrho(\varrho-\xi)^{-\frac{1}{2}}d\xi\right]^2\\
 		&=(4\varrho)\times\left[9\mathcal{B}_2\sum_{\mathfrak{h}=2^{\mathfrak{R}-1}+1}^\infty\sum_{\mathfrak{g}=\Lambda}^\infty\sum_{\mathfrak{h}'=2^{\mathfrak{R}'-1}+1}^\infty\sum_{\mathfrak{g}'=\Lambda'}^\infty \frac{1}{256(\mathfrak{h}\mathfrak{h}')^5(2\mathfrak{g}-3)^4(2\mathfrak{g}'-3)^4}\right]\\
 		&\leq 36\mathcal{B}_2\sum_{\mathfrak{h}=2^{\mathfrak{R}-1}+1}^\infty\sum_{\mathfrak{g}=\Lambda}^\infty\sum_{\mathfrak{h}'=2^{\mathfrak{R}'-1}+1}^\infty\sum_{\mathfrak{g}'=\Lambda'}^\infty \frac{1}{256(\mathfrak{h}\mathfrak{h}')^5(2\mathfrak{g}-3)^4(2\mathfrak{g}'-3)^4}.
 		\end{align*}
\end{proof}

   \begin{Theorem}\label{Thm:5.4}
   	Let $\left(\frac{\partial^\alpha \mathbb{U}(\varkappa,\varrho)}{\partial \varrho^\alpha}\right)_\mathcal{N}$ and $\left(\frac{\partial^\beta \mathbb{U}(\varkappa,\varrho)}{\partial \varkappa^\beta}\right)_\mathcal{N}$ be the approximation of continuous function $\frac{\partial^\alpha \mathbb{U}(\varkappa,\varrho)}{\partial \varrho^\alpha}$ and $\frac{\partial^\beta \mathbb{U}(\varkappa,\varrho)}{\partial \varkappa^\beta}$ $\alpha\in(0,1),~\beta \in(1,2)$  such that  $\left\lvert \frac{\partial^5\mathbb{U}(\varkappa,\varrho)}{\partial \varkappa^2\partial \varrho^3}\right\rvert\textless\mathcal{B}_1$ and $\left\lvert \frac{\partial^5\mathbb{U}(\varkappa,\varrho)}{\partial \varkappa^3\partial \varrho^2}\right\rvert\textless\mathcal{B}_3$, where $\mathcal{B}_1,~\mathcal{B}_3$ are positive constants, then
   	\begin{align*}
   	&\left\Vert\frac{\partial^\alpha \mathbb{U}(\varkappa,\varrho)}{\partial \varrho^\alpha}-\left(\frac{\partial^\alpha \mathbb{U}(\varkappa,\varrho)}{\partial \varrho^\alpha}\right)_\mathcal{N}\right\Vert_{L^2}^2\\
   	&\leq\frac{9{\mathcal{B}_1}^2}{[\Gamma{(2-\alpha)}]^2}\sum_{\mathfrak{h}=2^{\mathfrak{R}-1}+1}^\infty\sum_{\mathfrak{g}=\Lambda}^\infty\sum_{\mathfrak{h}'=2^{\mathfrak{R}'-1}+1}^\infty\sum_{\mathfrak{g}'=\Lambda'}^\infty \frac{1}{256(\mathfrak{h}\mathfrak{h}')^5(2\mathfrak{g}-3)^4(2\mathfrak{g}'-3)^4},
   	\end{align*}
   	and,
   	\begin{align*}
   	&\left\Vert\frac{\partial^\beta \mathbb{U}(\varkappa,\varrho)}{\partial \varkappa^\beta}-\left(\frac{\partial^\beta \mathbb{U}(\varkappa,\varrho)}{\partial \varkappa^\beta}\right)_\mathcal{N}\right\Vert_{L^2}^2\\
   	&\leq\frac{(2!)^2\times 9{\mathcal{B}_3}^2}{[\Gamma{(3-\beta)}]^2}\sum_{\mathfrak{h}=2^{\mathfrak{R}-1}+1}^\infty\sum_{\mathfrak{g}=\Lambda}^\infty\sum_{\mathfrak{h}'=2^{\mathfrak{R}'-1}+1}^\infty\sum_{\mathfrak{g}'=\Lambda'}^\infty \frac{1}{256(\mathfrak{h}\mathfrak{h}')^5(2\mathfrak{g}-3)^4(2\mathfrak{g}'-3)^4}.
   	\end{align*}
   \end{Theorem}
   \begin{proof}
   	Consider the term $\left\Vert\frac{\partial^\alpha \mathbb{U}(\varkappa,\varrho)}{\partial \varrho^\alpha}-\left(\frac{\partial^\alpha \mathbb{U}(\varkappa,\varrho)}{\partial \varrho^\alpha}\right)_\mathcal{N}\right\Vert_2^2$ and using equation \ref{Eqn:2.1} \& theorem \ref{Thm:5.01}, we have
   	\begin{align*}
   	\left\Vert\frac{\partial^\alpha \mathbb{U}(\varkappa,\varrho)}{\partial \varrho^\alpha}-\left(\frac{\partial^\alpha \mathbb{U}(\varkappa,\varrho)}{\partial \varrho^\alpha}\right)_\mathcal{N}\right\Vert_2^2=&\left\Vert\frac{1}{\Gamma(1-\alpha)}\int_0^\varrho(\varrho-s)^{-\alpha}\left(\frac{\partial \mathbb{U}(\varkappa,\varrho)}{\partial \varrho}-\left(\frac{\partial \mathbb{U}(\varkappa,\varrho)}{\partial \varrho}\right)_\mathcal{N}\right)ds\right\Vert_2^2
   	\end{align*}
   	\begin{align*}
   	\leq& \left[\frac{1}{\Gamma(1-\alpha)}\int_0^\varrho(\varrho-s)^{-\alpha}ds\right]^2\left\Vert\frac{\partial \mathbb{U}(\varkappa,\varrho)}{\partial \varrho}-\left(\frac{\partial \mathbb{U}(\varkappa,\varrho)}{\partial \varrho}\right)_\mathcal{N}\right\Vert_ 2^2\\
   	\leq& \left[\frac{9{\mathcal{B}_1}^2}{[\Gamma(1-\alpha)]^2}\sum_{\mathfrak{h}=2^{\mathfrak{R}-1}+1}^\infty\sum_{\mathfrak{g}=\Lambda}^\infty\sum_{\mathfrak{h}'=2^{\mathfrak{R}'-1}+1}^\infty\sum_{\mathfrak{g}'=\Lambda'}^\infty \frac{1}{256(\mathfrak{h}\mathfrak{h}')^5(2\mathfrak{g}-3)^4(2\mathfrak{g}'-3)^4}\right]\left[\int_0^\varrho(\varrho-s)^{-\alpha}ds\right]^2\\
   	=&
   	\left[\frac{\varrho^{1-\alpha}}{\Gamma(1-\alpha)(1-\alpha)}\right]^2\left[{9{\mathcal{B}_1}^2}\sum_{\mathfrak{h}=2^{\mathfrak{R}-1}+1}^\infty\sum_{\mathfrak{g}=\Lambda}^\infty\sum_{\mathfrak{h}'=2^{\mathfrak{R}'-1}+1}^\infty\sum_{\mathfrak{g}'=\Lambda'}^\infty \frac{1}{256(\mathfrak{h}\mathfrak{h}')^5(2\mathfrak{g}-3)^4(2\mathfrak{g}'-3)^4}\right]\\
   	\leq&
   	\frac{9{\mathcal{B}_1}^2}{[\Gamma{(2-\alpha)}]^2}\sum_{\mathfrak{h}=2^{\mathfrak{R}-1}+1}^\infty\sum_{\mathfrak{g}=\Lambda}^\infty\sum_{\mathfrak{h}'=2^{\mathfrak{R}'-1}+1}^\infty\sum_{\mathfrak{g}'=\Lambda'}^\infty \frac{1}{256(\mathfrak{h}\mathfrak{h}')^5(2\mathfrak{g}-3)^4(2\mathfrak{g}'-3)^4}.
   	\end{align*}
   	
   	Now	consider the term $\left\Vert\frac{\partial^\beta \mathbb{U}(\varkappa,\varrho)}{\partial \varkappa^\beta}-\left(\frac{\partial^\beta \mathbb{U}(\varkappa,\varrho)}{\partial \varkappa^\beta}\right)_\mathcal{N}\right\Vert_2^2$ and again applying the definition of Caputo derivative (section \ref{sec:2}) for vector $\varkappa$ along with theorem \ref{Thm:5.01}, we can write
   	\begin{align*}
   	\left\Vert\frac{\partial^\beta \mathbb{U}(\varkappa,\varrho)}{\partial \varkappa^\beta}-\left(\frac{\partial^\beta \mathbb{U}(\varkappa,\varrho)}{\partial \varkappa^\beta}\right)_\mathcal{N}\right\Vert_2^2=&\left\Vert\frac{1}{\Gamma(2-\beta)}\int_0^\varkappa(\varkappa-s)^{-1-\beta}\left(\frac{\partial^2 \mathbb{U}(\varkappa,\varrho)}{\partial \varkappa^2}-\left(\frac{\partial^2 \mathbb{U}(\varkappa,\varrho)}{\partial \varkappa^2}\right)_\mathcal{N}\right)ds\right\Vert_2^2
   	\end{align*}
   	\begin{align*}
   	\leq& \left[\frac{1}{\Gamma(2-\gamma)}\int_0^\varkappa(\varkappa-s)^{-1-\beta}ds\right]^2\left\Vert\frac{\partial^2 \mathbb{U}(\varkappa,\varrho)}{\partial \varkappa^2}-\left(\frac{\partial^2 \mathbb{U}(\varkappa,\varrho)}{\partial \varkappa^2}\right)_\mathcal{N}\right\Vert_ 2^2\\
   	\leq& \left[\frac{9{\mathcal{B}_2}^2}{[\Gamma(2-\beta)]^2}\sum_{\mathfrak{h}=2^{\mathfrak{R}-1}+1}^\infty\sum_{\mathfrak{g}=\Lambda}^\infty\sum_{\mathfrak{h}'=2^{\mathfrak{R}'-1}+1}^\infty\sum_{\mathfrak{g}'=\Lambda'}^\infty \frac{1}{256(\mathfrak{h}\mathfrak{h}')^5(2\mathfrak{g}-3)^4(2\mathfrak{g}'-3)^4}\right]\left[\int_0^\varkappa(\varkappa-s)^{-1-\beta}ds\right]^2\\
   	=&
   	\left[2!\frac{\varkappa^{2-\beta}}{\Gamma(2-\beta)(2-\beta)}\right]^2\left[{9{\mathcal{B}_2}^2}\sum_{\mathfrak{h}=2^{\mathfrak{R}-1}+1}^\infty\sum_{\mathfrak{g}=\Lambda}^\infty\sum_{\mathfrak{h}'=2^{\mathfrak{R}'-1}+1}^\infty\sum_{\mathfrak{g}'=\Lambda'}^\infty \frac{1}{256(\mathfrak{h}\mathfrak{h}')^5(2\mathfrak{g}-3)^4(2\mathfrak{g}'-3)^4}\right]\\
   	\leq&
   	\frac{(2!)^2\times 9{\mathcal{B}_2}^2}{[\Gamma{(3-\beta)}]^2}\sum_{\mathfrak{h}=2^{\mathfrak{R}-1}+1}^\infty\sum_{\mathfrak{g}=\Lambda}^\infty\sum_{\mathfrak{h}'=2^{\mathfrak{R}'-1}+1}^\infty\sum_{\mathfrak{g}'=\Lambda'}^\infty \frac{1}{256(\mathfrak{h}\mathfrak{h}')^5(2\mathfrak{g}-3)^4(2\mathfrak{g}'-3)^4}.
   	\end{align*}
   \end{proof}
   Now, define
   \begin{align*}
   L_1(\mathbb{U}(\varkappa,\varrho))=\int_{\alpha_1}^{\alpha_2}\rho(\alpha)\frac{\partial^{\alpha}\mathbb{U}(\varkappa,\varrho)}{\partial \varrho^{\alpha}}d\alpha+\mathbb{U}(\varkappa,\varrho)-\mathcal{K^*}\int_{\beta_1}^{\beta_2}\rho(\beta)\frac{\partial^{\beta}\mathbb{U}(\varkappa,\varrho)}{\partial\varkappa^{\beta}}d\beta\\
   -\int_0^\varrho(\varrho-\xi)^{-\frac{1}{2}}\left[\frac{\partial^2\mathbb{U}(\varkappa,\xi)}{\partial \varkappa^2}\right]d\xi
   -f(\varkappa,\varrho),
   \end{align*}
   \begin{align*}
   L_2(\mathbb{U}(\varkappa,\varrho))=\sum_{s=1}^P w_s\rho(\sigma_s)\left(\frac{\partial\sigma_s \mathbb{U}(\varkappa,\varrho)}{\partial \varrho^{\sigma_s}}\right)+\mathbb{U}(\varkappa,\varrho)-\mathcal{K^*}\sum_{r=1}^{P^*} w_r\rho(\sigma_r)\left(\frac{\partial\sigma_r \mathbb{U}(\varkappa,\varrho)}{\partial \varkappa^{\sigma_r}}\right)\\
   -\int_0^\varrho(\varrho-\xi)^{-\frac{1}{2}}\left[\frac{\partial^2\mathbb{U}(\varkappa,\xi)}{\partial \varkappa^2}\right]d\xi
   -f(\varkappa,\varrho).
   \end{align*}
   Also consider
   \begin{align*}
   L_1(\mathbb{U}(\varkappa,\varrho))-L_2(\mathbb{U}(\varkappa,\varrho))=P_1(P,\varkappa,\varrho)+P_2(P^*,\varkappa,\varrho).
   \end{align*}

   Let $P_1(P,\varkappa,\varrho)$ and $P_2(P^*,\varkappa,\varrho)$ denote the error for using P-point and $P^*$-point LGQ formula, respectively then
   \begin{eqnarray}
   P_1(P,\varkappa,\varrho)
   &=&
   \frac{(P!)^4}{(2P+1)(2P!)^4}\frac{\partial^{2P}}{\partial\alpha^{2P}}H_1(\varkappa,\varrho,\alpha)\nonumber\\
      &\approx& 
      \frac{\pi}{4^P}\frac{\partial^{2P}}{\partial \alpha^{2P}}H_1(\varkappa,\varrho,\alpha),~~\alpha\in[0,1].\nonumber\\
   \end{eqnarray}
   And
   \begin{eqnarray}
   P_2(P^*,\varkappa,\varrho)
   &=&
   \frac{(P^*!)^4}{(2P^*+1)(2P^*!)^4}\frac{\partial^{2P^*}}{\partial\beta^{2P^*}}H_2(\varkappa,\varrho,\beta)\nonumber\\
   &\approx& 
   \frac{\pi}{4^{P^*}}\frac{\partial^{2P^*}}{\partial \beta^{2P}}H_2(\varkappa,\varrho,\beta),~~\beta\in[1,2].\nonumber\\
   \end{eqnarray}
    Where,
   \begin{equation*}
   H_1(\varkappa,\varrho,\alpha)=\rho(\alpha)\frac{\partial^\alpha}{\partial \varrho^\alpha}\mathbb{U}(\varkappa,\varrho) ~and~H_2(\varkappa,\varrho,\beta)=\rho(\beta)\frac{\partial^\beta}{\partial \varkappa^\beta}\mathbb{U}(\varkappa,\varrho).
   \end{equation*}
 Now, for $H_1(\varkappa,\varrho,\alpha)\in C^{2P}([0,1]),~H_2(\varkappa,\varrho,\beta)\in C^{2P^*}([1,2])$ and fixed $(\varkappa,\varrho)\in[0,1]$, we have
   \begin{eqnarray}
   \left\lVert P_1(P,\varkappa,\varrho)\right\rVert_2^2
   &=&
   \int_0^1\int_0^1\lvert P_1(P,\varkappa,\varrho)\vert^2 d\varkappa d\varrho\nonumber\\
      &=&
      \int_0^1\int_0^1\frac{\pi^2}{4^{2P}}\left\lvert\frac{\partial^{2P}}{\partial \alpha^{2P}}H_1(\varkappa,\varrho,\alpha)\right\rvert^2 d\varkappa d\varrho\nonumber\\
      &\leq&
      \frac{C_1^2\pi^2}{4^{2P}},\nonumber\\
   \end{eqnarray}
   And
    \begin{eqnarray}
   \left\lVert P_2(P,\varkappa,\varrho)\right\rVert_2^2
   &=&
   \int_0^1\int_0^1\lvert P_2(P^*,\varkappa,\varrho)\vert^2 d\varkappa d\varrho\nonumber\\
   &=&
   \int_0^1\int_0^1\frac{\pi^2}{4^{2P^*}}\left\lvert\frac{\partial^{2P^*}}{\partial \beta^{2P^*}}H_2(\varkappa,\varrho,\beta)\right\rvert^2 d\varkappa d\varrho\nonumber\\
   &\leq&
   \frac{C_2^2\pi^2}{4^{2P^*}},\nonumber\\
   \end{eqnarray}
   
   where,
   \begin{align*}
   C_1=max \left\{ \left\lvert\frac{\partial^{2P}}{\partial \gamma^{2P}}H_1(\varkappa,\varrho,\alpha)\right\rvert~0\leq \varkappa,\varrho,\gamma\leq 1 \right\}~and\\ C_2=max \left\{ \left\lvert\frac{\partial^{2P^*}}{\partial \beta^{2P^*}}H_2(\varkappa,\varrho,\beta)\right\rvert~0\leq \varkappa,\varrho\leq 1,~1\leq \beta\leq 2 \right\}.
   \end{align*}
   Now, define the residual function as:
   \begin{align*}
   Res_\mathcal{N}(\mathbb{U}(\varkappa,\varrho))=\sum_{s=1}^P w_s\rho(\sigma_s)\left(\frac{\partial\sigma_s \mathbb{U}(\varkappa,\varrho)}{\partial \varrho^{\sigma_s}}\right)_\mathcal{N}+(\mathbb{U}(\varkappa,\varrho))_\mathcal{N}-\mathcal{K^*}\sum_{r=1}^{P^*} w_r\rho(\sigma_r)\left(\frac{\partial\sigma_r \mathbb{U}(\varkappa,\varrho)}{\partial \varkappa^{\sigma_r}}\right)_\mathcal{N}\\
   -\left(\int_0^\varrho(\varrho-\xi)^{-\frac{1}{2}}\left[\frac{\partial^2\mathbb{U}(\varkappa,\xi)}{\partial \varkappa^2}\right]d\xi\right)_\mathcal{N}
   -(f(\varkappa,\varrho))_\mathcal{N}.
   \end{align*}
   By theorems \ref{Thm:5.01}, \ref{Thm:5.4}, it is evident that  
   \begin{align*}
   \| L_2(\mathbb{U}(\varkappa,\varrho))-Res_\mathcal{N}(\mathbb{U}(\varkappa,\varrho))\|_2&=\Bigg\|\sum_{s=1}^P w_s\rho(\sigma_s)\left(\frac{\partial^{\sigma_s}\mathbb{U}(\varkappa,\varrho)}{\partial \varrho^{\sigma_s}}-\left(\frac{\partial^{\sigma_s}\mathbb{U}(\varkappa,\varrho)}{\partial \varrho^{\sigma_s}}\right)_\mathcal{N}\right)\\
   &+(\mathbb{U}(\varkappa,\varrho)-(\mathbb{U}(\varkappa,\varrho))_\mathcal{N})\\
   &-\mathcal{K^*}\sum_{r=1}^{P^*} w_r\rho(\sigma_r)\left(\frac{\partial^{\sigma_r}\mathbb{U}(\varkappa,\varrho)}{\partial \varkappa^{\sigma_r}}-\left(\frac{\partial^{\sigma_r}\mathbb{U}(\varkappa,\varrho)}{\partial \varrho^{\sigma_r}}\right)_\mathcal{N}\right)\\
   &-\left(\int_0^\varrho(\varrho-\xi)^{-\frac{1}{2}}\left[\frac{\partial^2\mathbb{U}(\varkappa,\xi)}{\partial \varkappa^2}\right]d\xi-\left(\int_0^\varrho(\varrho-\xi)^{-\frac{1}{2}}\left[\frac{\partial^2\mathbb{U}(\varkappa,\xi)}{\partial \varkappa^2}\right]d\xi\right)_\mathcal{N}\right)\\
   &-(f(\varkappa,\varrho)-(f(\varkappa,\varrho))_\mathcal{N})\Bigg\|_2\\
   &\leq
   \Bigg\|\sum_{s=1}^P w_s\rho(\sigma_s)\left(\frac{\partial^{\sigma_s}\mathbb{U}(\varkappa,\varrho)}{\partial \varrho^{\sigma_s}}-\left(\frac{\partial^{\sigma_s}\mathbb{U}(\varkappa,\varrho)}{\partial \varrho^{\sigma_s}}\right)_\mathcal{N}\right)\Bigg\|_2\\
   &+\Bigg\|(\mathbb{U}(\varkappa,\varrho)-(\mathbb{U}(\varkappa,\varrho))_\mathcal{N})\Bigg\|_2\\
   &+|\mathcal{K^*}|\Bigg\|\sum_{r=1}^{P^*} w_r\rho(\sigma_r)\left(\frac{\partial^{\sigma_r}\mathbb{U}(\varkappa,\varrho)}{\partial \varkappa^{\sigma_r}}-\left(\frac{\partial^{\sigma_r}\mathbb{U}(\varkappa,\varrho)}{\partial \varrho^{\sigma_r}}\right)_\mathcal{N}\right)\Bigg\|_2\\
   &+\Bigg\|\int_0^\varrho(\varrho-\xi)^{-\frac{1}{2}}\left[\frac{\partial^2\mathbb{U}(\varkappa,\xi)}{\partial \varkappa^2}\right]d\xi-\left(\int_0^\varrho(\varrho-\xi)^{-\frac{1}{2}}\left[\frac{\partial^2\mathbb{U}(\varkappa,\xi)}{\partial \varkappa^2}\right]d\xi\right)_\mathcal{N}\Bigg\|_2\\
   &+\Bigg\|(f(\varkappa,\varrho)-(f(\varkappa,\varrho))_\mathcal{N})\Bigg\|_2
   \end{align*}
   \begin{align}
   \leq PM_1M_2&\left[\frac{9{\mathcal{B}_1}^2}{[\Gamma{(2-\alpha)}]^2}\sum_{\mathfrak{h}=2^{\mathfrak{R}-1}+1}^\infty\sum_{\mathfrak{g}=M}^\infty\sum_{\mathfrak{h}'=2^{\mathfrak{R}'-1}+1}^\infty\sum_{\mathfrak{g}'=\Lambda'}^\infty \frac{1}{256(\mathfrak{h}\mathfrak{h}')^5(2\mathfrak{g}-3)^4(2\mathfrak{g}'-3)^4}\right]^{\frac{1}{2}}\nonumber\\
   +
   &\left[{9{\mathcal{B}_0}^2}\sum_{\mathfrak{h}=2^{\mathfrak{R}-1}+1}^\infty\sum_{\mathfrak{g}=\Lambda}^\infty\sum_{\mathfrak{h}'=2^{\mathfrak{R}'-1}+1}^\infty\sum_{\mathfrak{g}'=\Lambda'}^\infty \frac{1}{256(\mathfrak{h}\mathfrak{h}')^5(2\mathfrak{g}-3)^4(2\mathfrak{g}'-3)^4}\right]^{\frac{1}{2}}\nonumber\\
   +
   &P^*M_3M_3\left[\frac{(2!)^2\times 9{\mathcal{B}_2}^2}{[\Gamma{(3-\beta)}]^2}\sum_{\mathfrak{h}=2^{K-1}+1}^\infty\sum_{\mathfrak{g}=\Lambda}^\infty\sum_{\mathfrak{h}'=2^{\mathfrak{R}'-1}+1}^\infty\sum_{\mathfrak{g}'=\Lambda'}^\infty \frac{1}{256(\mathfrak{h}\mathfrak{h}')^5(2\mathfrak{g}-3)^4(2\mathfrak{g}'-3)^4}\right]^{\frac{1}{2}}\nonumber\\
   +
   &\left[{4\times 9{\mathcal{B}_2}^2}\sum_{\mathfrak{h}=2^{\mathfrak{R}-1}+1}^\infty\sum_{\mathfrak{g}=\Lambda}^\infty\sum_{\mathfrak{h}'=2^{\mathfrak{R}'-1}+1}^\infty\sum_{\mathfrak{g}'=\Lambda'}^\infty \frac{1}{256(\mathfrak{h}\mathfrak{h}')^5(2\mathfrak{g}-3)^4(2\mathfrak{g}'-3)^4}\right]^{\frac{1}{2}}\nonumber\\
   +
   &\left[9{\mathcal{B}}^2\sum_{\mathfrak{h}=2^{\mathfrak{R}-1}+1}^\infty\sum_{\mathfrak{g}=\Lambda}^\infty\sum_{\mathfrak{h}'=2^{\mathfrak{R}'-1}+1}^\infty\sum_{\mathfrak{g}'=\Lambda'}^\infty \frac{1}{256(\mathfrak{h}\mathfrak{h}')^5(2\mathfrak{g}-3)^4(2\mathfrak{g}'-3)^4}\right]^{\frac{1}{2}}.
   \end{align}
   
   Where, $M_1=max \{\lvert w_s\rvert,s=1,2,\cdots P\}$ and 
   $M_2=max \{\lvert \rho(\sigma_s)\rvert ,s=1,2,\cdots P\}$ and\\
 $ ~~~~~$ $M_3=max \{\lvert w_r\rvert,r=1,2,\cdots P^*\}$ and 
   $M_4=max \{\lvert \rho(\sigma_r)\rvert ,r=1,2,\cdots P^*\}.$
   
  Finally, we have
     \begin{eqnarray}
   \lVert Res_\mathcal{N}(\mathbb{U}(\varkappa,\varrho))\rVert_2
                                        &=&
   \lVert 0-Res_\mathcal{N}(\mathbb{U}(\varkappa,\varrho))\rVert_2\nonumber\\
    &=&
    \lVert L_1(\mathbb{U}(\varkappa,\varrho))-Res_\mathcal{N}(\mathbb{U}(\varkappa,\varrho))\rVert_2\nonumber\\
    &\leq& 
    \lVert L_1(\mathbb{U}(\varkappa,\varrho))-L_2(\mathbb{U}(\varkappa,\varrho))\rVert_2+\lVert L_2(\mathbb{U}(\varkappa,\varrho))-Res_\mathcal{N}(\mathbb{U}(\varkappa,\varrho))\rVert_2\nonumber
   \end{eqnarray}
   \begin{align*}
   \leq
   \frac{C_1\pi}{4^P}+PM_1M_2\frac{3{\mathcal{B}_1}}{\Gamma{(2-\alpha)}}&\left[\sum_{\mathfrak{h}=2^{\mathfrak{R}-1}+1}^\infty\sum_{\mathfrak{g}=\Lambda}^\infty\sum_{\mathfrak{h}'=2^{\mathfrak{R}'-1}+1}^\infty\sum_{\mathfrak{g}'=\Lambda'}^\infty \frac{1}{256(\mathfrak{h}\mathfrak{h}')^5(2\mathfrak{g}-3)^4(2\mathfrak{g}'-3)^4}\right]^{\frac{1}{2}}\\
   +
   {3{\mathcal{B}_0}}&\left[\sum_{\mathfrak{h}=2^{\mathfrak{R}-1}+1}^\infty\sum_{\mathfrak{g}=\Lambda}^\infty\sum_{\mathfrak{h}'=2^{\mathfrak{R}'-1}+1}^\infty\sum_{\mathfrak{g}'=\Lambda'}^\infty \frac{1}{256(\mathfrak{h}\mathfrak{h}')^5(2\mathfrak{g}-3)^4(2\mathfrak{g}'-3)^4}\right]^{\frac{1}{2}}\\
   +
   \frac{C_2\pi}{4^{P^*}}+6PM_3M_4{{\mathcal{B}_2}}|\mathcal{K^*}|&\left[\sum_{\mathfrak{h}=2^{K-1}+1}^\infty\sum_{\mathfrak{g}=\Lambda}^\infty\sum_{\mathfrak{h}'=2^{\mathfrak{R}'-1}+1}^\infty\sum_{\mathfrak{g}'=\Lambda'}^\infty \frac{1}{256(\mathfrak{h}\mathfrak{h}')^5(2\mathfrak{g}-3)^4(2\mathfrak{g}'-3)^4}\right]^{\frac{1}{2}}\\
   +
   {6{\mathcal{B}_2}}&\left[\sum_{\mathfrak{h}=2^{\mathfrak{R}-1}+1}^\infty\sum_{\mathfrak{g}=\Lambda}^\infty\sum_{\mathfrak{h}'=2^{\mathfrak{R}'-1}+1}^\infty\sum_{\mathfrak{g}'=\Lambda'}^\infty \frac{1}{256(\mathfrak{h}\mathfrak{h}')^5(2\mathfrak{g}-3)^4(2\mathfrak{g}'-3)^4}\right]^{\frac{1}{2}}\\
   +
   {3\mathcal{B}}&\left[\sum_{\mathfrak{h}=2^{\mathfrak{R}-1}+1}^\infty\sum_{\mathfrak{g}=\Lambda}^\infty\sum_{\mathfrak{h}'=2^{\mathfrak{R}'-1}+1}^\infty\sum_{\mathfrak{g}'=\Lambda'}^\infty \frac{1}{256(\mathfrak{h}\mathfrak{h}')^5(2\mathfrak{g}-3)^4(2\mathfrak{g}'-3)^4}\right]^{\frac{1}{2}}
   \end{align*}
   \begin{align}
   \leq
   \frac{C_1\pi}{4^P}+\frac{C_2\pi}{4^{P^*}}+3\mathcal{K}\left[\frac{PM_1M_2}{\Gamma(2-\alpha)}+1+|\mathcal{K^*}|\frac{2P^*M_3M_4}{\Gamma(3-\beta)}+2+1 \right]\nonumber\\
   \left[\sum_{\mathfrak{h}=2^{\mathfrak{R}-1}+1}^\infty\sum_{\mathfrak{g}=\Lambda}^\infty\sum_{\mathfrak{h}'=2^{\mathfrak{R}'-1}+1}^\infty\sum_{\mathfrak{g}'=\Lambda'}^\infty \frac{1}{256(\mathfrak{h}\mathfrak{h}')^5(2\mathfrak{g}-3)^4(2\mathfrak{g}'-3)^4}\right]^{\frac{1}{2}},
   \end{align}
   
   where, $\mathcal{K}=max\{\mathcal{B},\mathcal{B}_0,\mathcal{B}_1,\mathcal{B}_2\}$.

\section{Error estimation}\label{sec:6}

In this section, we discuss about an error estimation for the DOT--SFWSIPDE. We achieve this by rewriting the equations.
\begin{align}\label{Eqn:6.1}
\int_{\alpha_1}^{\alpha_2}\rho(\alpha)\frac{\partial^{\alpha}\mathbb{U}(\varkappa,\varrho)}{\partial \varrho^{\alpha}}d\alpha+\mathbb{U}(\varkappa,\varrho)=\mathcal{K^*}\int_{\beta_1}^{\beta_2}\rho(\beta)\frac{\partial^{\beta}\mathbb{U}(\varkappa,\varrho)}{\partial \varkappa^{\beta}}d\beta&+\int_0^\varrho(\varrho-\xi)^{-\frac{1}{2}}\left[\frac{\partial^2\mathbb{U}(\varkappa,\xi)}{\partial \varkappa^2}\right]d\xi\nonumber\\
&+f(\varkappa,\varrho),
\end{align}
where, $\mathcal{K^*}$ is viscosity constant and $(\varkappa,\varrho)\in[0,1]\times[0,T], \alpha_1=0, \alpha_2=1, \beta_1=1,\beta_2=2.$

Equation \ref{Eqn:6.1} is acquired with the IC and BCs
\begin{align}\label{Eqn:6.2}
\mathbb{U}(\varkappa,0)=\nu(\varkappa) ,~~  0\textless \varkappa \textless 1,\nonumber\\
\mathbb{U}(0,\varrho)=\mathfrak{p_1}(\varrho),~~~0\textless \varrho\textless T,\nonumber\\
\mathbb{U}(1,\varrho)=\mathfrak{p_2}(\varrho), ~~0\textless \varrho\textless T.
\end{align}
Let $E_{\mathcal{N}}(\varkappa,\varrho)=\mathbb{U}(\varkappa,\varrho)-\mathbb{U}_{\mathcal{N}}(\varkappa,\varrho)$ denotes the error function, where $\mathbb{U}(\varkappa,\varrho)$, $\mathbb{U}_{\mathcal{N}}(\varkappa,\varrho)$ denotes the exact \& approximate solution, respectively, of equation \ref{Eqn:6.1}. Inserting the approximate solution into the equation \ref{Eqn:6.1} then we get
\begin{align}\label{Eqn:6.3}
\int_{\alpha_1}^{\alpha_2}\rho(\alpha)\frac{\partial^{\alpha}\mathbb{U}_\mathcal{N}(\varkappa,\varrho)}{\partial \varrho^{\alpha}}d\alpha+\mathbb{U}_\mathcal{N}(\varkappa,\varrho)=\mathcal{K^*}\int_{\beta_1}^{\beta_2}\rho(\beta)\frac{\partial^{\beta}\mathbb{U}_\mathcal{N}(\varkappa,\varrho)}{\partial \varkappa^{\beta}}d\beta&+\int_0^\varrho(\varrho-\xi)^{-\frac{1}{2}}\left[\frac{\partial^2\mathbb{U}_\mathcal{N}(\varkappa,\xi)}{\partial \varkappa^2}\right]d\xi\nonumber\\
&+f(\varkappa,\varrho)+\mathcal{R_N}(\varkappa,\varrho), 
\end{align}
where, $(\varkappa,\varrho)\in L^2(\Omega), \alpha_1=0, \alpha_2=1, \beta_1=1, \beta_2=2$ and $\Omega=[0,1]\times[0,T]$,

 with 
 \begin{align}\label{Eqn:6.4}
	(\mathbb{U}(\varkappa,0))_\mathcal{N}=(\nu(\varkappa))_\mathcal{N},\nonumber\\
	(\mathbb{U}(0,\varrho))_\mathcal{N}=(\mathfrak{p_1}(\varrho))_\mathcal{N},\nonumber\\
	(\mathbb{U}(1,\varrho))_\mathcal{N}=(\mathfrak{p_2}(\varrho))_\mathcal{N}.
	\end{align} 
Now we deduct equations \ref{Eqn:6.3} and \ref{Eqn:6.4} from equations \ref{Eqn:6.1} and \ref{Eqn:6.2}, respectively to get

\begin{align}\label{Eqn:6.5}
\int_{\alpha_1}^{\alpha_2}\rho(\alpha)\frac{\partial^{\alpha}\mathcal{E_N}(\varkappa,\varrho)}{\partial \varrho^{\alpha}}d\alpha+\mathfrak{E}_{\mathcal{N}}(\varkappa,\varrho)=\mathcal{K^*}\int_{\beta_1}^{\beta_2}\rho(\beta)\frac{\partial^{\beta}\mathfrak{E}_{\mathcal{N}}(\varkappa,\varrho)}{\partial \varkappa^{\beta}}d\beta&+\int_0^\varrho(\varrho-\xi)^{-\frac{1}{2}}\left[\frac{\partial^2\mathfrak{E}_{\mathcal{N}}(\varkappa,\xi)}{\partial \varkappa^2}\right]d\xi\nonumber\\
&-\mathcal{R_N}(\varkappa,\varrho), 
\end{align}
where, $(\varkappa,\varrho)\in L^2(\Omega), \alpha_1=0, \alpha_2=1, \beta_1=1, \beta_2=2$ and $\Omega=[0,1]\times[0,T]$,

 with
\begin{align}\label{Eqn:6.6}
	(\mathfrak{E}(\varkappa,0))_\mathcal{N}=0,\nonumber\\
	(\mathfrak{E}(0,\varrho))_\mathcal{N}=0,\nonumber\\
	(\mathfrak{E}(1,\varrho))_\mathcal{N}=0.
	\end{align}

Where, $\mathcal{R_N}(\varkappa,\varrho)$ represents the function of perturbation, which depends on $\mathbb{U}_{\mathcal{N}},  (\mathbb{U}_{\varkappa\varkappa})_\mathcal{N}$ and IC \& BCs. The above mentioned equations  \ref{Eqn:6.5} \& \ref{Eqn:6.6} can be solved for $\mathfrak{E}_{\mathcal{N}}\approx\Psi^T(\varrho)C_*\Psi(\varkappa)$ by using the mechanism, depicted in section \ref{sec:4.1} to find the value of vector $C_*$.\\

	Therefore, the maximum absolute error can be evaluated approximately by the following formula
\begin{equation}
	\mathfrak{E}_{\mathcal{N}}=max\{|\mathfrak{E}_{\mathcal{N}}(\varkappa,\varrho)|,0\leq\varkappa,\varrho<1\}.
	\end{equation}

\section{Numerical examples}\label{sec:7}


In this part two subsections are incorporated. Four test examples are considered, two for one dimensional case study, which are described in first subsection and two for two dimensional case study, which are described in second subsection. To ensure the method's robustness and utility the numerical results are taken in the form of Figures \& Table for various values of $\Lambda$ and $P$ using the presented method. Tables \ref{Tb:2}--\ref{Tb:5} offer numerical results for pointwise error,$L_2$-error, $L_\infty$-error and mean error, as well as used CPU time. Figures \ref{fig:1}-\ref{fig:8} show numerical results of approximate, exact solutions as well as absolute errors. Example \ref{Ex1}-\ref{Ex3} are examined at the ultimate time level, $\varrho=0.5$, whereas examples \ref{Ex4} is assessed at  $\varrho=1.0$ time level.

 The following formula's  for one \& two dimensional will be used
	\begin{eqnarray}
\lVert \mathbb{U}_{ex}-\mathbb{U}_\mathcal{N}\rVert= \left\{ \begin{array}{l}  \sqrt{\displaystyle\left(\sum_{i=0}^{N_p-1}h|\mathbb{U}_{ex}(\varkappa_i,\varrho)-\mathbb{U}_\mathcal{N}(\varkappa_i,\varrho)|\right)},~L_2-error
\vspace{0.4cm}\\ 
\max_{0\leq i\leq N_p-1}|\mathbb{U}_{ex}(\varkappa_i,\varrho)-\mathbb{U}_\mathcal{N}(\varkappa_i,\varrho)| ~,~~~~L_{\infty}-error
\vspace{0.4cm}\\
\frac{1}{N_p}\sum_{i=0}^{N_p-1}|\mathbb{U}_{ex}(\varkappa_i,\varrho)-\mathbb{U}_\mathcal{N}(\varkappa_i,\varrho)|,~~~~~mean~error
\end{array}\right.
\end{eqnarray} 

\begin{eqnarray}
\lVert \mathbb{U}_{ex}-\mathbb{U}_\mathcal{N}\rVert= \left\{ \begin{array}{l}  \sqrt{\displaystyle\left(\sum_{i=0}^{N_p-1}\sum_{j=0}^{N_p-1}h_\varkappa h_\eta|\mathbb{U}_{ex}(\varkappa_i,\eta_j,\varrho)-\mathbb{U}_\mathcal{N}(\varkappa_i,\eta_j,\varrho)|\right)},~L_2-error
\vspace{0.4cm}\\ 
\max_{0\leq i\leq N_p-1}\max_{0\leq j\leq N_p-1}|\mathbb{U}_{ex}(\varkappa_i,\eta_j,\varrho)-\mathbb{U}_\mathcal{N}(\varkappa_i,\eta_j,\varrho)| ~,~~~~L_{\infty}-error
\vspace{0.4cm}\\
\frac{1}{N_p}\sum_{i=0}^{N_p-1}\sum_{j=0}^{N_p-1}|\mathbb{U}_{ex}(\varkappa_i,\eta_j,\varrho)-\mathbb{U}_\mathcal{N}(\varkappa_i,\eta_j,\varrho)|,~~~~~~~~~~~mean~error
\end{array}\right.
\end{eqnarray}

\begin{Remark}
	In graphs, we denotes $E_1, E_2, E_3, E_4$ as errors corresponding to various values of fixed $\Lambda$ and variable P or vice-versa. For their identification, they are colored red, blue, green, and black, accordingly.
\end{Remark}

\begin{table}[H]
	\caption{Table of Notation}
	\label{Tb:1}
	\begin{center}
		\begin{tabular}{|c|c|}
			\hline
			General notation & Notation meaning \\   	
			\hline  
			$\mathcal{LW}$ & Legendre wavelet (LW)\\
			$\Lambda,\Lambda',\Lambda''$ & Number of basis element(upto degree $\Lambda-1,\Lambda'-1,\Lambda''-1$) \& $\Lambda=\mathfrak{g}+1$\\
			$P,P^*,P^{**}$ & Number of node points in LGQ for $\varrho, \varkappa, \eta $ direction vectors\\
			$\Lambda,\Lambda',\Lambda''$ & For time vector ($\varrho$) and space vectors ($\varkappa,\eta$)\\ 
			$h=\left(\frac{1}{N_p-1}\right)$ & Step size\\
			\hline
		
		\end{tabular}
	\end{center}  
\end{table}
\newpage
\subsection{Solving one dimensional time-space DOF integero-differential equation}
\begin{Example}\label{Ex1}
	We used the proposed approach on the following DOT-SFWSIPDE in one dimension
		\begin{equation*}
	\int_{0}^{1}\rho(\alpha)\frac{\partial^{\alpha}\mathbb{U}(\varkappa,\varrho)}{\partial \varrho^{\alpha}}d\alpha+\mathbb{U}(\varkappa,\varrho)=\mathcal{K^*}\int_{1}^{2}\rho(\beta)\frac{\partial^{\beta}\mathbb{U}(\varkappa,\varrho)}{\partial \varkappa^{\beta}}d\beta+\int_0^\varrho(\varrho-\xi)^{-\frac{1}{2}}\left[\frac{\partial^2\mathbb{U}(\varkappa,\xi)}{\partial \varkappa^2}\right]d\xi+f(\varkappa,\varrho),
	\end{equation*}
	
 	where, the source term   
	\begin{align*}
	f=\frac{2\varkappa^2\varrho(\varrho-1)}{log(\varrho)}+\varkappa^2\varrho^2-\mathcal{K^*}\frac{\varrho^2(\varkappa-1)}{log(\varkappa)}-\frac{32\varrho^{5/2}}{15},
	\end{align*}
	and $\mathbb{U}_{ex}=\varrho^2\varkappa^2$ be the exact solution,
	with the given IC $\mathbb{U}(\varkappa,0)=0$ and the BCs $\mathbb{U}(0,\varrho)=0,~\mathbb{U}(1,\varrho)=\varrho^2$. Here we take the distributed weight function as follows, $\rho(\alpha)=\Gamma{(3-\alpha)}$ and $\rho(\beta)=\frac{\Gamma{(3-\beta)}}{2}$. 
	
For various parameters $P$, $P*$, $\Lambda$, $\Lambda'$, $\mathfrak{R}$, and $\mathfrak{R}'$, numerical results are produced and displayed for this example. This example is solved for $\Lambda=\Lambda'=4, P=P^*=3,5,7,9$ and $\mathfrak{R}=\mathfrak{R}'=1$. 
\begin{itemize}
\item Table \ref{Tb:2} shows numerical results of pointwise errors, $L_2$-errors , $L_{\infty}$-errors and mean errors for $\Lambda=\Lambda'=4$, P=$P^*$=3,5,7,9, and $\mathfrak{R}=\mathfrak{R}'=1$. Here we take value of viscosity constant $\mathcal{K^*}$=1.

\item Figure \ref{fig:1} displays the results related to absolute errors for $\Lambda=\Lambda'=4$, $\mathfrak{R}=\mathfrak{R}'=1$ and $P=P^*=9$ of example \ref{Ex1}.  
\item Figure \ref{fig:2} displays the results related to absolute errors for $\Lambda=\Lambda'=4$, $\mathfrak{R}=\mathfrak{R}'=1$ and P=3,5,7,9 of example \ref{Ex1}. Here $E_1, E_2, E_3, E_4$ denotes the errors at $\varrho$=0.5 corresponds for fixed $\Lambda=\Lambda'=4$, $\mathfrak{R}=\mathfrak{R}'=1$ and P=$P^*$=3,5,7,9, respectively.
\item We labeled the graphs in Figure \ref{fig:2}, to check the nature of produced error. For this purpose, we multiply $E_1, E_2, E_3, E_4$ with  $10^{-9},10^{-6},10^{-6},1$, respectively.
\item Figure \ref{fig:03} refers to exact solution and Figure \ref{fig:04} show the approximate solution for $\Lambda=\Lambda'=4$, $\mathfrak{R}=\mathfrak{R}'=1$ and P=$P^*$=9 of example \ref{Ex1}.
	\end{itemize}

	\begin{table}[H]
	\caption{Results of errors and used CPU time of example \ref{Ex1} for $\mathcal{K^*}$=1, $h=0.05, \Lambda=\Lambda'=4, \mathfrak{R}=\mathfrak{R}'=1$.}
	\label{Tb:2}
	\begin{center}
		\begin{tabular}{|c|c|c|c|c|}
			\hline
			& P=$P^*$=3 & P=$P^*$=5 & P=$P^*$=7&P=$P^*$=9 \\   	
			\hline  
			
			$(\varkappa,{\varrho})$ & ${|\mathbb{U}_{ex}-\mathbb{U}_\mathcal{N}|_\mathcal{LW}}$ & ${|\mathbb{U}_{ex}-\mathbb{U}_\mathcal{N}|_\mathcal{LW}}$ & ${|\mathbb{U}_{ex}-\mathbb{U}_\mathcal{N}|_\mathcal{LW}}$&${|\mathbb{U}_{ex}-\mathbb{U}_\mathcal{N}|_\mathcal{LW}}$ \\
			\hline
			({0.0,0.0})&1.093E-40 & 1.877E-41&1.404E-41 & 3.411E-41\\
			({0.1,0.1})&7.880E-09 & 1.559E-11 & 1.672E-14 & 3.659E-19 \\
			({0.2,0.2})&2.955E-08 & 8.926E-11 & 8.655E-14 & 8.659E-19 \\
			({0.3,0.3})&1.575E-07 & 2.145E-10 & 1.973E-13 & 1.388E-17 \\
			({0.4,0.4})&3.551E-07 & 3.453E-10 & 3.052E-13 & 6.939E-17 \\
			({0.5,0.5})&5.463E-07 & 4.236E-10 & 3.609E-13 & 1.804E-16\\
			({0.6,0.6})&6.276E-07 & 4.031E-10 & 3.305E-13 & 3.608E-16 \\
			({0.7,0.7})&5.162E-07 & 2.717E-10 & 2.121E-13 & 4.718E-16 \\
			({0.8,0.8})&2.145E-07 & 7.118E-11 & 4.735E-14 & 4.441E-16\\
			({0.9,0.9})&1.054E-07 & 8.563E-11 & 7.161E-14 & 1.110E-16 \\
			({1.0,1.0})&1.016E-40 & 1.110E-16 &1.016E-40 & 7.772E-17 \\
			\hline
			$L_2$-error&1.762E-07 & 1.465E-10 & 1.289E-12 & 6.910E-17  \\
			\hline
			$L_{\infty}$-error&6.880E-07 & 5.870E-10 & 5.100E-13 &2.640E-16\\
			\hline
			mean error&3.471E-07 & 2.789E-10 & 2.404E-13 & 1.384E-16 \\
			\hline
			{CPU time(s)}&30.675 & 30.148 & 32.417 &35.396\\
			\hline
		\end{tabular}
	\end{center}  
\end{table}

\begin{figure}[H]
	\centering
	\includegraphics[width=0.7\linewidth]{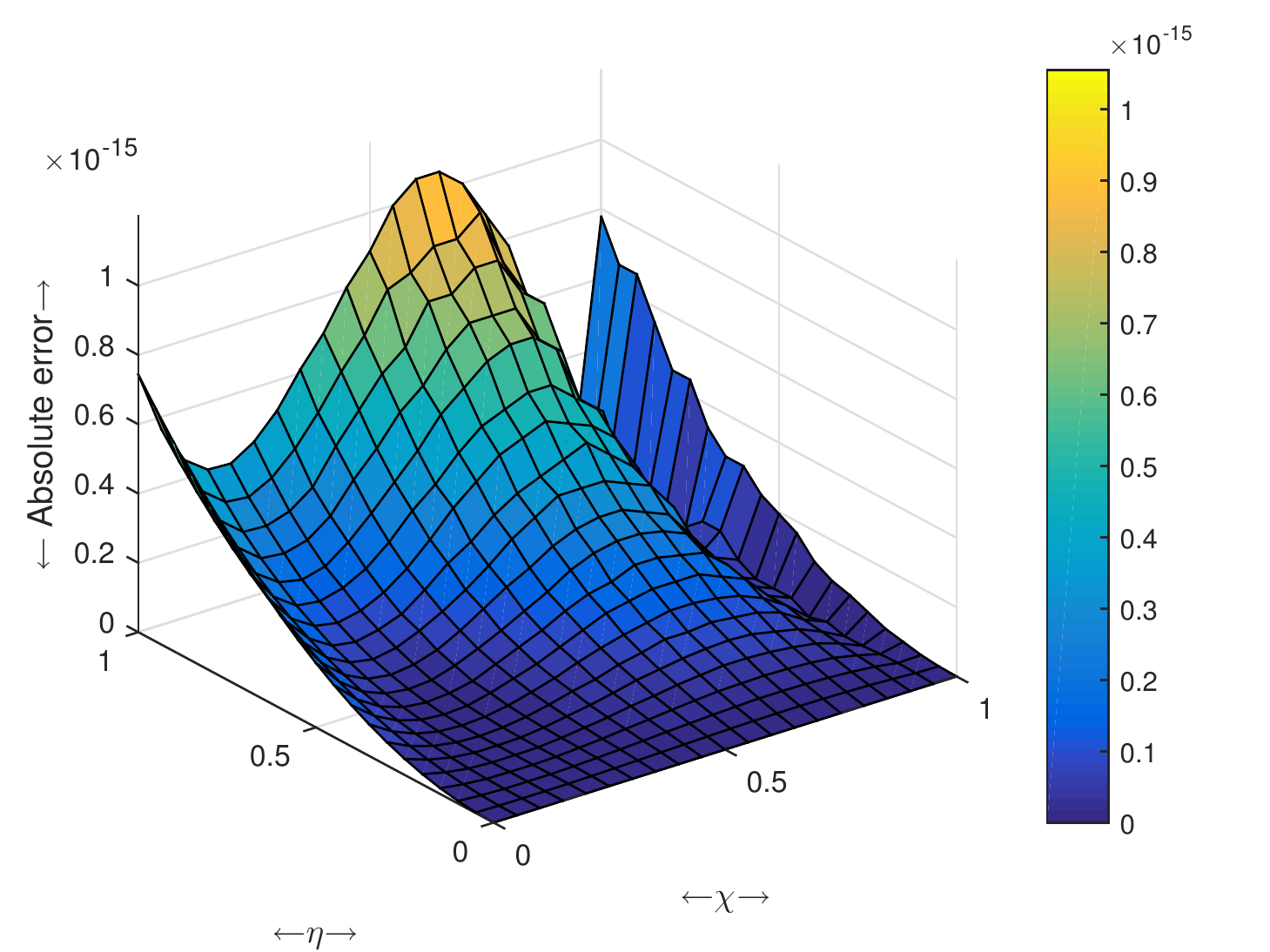}
	\caption{Error graph  for $\Lambda=\Lambda'=4, P=P^*=9, \mathfrak{R}=\mathfrak{R}'=1$ of example \ref{Ex1}.}
	\label{fig:1}
\end{figure}

\begin{figure}[H]
	\centering
	\includegraphics[width=0.7\linewidth]{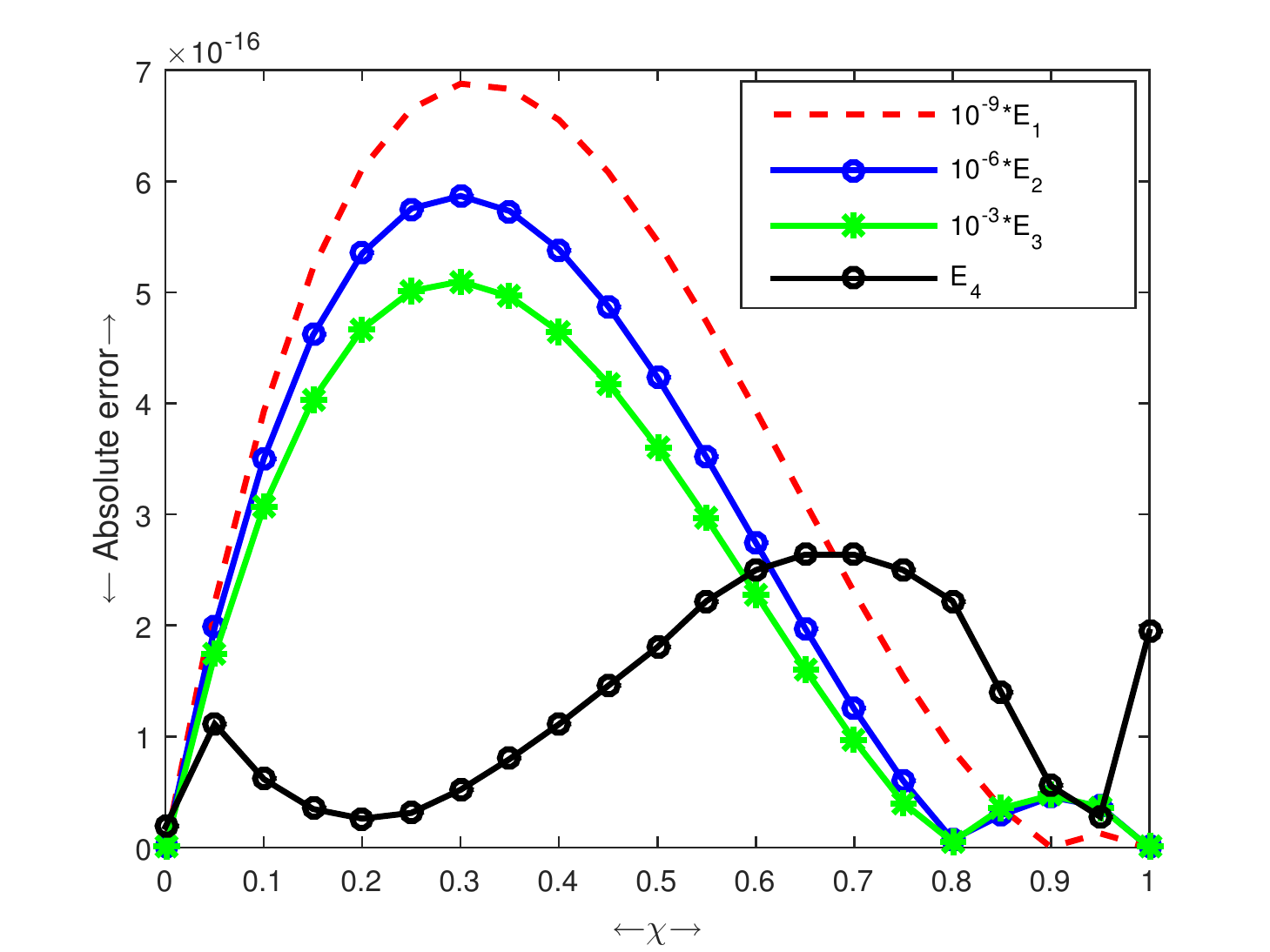}
	\caption{Error graph of  for $\Lambda=\Lambda'=4, P=P^*=3,5,7,9, \mathfrak{R}=\mathfrak{R}'=1$ and $\varrho$=0.5 example \ref{Ex1}.}
	\label{fig:2}
\end{figure}

\begin{figure}[H]
	\centering
	\includegraphics[width=0.7\linewidth]{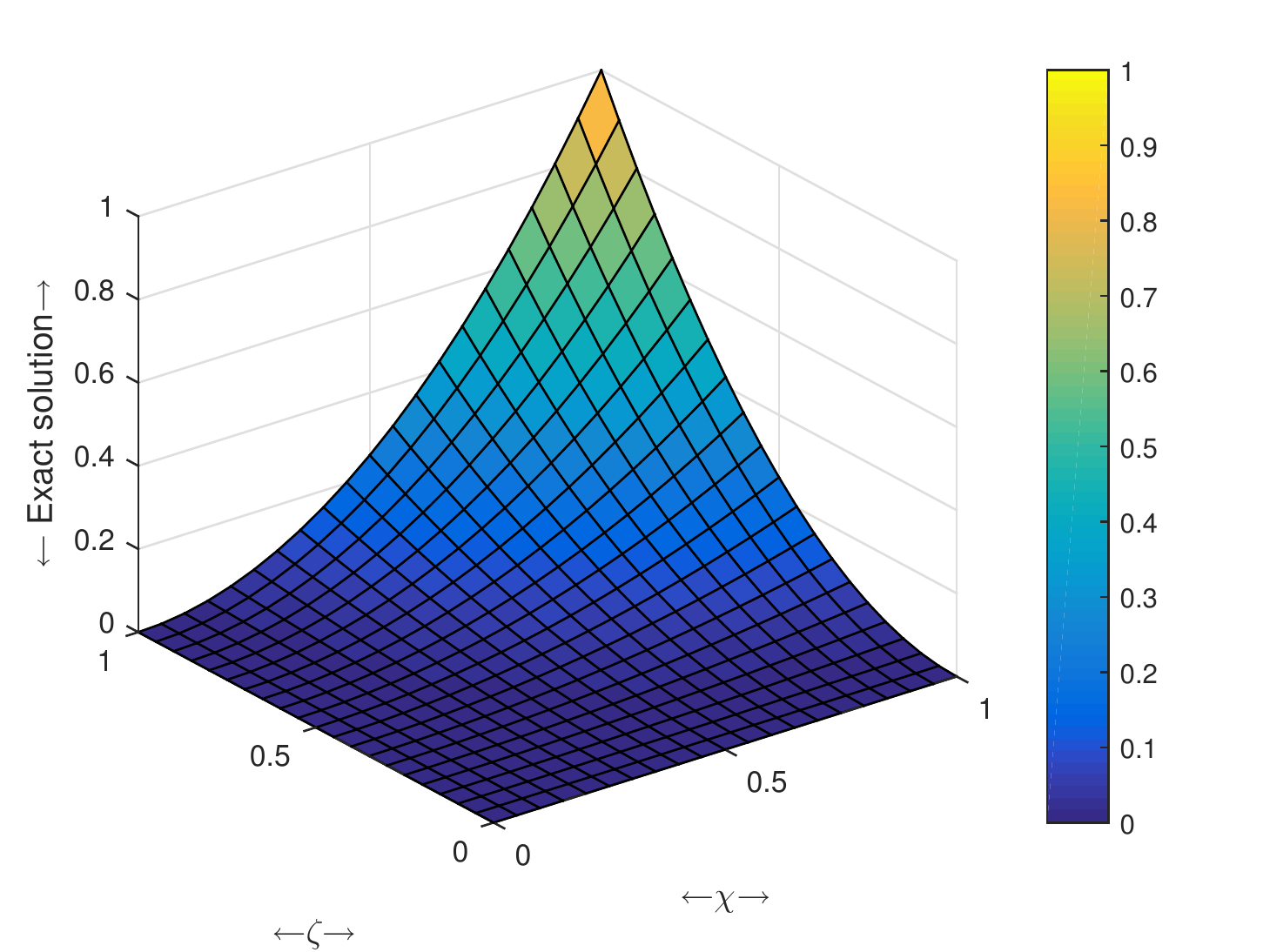}
	\caption{Exact solution graph  for $\Lambda=\Lambda'=4, P=P^*=9, \mathfrak{R}=\mathfrak{R}'=1$ and $\varrho$=0.5 of example \ref{Ex1}.}
	\label{fig:03}
\end{figure}

\begin{figure}[H]
	\centering
	\includegraphics[width=0.7\linewidth]{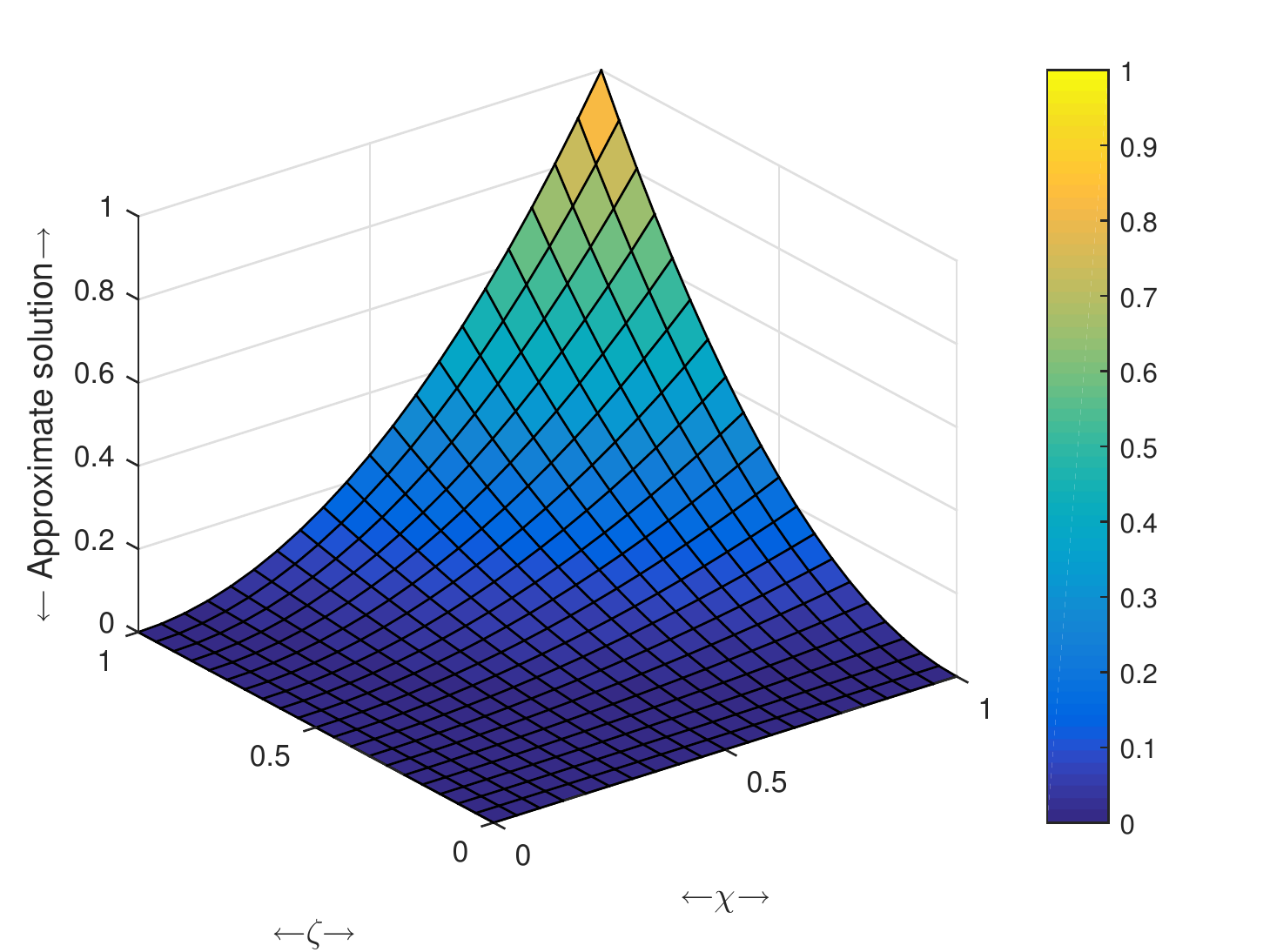}
	\caption{Approximate solution graph  for $\Lambda=\Lambda'=4, P=P^*=9, \mathfrak{R}=\mathfrak{R}'=1$ and $\varrho$=0.5 of example \ref{Ex1}.}
	\label{fig:04}
\end{figure}
\end{Example}

 \begin{Example}\label{Ex2}
 	Take the following DOT--SFWSIPDE 
 \begin{equation*}
 \int_{0}^{1}\rho(\alpha)D_\varrho^\alpha \mathbb{U}(\varkappa,\eta) d\alpha+\mathbb{U}(\varkappa,\varrho)=\mathcal{K^*}\int_{1}^{2}\rho(\beta)D_\varkappa^\beta \mathbb{U}(\varkappa,\eta)d\beta+\int_0^\varrho \frac{\mathbb{U}_{\varkappa\varkappa}}{(\varrho-\xi)^{\frac{1}{2}}}d\xi+f(\varkappa,\varrho)
 \end{equation*}
 	with IC $\mathbb{U}(\varkappa,0)=\varkappa^b$ and BCs $\mathbb{U}(0,\varrho)=\varrho^a$, $\mathbb{U}(1,\varrho)=1+\varrho^a,$ 
 	where, $f=\Gamma{(a+1)}\frac{(\varrho-1)(\varrho^{a-1})}{log(\varrho)}+(\varrho^a+\varkappa^b)-\mathcal{K^*}\Gamma{(b+1)}\frac{(\varkappa-1)\varkappa^{b-2}}{log(\varkappa)}-2b(b-1)\sqrt{\varrho}\varkappa^{b-2}$.  The exact solution for this example is $\mathbb{U}_{ex}=\varkappa^b+\varrho^a$ and parameters are  a=2, b=2. Distributed weight functions are $\rho(\alpha)=\Gamma(a+1-\alpha)$ and $\rho(\beta)=\Gamma(b+1-\beta)$.  	
 	\begin{itemize}
 		\item Table \ref{Tb:3} shown the numerical results of errors and used CPU time for P=$P^*$=9, $\Lambda=\Lambda'$=3,5,7,9 and $\mathfrak{R}=\mathfrak{R}'=1$.
 		\item Figure \ref{fig:3} displays the results of absolute error for fixed $P=P^*=9, \Lambda=\Lambda'=9$ and $\mathfrak{R}=\mathfrak{R}'=1$ of example \ref{Ex2}.
 		\item Figure \ref{fig:4} shows the results of absolute errors at time $\varrho$=0.5 for example \ref{Ex2}. In this figure $E_1,E_2,E_3,E_4$ correspond for fixed P=$P^*$=9, $\mathfrak{R}=\mathfrak{R}'=1$ and $\Lambda=\Lambda'$=3,5,7,9 respectively.
 		\item For labeling the graph in Figures \ref{fig:4} we multiply $E_1,E_2,E_3,E_4$ with  factors $0.3,2,1,1$, respectively.
 	
 	\end{itemize}
 	
   \begin{table}[H]
   	\caption{Results of errors with used CPU time for $\mathcal{K^*}$=1, $h=0.05, P=P^*=9, \mathfrak{R}=\mathfrak{R}'=1$ of example \ref{Ex1}}.
   	\label{Tb:3}
   	\begin{center}
   		\begin{tabular}{|c|c|c|c|c|}
   			\hline
   			& $\Lambda=\Lambda'=3$ & $\Lambda=\Lambda'=5$ & $\Lambda=\Lambda'=7$ & $\Lambda=\Lambda'=9$ \\   	
   			\hline  
   			
   			$(\varkappa,{\varrho})$ & ${|\mathbb{U}_{ex}-\mathbb{U}_\mathcal{N}|_\mathcal{LW}}$ & ${|\mathbb{U}_{ex}-\mathbb{U}_\mathcal{N}|_\mathcal{LW}}$ & ${|\mathbb{U}_{ex}-\mathbb{U}_\mathcal{N}|_\mathcal{LW}}$&${|\mathbb{U}_{ex}-\mathbb{U}_\mathcal{N}|_\mathcal{LW}}$ \\
   			\hline
   			({0.0,0.0})&2.322E-40 & 5.632E-41 &1.549E-41 & 8.185E-41\\
   			({0.1,0.1})&5.898E-15 & 1.818E-15 & 4.996E-16 & 1.079E-15 \\
   			({0.2,0.2})&1.929E-14 & 1.138E-15 & 3.886E-16 & 4.580E-16 \\
   			({0.3,0.3})&3.497E-14 & 8.327E-16 & 2.776E-17 & 4.441E-16 \\
   			({0.4,0.4})&4.540E-14 & 3.331E-15 &9.992E-16 & 2.776E-16 \\
   			({0.5,0.5})&6.051E-14 & 8.882E-15 & 3.331E-16 & 7.772E-16\\
   			({0.6,0.6})&6.772E-14 & 1.110E-15 &1.110E-16 & 4.441E-16 \\
   			({0.7,0.7})&7.327E-14 & 1.776E-14 & 4.441E-16 & 2.220E-16 \\
   			({0.8,0.8})&7.994E-15 & 3.111E-15 & 8.882E-16 & 2.220E-16\\
   			({0.9,0.9})&9.104E-14 & 4.441E-15 & 1.110E-15 & 1.663E-17 \\
   			({1.0,1.0})&1.132E-14 & 2.210E-40 &1.036E-40 & 3.742E-41 \\
   			\hline
   			$L_2$-error&2.794E-14 & 1.889E-15 & 5.030E-16 & 1.115E-16  \\
   			\hline
   			$L_{\infty}$-error&6.050E-14 & 8.880E-15& 2.227E-15 &6.110E-16\\
   			\hline
   			mean error&4.879E-14 & 4.097E-15 & 3.332E-16 & 1.929E-16 \\
   			\hline
   			{CPU time(s)}&18.417 & 56.719 & 164.539 &267.399\\
   			\hline
   		\end{tabular}
   	\end{center}  
   \end{table}

   \begin{figure}[H]
   	\centering
   	\includegraphics[width=0.7\linewidth]{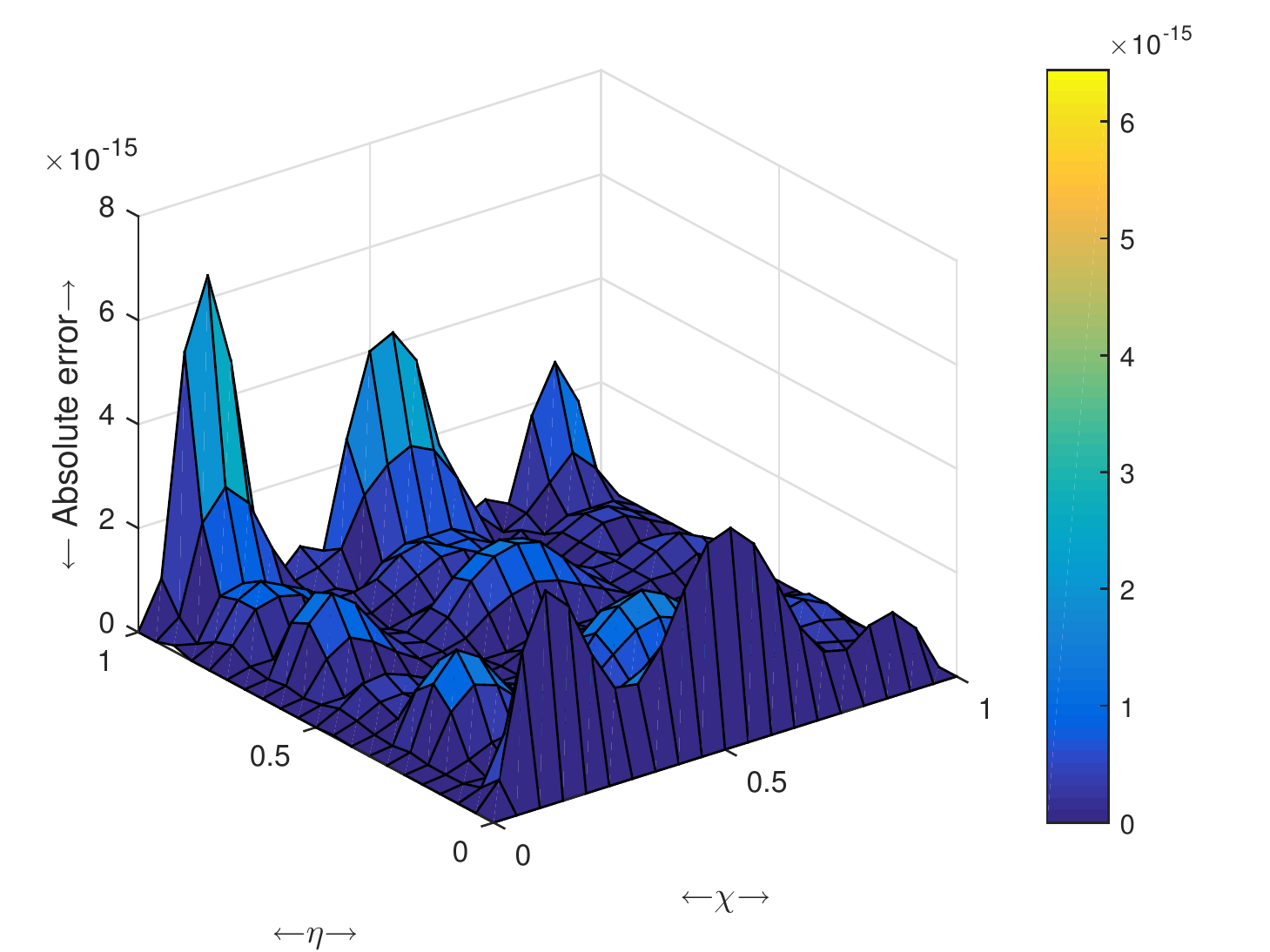}
   	\caption{Error graph  for $\Lambda=\Lambda'=9, P=P^*=9, \mathfrak{R}=\mathfrak{R}'=1$ of example \ref{Ex2}.}
   	\label{fig:3}
   \end{figure}
   \begin{figure}[H]
   	\centering
   	\includegraphics[width=0.7\linewidth]{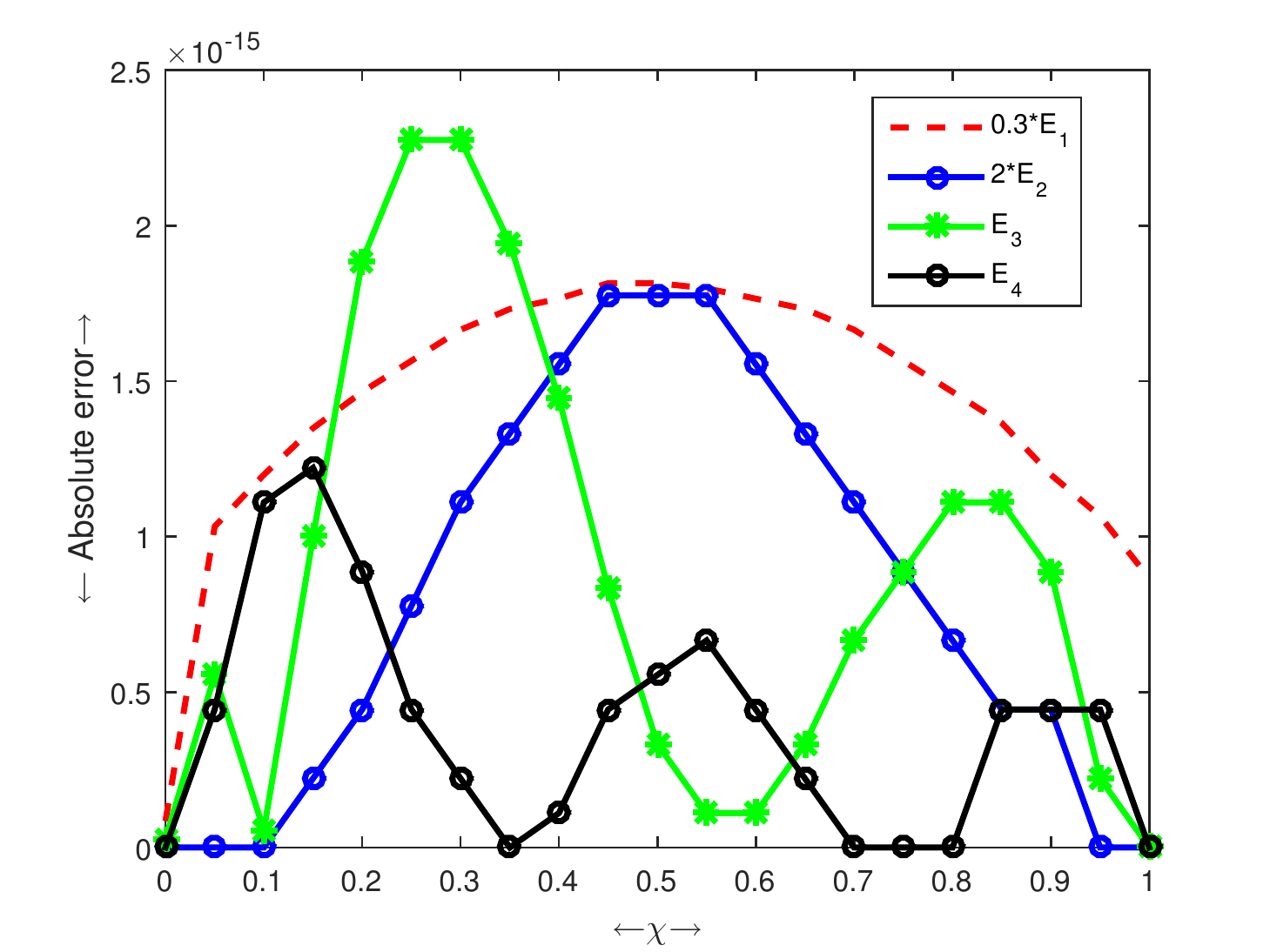}
   	\caption{Error graph for $\Lambda=\Lambda'=3,5,7,9, P=P^*=9, \mathfrak{R}=\mathfrak{R}'=1, \varrho=0.5$ of example \ref{Ex2}.}
   	\label{fig:4}
   \end{figure} 
 	\end{Example}
  
 \subsection{Two dimensional time-space DOF integro-differential equation}
  \begin{Example}\label{Ex3}
  	We consider the following DOT--SFWSIPDE with nonhomogeneous boundary conditions
  	\begin{eqnarray}
  	\int_{\alpha_1}^{\alpha_2}\rho(\alpha)\frac{\partial^{\alpha}\mathbb{U}(\varkappa,\eta,\varrho)}{\partial \varrho^{\alpha}}d\alpha+\mathbb{U}(\varkappa,\eta,\varrho)=\mathcal{K^*}\int_{\beta_1}^{\beta_2}\rho(\beta)\left[\frac{\partial^{\beta}\mathbb{U}(\varkappa,\eta,\varrho)}{\partial \varkappa^{\beta}}+\frac{\partial^{\beta}\mathbb{U}(\varkappa,\eta,\varrho)}{\partial \eta^{\beta}}\right]d\beta\nonumber\\
  	+\int_0^\varrho(\varrho-\xi)^{-\frac{1}{2}}\left[\frac{\partial^2\mathbb{U}(\varkappa,\eta,\varrho)}{\partial \varkappa^2}+\frac{\partial^2\mathbb{U}(\varkappa,\eta,\varrho)}{\partial \eta^2}\right]d\xi+f(\varkappa,\eta,\varrho),\nonumber
  	\end{eqnarray}
      with the IC $\mathbb{U}(\varkappa,\eta,0)=a\varkappa^2+b\eta^2$, and the BCs
      \begin{align*}
    \mathbb{U}(0,\eta,\varrho)=b\eta^2+c\varrho^2,~~~~~~\mathbb{U}(1,\eta,\varrho)=a+b\eta^2+c\varrho^2,\\
    \mathbb{U}(\varkappa,0,\varrho)=a\varkappa^2+c\varrho^2,~~~~~~\mathbb{U}(\varkappa,1,\varrho)=a\varkappa^2+b+c\varrho^2. 
    \end{align*}
    The source function, $f=\frac{2c\varrho(\varrho-1)}{log(\varrho)}(a\varkappa^2+b\eta^2+c\varrho^2)-\mathcal{K^*}\left(\frac{2a(\varkappa-1)}{log(\varkappa)}+\frac{2b(\eta-1)}{log(\eta)}\right)-4(a+b)\sqrt(\varrho)$ 
    and exact solution is  $\mathbb{U}_{ex}=a\varkappa^2+b\eta^2+c\varrho^2$.The values of parameters are a=1,b=1,c=1 and the value of distributed weight functions are $\rho(\alpha)=\Gamma{(3-\alpha)}$ and $\rho(\beta)=\Gamma{(3-\beta)}$.  
    
   \begin{itemize}
   	\item Table \ref{Tb:4} shows the results of pointwise errors, $L_2$-errors, $L_{\infty}$-errors, mean errors for $\Lambda=\Lambda'=\Lambda''=4$, $\mathfrak{R}=\mathfrak{R}'=\mathfrak{R}''=1$, and P=$P^*$=$P^{**}$=5,7,9,11 of example \ref{Ex3}.
   	\item Figure \ref{fig:5} describes the results of absolute errors for $\Lambda=\Lambda'=\Lambda''=4$, $\mathfrak{R}=\mathfrak{R}'=\mathfrak{R}''=1$ and P=$P^*$=$P^{**}$=11.
   	\item Figures \ref{fig:6} and \ref{fig:7} are referred to exact and approximate solution graph, respectively.   	
   	\end{itemize}
    
   \begin{table}[H]
   	\caption{Results of errors with used CPU time  for $\mathcal{K^*}$=1, $h_\varkappa$=0.05, $h_\eta$=0.05, $\Lambda=\Lambda'=\Lambda''=4, \mathfrak{R}=\mathfrak{R}'=\mathfrak{R}''=1$ of example \ref{Ex3}.}
   	\label{Tb:4}
   	\begin{center}
   		\begin{tabular}{|c|c|c|c|c|}
   			\hline
   			& $P$=$P^*=P^{**}$=5 & $P$=$P^*=P^{**}$=7& $P$=$P^*=P^{**}$=9&$P$=$P^*=P^{**}$=11\\   	
   			\hline  
   			
   			$(\varkappa,{\varrho})$ & ${|\mathbb{U}_{ex}-\mathbb{U}_\mathcal{N}|_\mathcal{LW}}$ & ${|\mathbb{U}_{ex}-\mathbb{U}_\mathcal{N}|_\mathcal{LW}}$ & ${|\mathbb{U}_{ex}-\mathbb{U}_\mathcal{N}|_\mathcal{LW}}$&${|\mathbb{U}_{ex}-\mathbb{U}_\mathcal{N}|_\mathcal{LW}}$ \\
   			\hline
   			({0.0,0.0})&5.551E-17 & 5.551E-17 &7.772E-15 & 2.770E-17\\
   			({0.1,0.1})&3.214E-09 & 2.838E-12 & 1.021E-14 & 1.110E-16 \\
   			({0.2,0.2})&6.840E-09 & 6.036E-12 & 1.299E-14 & 2.220E-16 \\
   			({0.3,0.3})&7.356E-09 & 6.484E-12 & 1.332E-14 & 2.221E-16 \\
   			({0.4,0.4})&5.473E-09 & 4.810E-12 &1.188E-14 & 1.110E-16 \\
   			({0.5,0.5})&3.159E-09 & 2.756E-12 & 1.010E-14 & 1.111E-16\\
   			({0.6,0.6})&1.773E-09 & 1.528E-12 &9.104E-15 & 4.301E-17 \\
   			({0.7,0.7})&1.314E-09 & 1.128E-12 & 8.660E-15 & 3.221E-17 \\
   			({0.8,0.8})&7.592E-10 & 6.530E-13 & 8.438E-15 & 7.115E-17\\
   			({0.9,0.9})&4.800E-10 & 4.243E-13 & 7.55E-15 & 2.221E-16 \\
   			({1.0,1.0})&1.132E-40 & 3.410E-40 &7.550E-15 & 3.742E-41 \\
   			\hline
   			$L_2$-error&3.2961E-09 & 9.516E-12& 9.948E-15 & 1.626E-16  \\
   			\hline
   			$L_{\infty}$-error&1.0747E-08 & 2.906E-12& 1.598E-14 &4.440E-16\\
   			\hline
   			mean error&3.2043E-09 & 2.8151E-12 & 1.017E-14 & 2.430E-16 \\
   			\hline
   			{CPU time(s)}&142.440 & 148.064 & 148.916 &152.864\\
   			\hline
   		\end{tabular}
   	\end{center}  
   \end{table}
      
   \begin{figure}[H]
   	\centering
   	\includegraphics[width=0.7\linewidth]{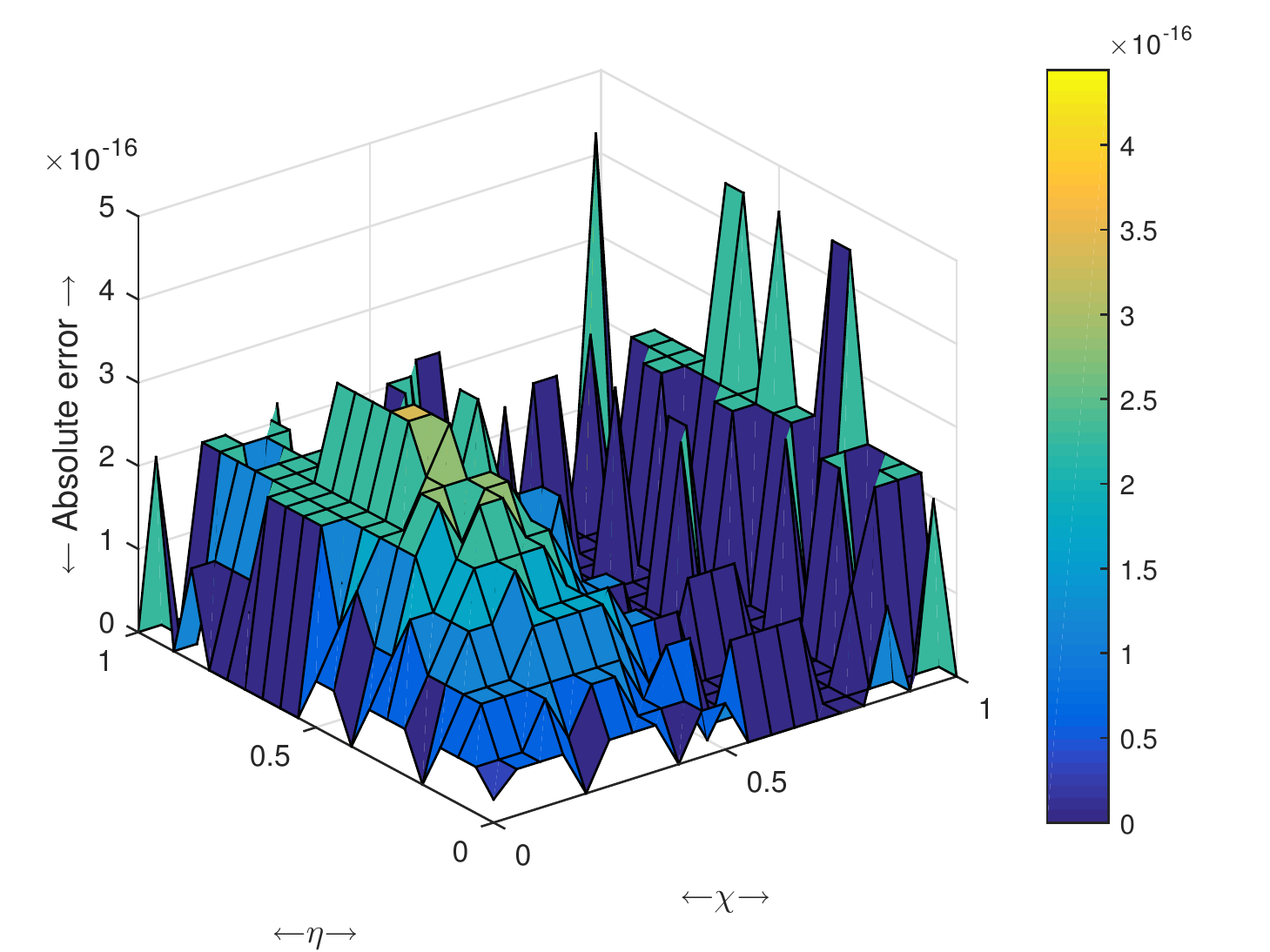}
   	\caption{Error graph for $\Lambda=\Lambda'=\Lambda''=4$, $P=P^*=P^{**}=11$, $\mathfrak{R}=\mathfrak{R}'=\mathfrak{R}''=1$  of example \ref{Ex3}.}
   	\label{fig:5}
   \end{figure}
   \begin{figure}[H]
   	\centering
   	\includegraphics[width=0.7\linewidth]{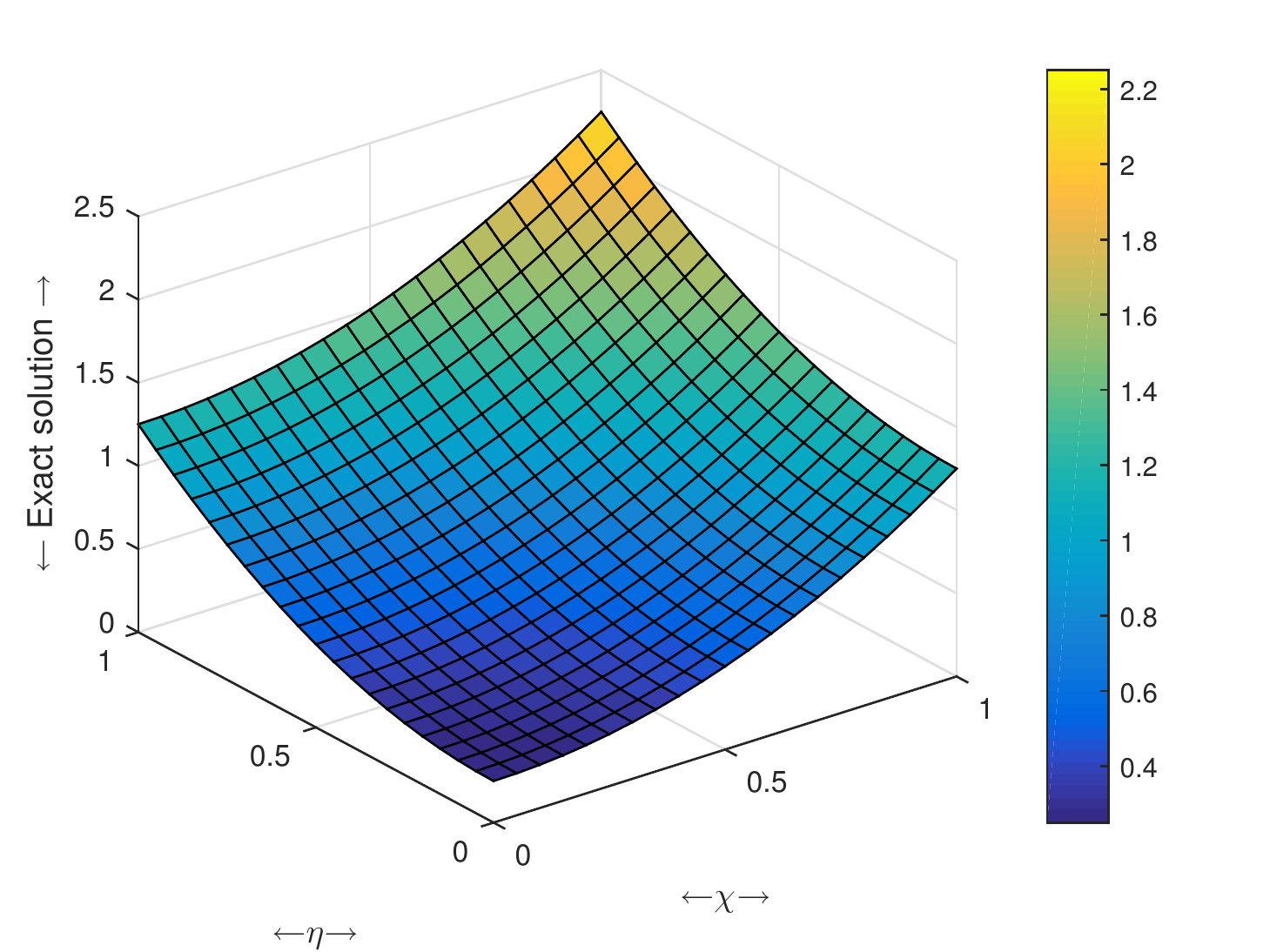}
   	\caption{Exact solution graph  for $\varrho$=0.5 of example \ref{Ex3}}
   	\label{fig:6}
   \end{figure}
   \begin{figure}[H]
   	\centering
   	\includegraphics[width=0.7\linewidth]{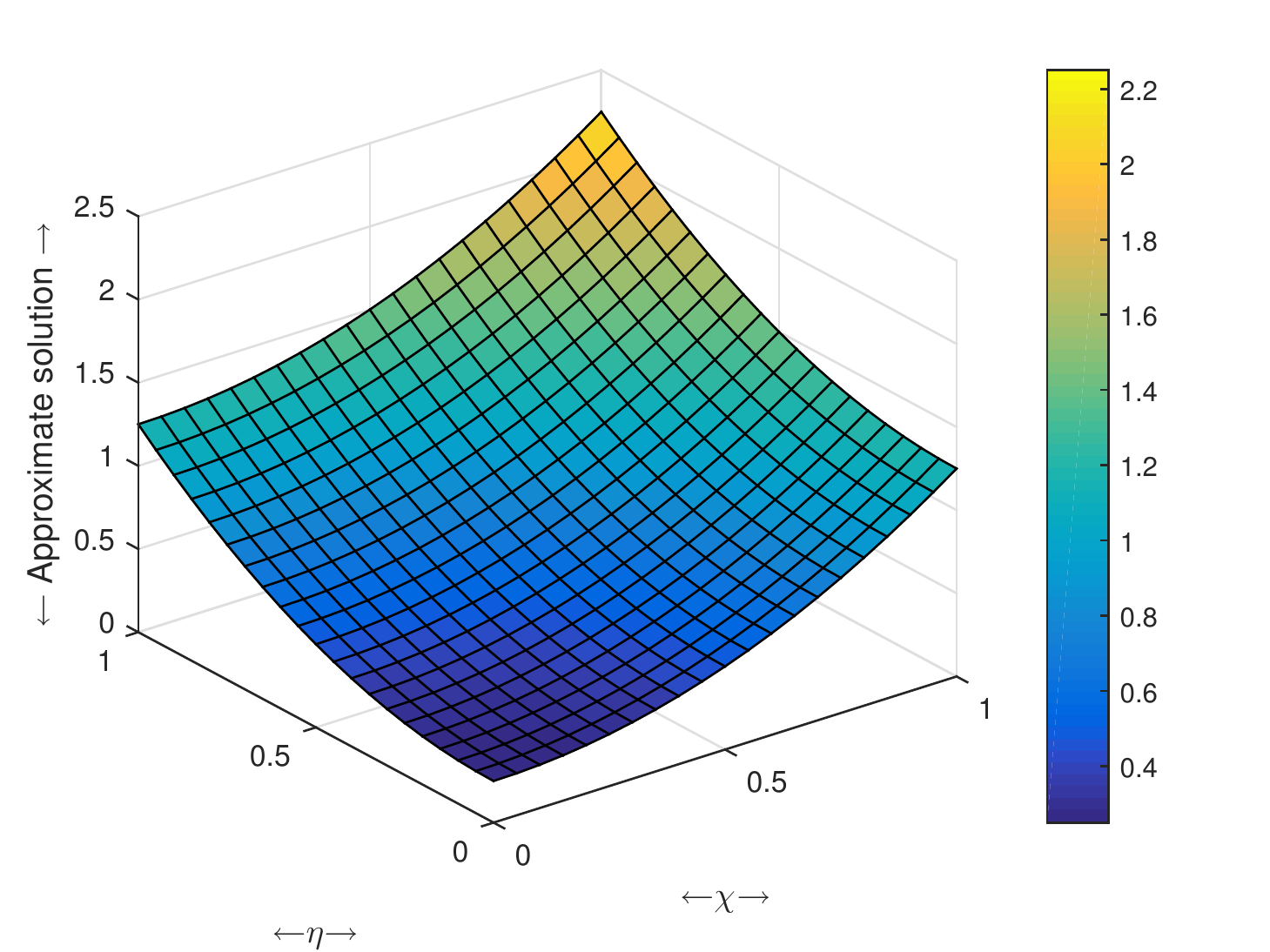}
   	\caption{Approximate solution graph  for $\Lambda=\Lambda'=\Lambda''=5, P=P^*=P^{**}=4, \mathfrak{R}=\mathfrak{R}'=\mathfrak{R}''=1$ and $\varrho$=0.5 of example \ref{Ex3}.}
   	\label{fig:7}
   \end{figure}
   
\end{Example}

\begin{Example}\label{Ex4}
We consider the following DOT-SFWSIPDE in 2D
	\begin{align*}
	\int_{0}^{1}\rho(\alpha)D_\varrho^\alpha \mathbb{U}(\varkappa,\eta,\varrho) d\alpha+\mathbb{U}(\varkappa,\eta,\varrho)=\mathcal{K^*}\int_{1}^{2}\rho(\beta)[D_\varkappa^\beta \mathbb{U}(\varkappa,\eta,\varrho)+D_\eta^\beta \mathbb{U}(\varkappa,\eta,\varrho)]d\beta\\
	+\int_0^\varrho \frac{\mathbb{U}_{\varkappa\varkappa}(\varkappa,\eta,\varrho)+\mathbb{U}_{\eta\eta}(\varkappa,\eta,\varrho)}{(\varrho-\xi)^{\frac{1}{2}}}d\xi+f(\varkappa,\eta,\varrho),
	\end{align*}
	where, 
	\begin{equation*}
	f(\varkappa,\eta,\varrho)=a(\varkappa+\eta)\frac{(\varrho-1)}{log(\varrho)}+(\varkappa\eta+\eta\varrho+\varrho\varkappa)-b(\eta+\varrho)\frac{(\varkappa-1)}{\varkappa log(\varkappa)}-b(\varkappa+\varrho)\frac{(\eta-1)}{\eta log(\eta)}. 
	\end{equation*}
	Furthermore, the equation is facilitate with the IC $\mathbb{U}(\varkappa,\eta,0)=\varkappa\eta$ and BCs \begin{align*}
	\mathbb{U}(0,\eta,\varrho)=\eta\varrho,~~~~~~\mathbb{U}(1,\eta,\varrho)=\eta+\varrho+\eta\varrho,\\
	\mathbb{U}(\varkappa,0,\varrho)=\varrho\varkappa,~~~~~~\mathbb{U}(\varkappa,1,\varrho)=\varkappa+\varrho+\varkappa\varrho. 
	\end{align*}
	 The exact solution is $\mathbb{U}_{ex}=\varkappa\eta+\eta\varrho+\varrho\varkappa$. Here we take distributed weight functions as $\rho(\alpha)=a\Gamma(2-\alpha)$ and $\rho(\beta)=b\Gamma{(2-\beta)}$. Numerical results are obtained and depicted for $\Lambda=\Lambda'=\Lambda''=3, P=P^*=P^{**}=4,6,8,10, \mathfrak{R}=\mathfrak{R}'=\mathfrak{R}''=1$. 
	\begin{itemize}
		\item Table \ref{Tb:5} shows the numerical results of errors for $\Lambda=\Lambda'=\Lambda''=3$, $\mathfrak{R}=\mathfrak{R}'=\mathfrak{R}''=1$ and P=$P^*=P^{**}$=4,6,8,10.
		\item The graph of absolute errors for $\Lambda=\Lambda'=\Lambda''=3$, $\mathfrak{R}=\mathfrak{R}'=\mathfrak{R}''=1$ and P=$P^*=P^{**}$=8 is shown in Figure \ref{fig:8}.
	\end{itemize}
	\begin{table}[H]
		\caption{Results of erros with used CPU time for $\mathcal{K^*}$=1, $h_\varkappa$=0.05, $h_\eta$=0.05, $\Lambda=\Lambda'=\Lambda''=3, \mathfrak{R}=\mathfrak{R}'=\mathfrak{R}''=1$ of example \ref{Ex4}.}
		\label{Tb:5}
		\begin{center}
			\begin{tabular}{|c|c|c|c|c|}
				\hline
				& $P$=4,$a$=1,$b$=0.5 & $P$=6,$a$=0.5,$b$=1 & $P$=8,$a$=2,$b$=1.5 &$P$=10,$a$=1.5,$b$=2\\   	
				\hline  
				
				$(\varkappa,{\varrho})$ & ${|\mathbb{U}_{ex}-\mathbb{U}_\mathcal{N}|_\mathcal{LW}}$ & ${|\mathbb{U}_{ex}-\mathbb{U}_\mathcal{N}|_\mathcal{LW}}$ & ${|\mathbb{U}_{ex}-\mathbb{U}_\mathcal{N}|_\mathcal{LW}}$&${|\mathbb{U}_{ex}-\mathbb{U}_\mathcal{N}|_\mathcal{LW}}$ \\
				\hline
				({0.0,0.0})&4.471E-15 & 4.700E-16 &3.727E-17 & 2.282E-17\\
				({0.1,0.1})&5.551E-15 & 5.154E-16 & 2.776E-17 & 2.7345E-17 \\
				({0.2,0.2})&5.251E-15 & 3.245E-16 & 2.540E-17 & 2.568E-18 \\
				({0.3,0.3})&1.110E-14 & 1.658E-16 & 2.115E-17 & 5.235E-18 \\
				({0.4,0.4})&1.111E-14 & 2.124E-16 &1.758E-17 & 5.267E-18 \\
				({0.5,0.5})&1.234E-15 & 2.554E-16 & 5.667E-17 & 5.246E-18\\
				({0.6,0.6})&1.236E-15 & 2.650E-16 &9.457E-17 & 4.324E-18 \\
				({0.7,0.7})&2.220E-14 & 1.254E-16 & 5.325E-17 & 4.821E-17 \\
				({0.8,0.8})&1.348E-15 & 4.441E-16 & 4.257E-17 & 4.857E-16\\
				({0.9,0.9})&1.257E-15 & 8.335E-16 & 3.448E-17 & 2.325E-17 \\
				({1.0,1.0})&1.532E-40 & 2.412E-40 &3.550E-40 & 7.732E-41 \\
				\hline
				$L_2$-error&8.108E-15 & 7.876E-16& 9.788E-17 & 3.792E-17  \\
				\hline
				$L_{\infty}$-error&4.440E-14 & 4.440E-16& 4.441E-16 &7.795E-17\\
				\hline
				mean error&3.700E-15 & 2.114E-16& 3.700E-17 & 2.231E-17 \\
				\hline
				{CPU time(s)}&68.884 & 71.312 & 72.934 &81.361\\
				\hline
			\end{tabular}
		\end{center}  
	\end{table}

\begin{figure}[H]
	\centering
	\includegraphics[width=0.7\linewidth]{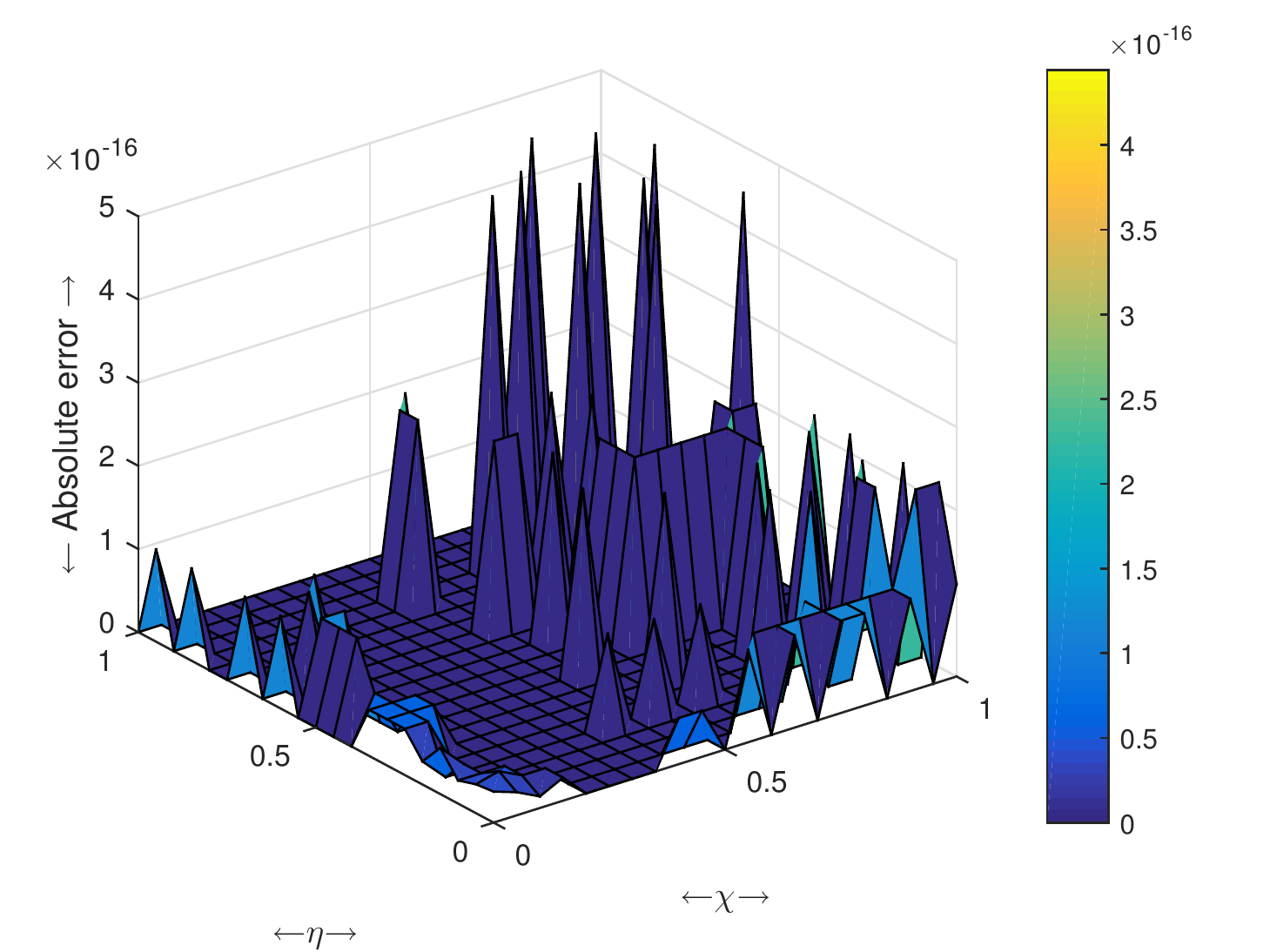}
	\caption{Error graph  for $\Lambda=\Lambda'=\Lambda''=3, P=8, \mathfrak{R}=\mathfrak{R}'=\mathfrak{R}''=1$ of example \ref{Ex4}.}
	\label{fig:8}
\end{figure}

\end{Example}

\section{Conclusion}\label{sec:8}
In this manuscript, a robust numerical method based on standard tau approach and collocation technique has been constructed for solving multi-dimensional DOT-SFWSIPDE. To this end, the original defined problem is converted into a system of linear algebraic equations using a variety of operational matrices, the standard tau technique, and the LGQ rule. Four test examples are performed for testing of the proposed method. To determined the method's efficiency and accuracy pointwise errors, $L_{2}$-errors, $L_{\infty}$-errors are evaluated and summed up in Tables \ref{Tb:2}-\ref{Tb:5} and Figures \ref{fig:1}-\ref{fig:8}. The used CPU time is also evaluated. The produced numerical results  conclude that the method has the capability of providing high accuracy with minimal computing efforts. It is also clear from the error Tables \ref{Tb:2}-\ref{Tb:5} that error is minimized when we continuously increase basis functions or increase the nodes points in LGQ.  Hence it is observed that the proposed method is proved to be an effective numerical procedure in terms of accuracy \& computational cost for handling  multi-dimensional DOT-SFWSIPDE. The method is also applicable on nonlinear  time-space DOFDEs, which is one of our goals for future study. In this article, we have provided the convergence analysis for 1D case only, the analysis for 2D case is a task for future study. \\


\textbf{\Large Acknowledgment}\label{sec:9}\\

The first author acknowledges the financial support from council of scientific \& industrial research (CSIR), India, under senior research fellow (SRF) scheme. The second author acknowledges the financial support from national board for higher mathematics, department of atomic energy, India, with sanctioned order no. 0204/17/2019/R \& D-II/9722.     
 
\bibliographystyle{elsarticle-num}
\bibliography{Distributed_integero1}

\begin{thebibliography}{10}
\expandafter\ifx\csname url\endcsname\relax
  \def\url#1{\texttt{#1}}\fi
\expandafter\ifx\csname urlprefix\endcsname\relax\def\urlprefix{URL }\fi
\expandafter\ifx\csname href\endcsname\relax
  \def\href#1#2{#2} \def\path#1{#1}\fi

\bibitem{podlubny1998fractional}
I.~Podlubny, Fractional differential equations: an introduction to fractional
  derivatives, fractional differential equations, to methods of their solution
  and some of their applications, Elsevier, 1998.

\bibitem{bagley1985fractional}
R.~L. Bagley, P.~J. Torvik, Fractional calculus in the transient analysis of
  viscoelastically damped structures, AIAA journal 23~(6) (1985) 918--925.

\bibitem{ichise1971analog}
M.~Ichise, Y.~Nagayanagi, T.~Kojima, An analog simulation of non-integer order
  transfer functions for analysis of electrode processes, Journal of
  Electroanalytical Chemistry and Interfacial Electrochemistry 33~(2) (1971)
  253--265.

\bibitem{hajipour2018efficient}
M.~Hajipour, A.~Jajarmi, D.~Baleanu, An efficient nonstandard finite difference
  scheme for a class of fractional chaotic systems, Journal of Computational
  and Nonlinear Dynamics 13~(2) (2018).

\bibitem{baleanu2017motion}
D.~Baleanu, A.~Jajarmi, J.~H. Asad, T.~Blaszczyk, The motion of a bead sliding
  on a wire in fractional sense (2017).

\bibitem{jajarmi2018new}
A.~Jajarmi, M.~Hajipour, E.~Mohammadzadeh, D.~Baleanu, A new approach for the
  nonlinear fractional optimal control problems with external persistent
  disturbances, Journal of the Franklin Institute 355~(9) (2018) 3938--3967.

\bibitem{baleanu2018nonlinear}
D.~Baleanu, A.~Jajarmi, M.~Hajipour, On the nonlinear dynamical systems within
  the generalized fractional derivatives with mittag--leffler kernel, Nonlinear
  dynamics 94~(1) (2018) 397--414.

\bibitem{sierociuk2013modelling}
D.~Sierociuk, A.~Dzieli{\'n}ski, G.~Sarwas, I.~Petras, I.~Podlubny,
  T.~Skovranek, Modelling heat transfer in heterogeneous media using fractional
  calculus, Philosophical Transactions of the Royal Society A: Mathematical,
  Physical and Engineering Sciences 371~(1990) (2013) 20120146.

\bibitem{chow2005fractional}
T.~Chow, Fractional dynamics of interfaces between soft-nanoparticles and rough
  substrates, Physics Letters A 342~(1-2) (2005) 148--155.

\bibitem{yang2017new}
X.-J. Yang, et~al., New general fractional-order rheological models with
  kernels of mittag-leffler functions, Rom. Rep. Phys 69~(4) (2017) 118.

\bibitem{gao2016fractional}
F.~Gao, X.-J. Yang, Fractional maxwell fluid with fractional derivative without
  singular kernel, Thermal Science 20~(suppl. 3) (2016) 871--877.

\bibitem{odibat2010study}
Z.~M. Odibat, A study on the convergence of variational iteration method,
  Mathematical and Computer Modelling 51~(9-10) (2010) 1181--1192.

\bibitem{momani2007generalized}
S.~Momani, Z.~Odibat, V.~S. Erturk, Generalized differential transform method
  for solving a space-and time-fractional diffusion-wave equation, Physics
  Letters A 370~(5-6) (2007) 379--387.

\bibitem{odibat2008generalized}
Z.~Odibat, S.~Momani, V.~S. Erturk, Generalized differential transform method:
  application to differential equations of fractional order, Applied
  Mathematics and Computation 197~(2) (2008) 467--477.

\bibitem{garg2011solution}
M.~Garg, A.~Sharma, Solution of space-time fractional telegraph equation by
  adomian decomposition method, Journal of Inequalities and Special Functions
  2~(1) (2011) 1--7.

\bibitem{ray2005approximate}
S.~S. Ray, R.~Bera, An approximate solution of a nonlinear fractional
  differential equation by adomian decomposition method, Applied Mathematics
  and Computation 167~(1) (2005) 561--571.

\bibitem{chen2012error}
Y.~Chen, M.~Yi, C.~Yu, Error analysis for numerical solution of fractional
  differential equation by haar wavelets method, Journal of Computational
  Science 3~(5) (2012) 367--373.

\bibitem{babolian2007numerical}
E.~Babolian, F.~Fattahzadeh, Numerical solution of differential equations by
  using chebyshev wavelet operational matrix of integration, Applied
  Mathematics and computation 188~(1) (2007) 417--426.

\bibitem{saadatmandi2011tau}
A.~Saadatmandi, M.~Dehghan, A tau approach for solution of the space fractional
  diffusion equation, Computers \& Mathematics with Applications 62~(3) (2011)
  1135--1142.

\bibitem{srivastava2021efficient}
N.~Srivastava, A.~Singh, Y.~Kumar, V.~K. Singh, Efficient numerical algorithms
  for riesz-space fractional partial differential equations based on finite
  difference/operational matrix, Applied Numerical Mathematics 161 (2021)
  244--274.

\bibitem{dehghan2019error}
M.~Dehghan, M.~Abbaszadeh, Error estimate of finite element/finite difference
  technique for solution of two-dimensional weakly singular integro-partial
  differential equation with space and time fractional derivatives, Journal of
  Computational and Applied Mathematics 356 (2019) 314--328.

\bibitem{abbaszadeh2019meshless}
M.~Abbaszadeh, M.~Dehghan, Meshless upwind local radial basis function-finite
  difference technique to simulate the time-fractional distributed-order
  advection--diffusion equation, Engineering with computers (2019) 1--17.

\bibitem{abbaszadeh2019alternating}
M.~Abbaszadeh, M.~Dehghan, Y.~Zhou, Alternating direction implicit-spectral
  element method (adi-sem) for solving multi-dimensional generalized modified
  anomalous sub-diffusion equation, Computers \& Mathematics with Applications
  78~(5) (2019) 1772--1792.

\bibitem{dehghan2010solving}
M.~Dehghan, J.~Manafian, A.~Saadatmandi, Solving nonlinear fractional partial
  differential equations using the homotopy analysis method, Numerical Methods
  for Partial Differential Equations: An International Journal 26~(2) (2010)
  448--479.

\bibitem{sun2016some}
H.~Sun, Z.~Z. Sun, G.~H. Gao, Some high order difference schemes for the space
  and time fractional bloch-torrey equations, Applied Mathematics and
  Computation 281 (2016) 356--380.

\bibitem{jiao2012distributed}
Z.~Jiao, Y.-Q. Chen, I.~Podlubny, Distributed-order dynamic systems: Stability,
  Simulation, Applications and Perspectives, London (2012).

\bibitem{Caputo2001}
M.~Caputo, Distributed order differential equations modelling dielectric
  induction and diffusion, Fractional Calculus and Applied Analysis 4~(4)
  (2001) 421--442.

\bibitem{sokolov2004distributed}
I.~Sokolov, A.~Chechkin, J.~Klafter, Distributed-order fractional kinetics,
  arXiv preprint cond-mat/0401146 (2004).

\bibitem{umarov2006random}
S.~Umarov, S.~Steinberg, et~al., Random walk models associated with distributed
  fractional order differential equations, in: High dimensional probability,
  Institute of Mathematical Statistics, 2006, pp. 117--127.

\bibitem{morgado2019black}
L.~Morgado, M.~Rebelo, Black-scholes equation with distributed order in time,
  in: Progress in Industrial Mathematics at ECMI 2018, Springer, 2019, pp.
  313--319.

\bibitem{kumar2020wavelet}
Y.~Kumar, S.~Singh, N.~Srivastava, A.~Singh, V.~K. Singh, Wavelet approximation
  scheme for distributed order fractional differential equations, Computers \&
  Mathematics with Applications 80~(8) (2020) 1985--2017.

\bibitem{abbaszadeh2019error}
M.~Abbaszadeh, Error estimate of second-order finite difference scheme for
  solving the riesz space distributed-order diffusion equation, Applied
  Mathematics Letters 88 (2019) 179--185.

\bibitem{gao2017temporal}
G.~H. Gao, A.~A. Alikhanov, Z.~Z. Sun, The temporal second order difference
  schemes based on the interpolation approximation for solving the time
  multi-term and distributed-order fractional sub-diffusion equations, Journal
  of Scientific Computing 73~(1) (2017) 93--121.

\bibitem{abbaszadeh2020crank}
M.~Abbaszadeh, M.~Dehghan, Y.~Zhou, Crank--nicolson/galerkin spectral method
  for solving two-dimensional time-space distributed-order weakly singular
  integro-partial differential equation, Journal of Computational and Applied
  Mathematics 374 (2020) 112739.

\bibitem{christensen2012mechanics}
R.~M. Christensen, Mechanics of composite materials, Courier Corporation, 2012.

\bibitem{miller1978integrodifferential}
R.~Miller, An integrodifferential equation for rigid heat conductors with
  memory, Journal of Mathematical Analysis and Applications 66~(2) (1978)
  313--332.

\bibitem{renardy1989mathematical}
M.~Renardy, Mathematical analysis of viscoelastic flows, Annual review of fluid
  mechanics 21 (1989) 21--36.

\bibitem{gorenflo2013fundamental}
R.~Gorenflo, Y.~Luchko, M.~Stojanovi{\'c}, Fundamental solution of a
  distributed order time-fractional diffusion-wave equation as probability
  density, Fractional Calculus and Applied Analysis 16~(2) (2013) 297--316.

\bibitem{li2017analyticity}
Z.~Li, Y.~Luchko, M.~Yamamoto, Analyticity of solutions to a distributed order
  time-fractional diffusion equation and its application to an inverse problem,
  Computers \& Mathematics with Applications 73~(6) (2017) 1041--1052.

\bibitem{morgado2017numerical}
M.~L. Morgado, M.~Rebelo, L.~L. Ferras, N.~J. Ford, Numerical solution for
  diffusion equations with distributed order in time using a chebyshev
  collocation method, Applied Numerical Mathematics 114 (2017) 108--123.

\bibitem{saadatmandi2010new}
A.~Saadatmandi, M.~Dehghan, A new operational matrix for solving
  fractional-order differential equations, Computers \& mathematics with
  applications 59~(3) (2010) 1326--1336.

\bibitem{li2010haar}
Y.~Li, W.~Zhao, Haar wavelet operational matrix of fractional order integration
  and its applications in solving the fractional order differential equations,
  Applied Mathematics and Computation 216~(8) (2010) 2276--2285.

\bibitem{keshavarz2014bernoulli}
E.~Keshavarz, Y.~Ordokhani, M.~Razzaghi, Bernoulli wavelet operational matrix
  of fractional order integration and its applications in solving the
  fractional order differential equations, Applied Mathematical Modelling
  38~(24) (2014) 6038--6051.

\bibitem{bhrawy2015review}
A.~H. Bhrawy, T.~M. Taha, J.~A.~T. Machado, A review of operational matrices
  and spectral techniques for fractional calculus, Nonlinear Dynamics 81~(3)
  (2015) 1023--1052.

\bibitem{pourbabaee2019novel}
M.~Pourbabaee, A.~Saadatmandi, A novel legendre operational matrix for
  distributed order fractional differential equations, Applied Mathematics and
  Computation 361 (2019) 215--231.

\bibitem{singh2017numerical}
S.~Singh, V.~K. Patel, V.~K. Singh, E.~Tohidi, Numerical solution of nonlinear
  weakly singular partial integro-differential equation via operational
  matrices, Applied Mathematics and Computation 298 (2017) 310--321.

\bibitem{singh2018application}
S.~Singh, V.~K. Patel, V.~K. Singh, Application of wavelet collocation method
  for hyperbolic partial differential equations via matrices, Applied
  Mathematics and Computation 320 (2018) 407--424.

\bibitem{singh2018convergence}
S.~Singh, V.~K. Patel, V.~K. Singh, Convergence rate of collocation method
  based on wavelet for nonlinear weakly singular partial integro-differential
  equation arising from viscoelasticity, Numerical Methods for Partial
  Differential Equations 34~(5) (2018) 1781--1798.

\bibitem{chen1997haar}
C.~Chen, C.~Hsiao, Haar wavelet method for solving lumped and
  distributed-parameter systems, IEE Proceedings-Control Theory and
  Applications 144~(1) (1997) 87--94.

\bibitem{sahu2015legendre}
P.~K. Sahu, S.~S. Ray, Legendre wavelets operational method for the numerical
  solutions of nonlinear volterra integro-differential equations system,
  Applied mathematics and computation 256 (2015) 715--723.

\bibitem{heydari2014legendre}
M.~H. Heydari, M.~R. Hooshmandasl, F.~Mohammadi, Legendre wavelets method for
  solving fractional partial differential equations with dirichlet boundary
  conditions, Applied Mathematics and Computation 234 (2014) 267--276.

\bibitem{meng2015legendre}
Z.~Meng, L.~Wang, H.~Li, W.~Zhang, Legendre wavelets method for solving
  fractional integro-differential equations, International Journal of Computer
  Mathematics 92~(6) (2015) 1275--1291.

\bibitem{mehra2018wavelets}
M.~Mehra, Mehra, Ahmad, Wavelets Theory and Its Applications, Springer, 2018.

\bibitem{kumar2021computational}
Y.~Kumar, V.~K. Singh, Computational approach based on wavelets for financial
  mathematical model governed by distributed order fractional differential
  equation, Mathematics and Computers in Simulation (2021).

\bibitem{behera2018adaptive}
R.~Behera, M.~Mehra, An adaptive wavelet collocation method for solution of the
  convection-dominated problem on a sphere, International Journal of
  Computational Methods 15~(08) (2018) 1850080.

\bibitem{alikhanov2015numerical}
A.~A. Alikhanov, Numerical methods of solutions of boundary value problems for
  the multi-term variable-distributed order diffusion equation, Applied
  Mathematics and Computation 268 (2015) 12--22.

\bibitem{hildebrand1987introduction}
F.~B. Hildebrand, Introduction to numerical analysis, Courier Corporation,
  1987.

\bibitem{mohammadi2011new}
F.~Mohammadi, M.~Hosseini, A new legendre wavelet operational matrix of
  derivative and its applications in solving the singular ordinary differential
  equations, Journal of the Franklin Institute 348~(8) (2011) 1787--1796.

\bibitem{zaky2020multi}
M.~A. Zaky, J.~T. Machado, Multi-dimensional spectral tau methods for
  distributed-order fractional diffusion equations, Computers \& Mathematics
  with Applications 79~(2) (2020) 476--488.

\end{thebibliography}
\end{document}